\documentclass[11pt,leqno]{amsart}
\usepackage{epsfig}
\usepackage{amssymb}
\usepackage{amscd}
\usepackage{ifthen}
\usepackage{enumitem}
\setenumerate[1]{label=(\alph*)}
\setenumerate[2]{label=(\roman*)}

\usepackage{amsthm}
\usepackage{newlfont}
\usepackage{amsmath}
\usepackage{amsfonts}
\usepackage[mathscr]{euscript}
\usepackage[ngerman,french,english]{babel}
\usepackage{mathrsfs}

\usepackage{mathtools}

\usepackage{fancybox}
\usepackage{nicefrac}
\usepackage{pifont}
\usepackage[all]{xy}
\usepackage{stmaryrd}
\usepackage{srcltx}
\usepackage{ifthen}
\usepackage{ulem}
\usepackage{accents}
\usepackage{leftidx}

\usepackage{xr-hyper}

\usepackage{hyperref}
\hypersetup{colorlinks=true, anchorcolor=red, citecolor= blue,
filecolor=magenta, linkcolor=red, menucolor=red, urlcolor=blue}

\newtheorem{theorem}{Theorem}[section]
\newtheorem{lemma}[theorem]{Lemma}
\newtheorem{definition}[theorem]{Definition}
\newtheorem{proposition}[theorem]{Proposition}

\newtheorem{corollary}[theorem]{Corollary}
\newtheorem{remark}[theorem]{Remark}

\newtheorem{example}[theorem]{Example}

\numberwithin{equation}{section}

\newcommand{\bR}{\mathbf{R}}
\newcommand{\bL}{\mathbf{L}}

\def\pre-tr{\operatorname{pre-tr}}

\newcommand{\Tot}{\operatorname{Tot}}

\newcommand{\Perf}{\operatorname{Perf}}

\newcommand{\Perfotimes}{\odot}
\newcommand{\Ext}{\operatorname{Ext}}

\newcommand{\res}{\operatorname{res}}

\newcommand{\Spec}{\operatorname{Spec}}

\newcommand{\Vect}{{\operatorname{Vect}}}
\newcommand{\fd}{{\operatorname{fd}}}

\newcommand{\id}{\operatorname{id}}

\newcommand{\dg}{{\operatorname{dg}}}

\DeclareMathAlphabet{\mathpzc}{OT1}{pzc}{m}{it}

\newcommand{\groundring}{\mathsf{k}}

\newcommand{\matfak}[4]{{\xymatrix@C3ex{{{#1}} \ar@<0.3ex>[r]^-{{#2}} & {{#3}} \ar@<0.3ex>[l]^-{{#4}}}}}

\newcommand{\op}{\operatorname}

\newcommand{\ra}{\rightarrow}

\newcommand{\la}{\leftarrow}

\newcommand{\sra}{\twoheadrightarrow}
\newcommand{\hra}{\hookrightarrow}

\newcommand{\xra}[1]{\xrightarrow{#1}}
\newcommand{\xla}[1]{\xleftarrow{#1}}

\newcommand{\sira}{\xra{\sim}}
\newcommand{\xsira}[1]{\xrightarrow[\sim]{#1}}

\newcommand{\sila}{\overset{\sim}{\leftarrow}}
\newcommand{\xsila}[1]{\xleftarrow[\sim]{#1}}

\newcommand{\mar}{\ar@{|->}}
\newcommand{\sar}{\ar@{->>}}
\newcommand{\iar}{\ar@{^{(}->}}
\newcommand{\gar}{\ar@{=}}
\newcommand{\gleichar}{\ar@{}|{=}}
\newcommand{\congar}{\ar@{}|{\cong}}

\newcommand{\Bl}[1]{{\mathbb{#1}}}
\newcommand{\DZ}{\Bl{Z}}
\newcommand{\DN}{\Bl{N}}

\newcommand{\DA}{{\Bl{A}}}
\newcommand{\DP}{{\Bl{P}}}

\newcommand{\DL}{\Bl{L}}

\newcommand{\Hom}{\op{Hom}}

\newcommand{\calMod}{\mathcal{M}od}
\newcommand{\Yoneda}[1]{{\widehat{#1}}}

\newcommand{\Sh}{{\op{Sh}}}

\newcommand{\Mod}{{\op{Mod}}}
\newcommand{\Qcoh}{{\op{Qcoh}}}
\newcommand{\Coh}{{\op{Coh}}}
\newcommand{\mfPerf}{{{\mathfrak P}{\mathfrak e}{\mathfrak r}{\mathfrak f}}}
\newcommand{\Sg}{{\op{Sg}}}

\newcommand{\MF}{\op{MF}}
\newcommand{\InjSh}{{\op{InjSh}}}
\newcommand{\InjQcoh}{{\op{InjQcoh}}}

\newcommand{\Acycl}{\op{Acycl}}

\newcommand{\Cechmor}{{\mathrm{\check{C}mor}}}
\newcommand{\Cechobj}{{\mathrm{\check{C}ob}}}

\newcommand{\Cechobjbox}{{\mathrm{\check{C}ob}\boxtimes}}

\newcommand{\ord}{{\op{ord}}}

\newcommand{\Blow}{{\op{Bl}}}

\newcommand{\DSh}{\op{DSh}}
\newcommand{\DQcoh}{\op{DQcoh}}
\newcommand{\DCoh}{\op{DCoh}}
\newcommand{\bfMF}{\mathbf{MF}}

\newcommand{\AcyclMF}{\op{AcyclMF}}

\newcommand{\co}{{\op{co}}}

\newcommand{\sheafHom}{{\mkern3mu\mathcal{H}{om}\mkern3mu}}

\newcommand{\End}{\op{End}}

\newcommand{\pr}{\op{pr}}

\newcommand{\ol}[1]{{\overline{#1}}}
\newcommand{\ul}[1]{{\underline{#1}}}

\newcommand{\leftadjointtores}{\op{prod}}
\newcommand{\pro}{\leftadjointtores}

\newcommand{\coker}{\op{cok}}
\newcommand{\cokern}{\op{cok}}

\newcommand{\cek}{\vee}

\newcommand{\inv}{^{-1}}
\newcommand{\can}{\op{can}}

\newcommand{\per}{\op{per}}

\newcommand{\sat}{\op{sat}}
\newcommand{\prsm}{\op{prsm}}
\newcommand{\prsmalg}{\op{prsmalg}}

\newcommand{\dgcat}{\op{dgcat}}
\newcommand{\Heq}{\mathsf{Heq}}
\newcommand{\Hmo}{\mathsf{Hmo}}

\newcommand{\tzmat}[4]{{\left[\begin{smallmatrix} {#1} & {#2} \\ {#3} & {#4} \end{smallmatrix}\right]}}

\newcommand{\roundtzmat}[4]{{\left(\begin{smallmatrix} {#1} & {#2} \\ {#3} & {#4} \end{smallmatrix}\right)}}

\newcommand{\tildew}[1]{\widetilde{#1}}
\newcommand{\hatw}[1]{\widehat{#1}}

\newcommand{\Var}{{\op{Var}}}
\newcommand{\sm}{{\op{sm}}}
\newcommand{\bl}{{\op{bl}}}

\newcommand{\comp}{\circ}

\newcommand{\opp}{{\op{op}}}

\newcommand{\define}[1]{{\textbf{#1}}}

\newboolean{showkeepdontpublishbool}
\newcommand{\keepdontpublish}[1]
{\ifthenelse{\boolean{showkeepdontpublishbool}}
  {
    \rule{0cm}{0.05cm} \\
    \rule{15cm}{0.05cm} \\
    {#1} \rule{0cm}{0.05cm} \\ 
    \rule{15cm}{0.05cm} \\
    \rule{1mm}{0mm}
  }
} 

\setboolean{showkeepdontpublishbool}{true}          %Kommentare JA
\setboolean{showkeepdontpublishbool}{false}         %Kommentare NEIN

\newboolean{comment}
\newcommand{\comline}[1]{\ifthenelse{\boolean{comment}}{{\bf
      \noindent\shortstack[r]{\rule{12cm}{0.02cm}\\#1\\\rule{12cm}{0.02cm}
        \pagestyle{myheadings}\markboth{Links}{Rechts} 
        }}\\}}       

\newcommand{\comtxt}[1]{\ifthenelse{\boolean{comment}}{{\bf #1}}} 
\newcommand{\comtxtn}[1]{\ifthenelse{\boolean{comment}}{{\bf #1} \\ }} 
\newcommand{\ncomtxt}[1]{\ifthenelse{\boolean{comment}}{\\ {\bf #1}}} 
\newcommand{\ncomtxtn}[1]{\ifthenelse{\boolean{comment}}{\\ {\bf #1} \\ }} 

\newcommand{\commar}[1]{\ifthenelse{\boolean{comment}}{\marginpar{{#1}}}{}}

\setboolean{comment}{true}          %Kommentare JA
\newcommand{\kommentar}[1]{}

\numberwithin{equation}{section}

\newcommand{\Cone}{\op{Cone}}

\renewcommand{\epsilon}{{\varepsilon}}

\newcommand{\oktaeder}[6]
{
  \xymatrix@dr{
    & {[1]{#1}} \ar[r] &  {[1]{#2}} \ar[r] & {[1]{#3}}\\
    {#3} \ar@(ur,ul)[ru] \ar[r] &
    {#5}
    \ar[r] \ar[u] & 
    {#6} 
    \ar@(dr,dl)[ru] \ar[u]\\ 
    {#2} \ar[u] \ar[r] & {#4} \ar@(dr,dl)[ru] \ar[u]\\
    {#1} \ar[u] \ar@(dr,dl)[ru]
  }
}

\newcommand{\oktaedergepunktetdown}[6]
{
  \xymatrix@dr{
    & {[1]{#1}} \ar[r] &  {[1]{#2}} \ar[r] & {[1]{#3}}\\
    {#3} \ar@(ur,ul)[ru] \ar@{..>}[r] &
    {#5}
    \ar@{..>}[r] \ar[u] & 
    {#6} 
    \ar@(dr,dl)[ru] \ar[u]\\ 
    {#2} \ar[u] \ar[r] & {#4} \ar@(dr,dl)[ru] \ar[u]\\
    {#1} \ar[u] \ar@(dr,dl)[ru]
  }
}

\newcommand{\oktaedergepunktetdownmitopt}[7]
{
  \xymatrix@dr#7{
    & {[1]{#1}} \ar[r] &  {[1]{#2}} \ar[r] & {[1]{#3}}\\
    {#3} \ar@(ur,ul)[ru] \ar@{..>}[r] &
    {#5}
    \ar@{..>}[r] \ar[u] & 
    {#6} 
    \ar@(dr,dl)[ru] \ar[u]\\ 
    {#2} \ar[u] \ar[r] & {#4} \ar@(dr,dl)[ru] \ar[u]\\
    {#1} \ar[u] \ar@(dr,dl)[ru]
  }
}

\newcommand{\oktaedergepunktetup}[6]
{
  \xymatrix@dr{
    & {[1]{#1}} \ar[r] &  {[1]{#2}} \ar[r] & {[1]{#3}}\\
    {#3} \ar@(ur,ul)[ru] \ar[r] &
    {#5}
    \ar[r] \ar@{..>}[u] & 
    {#6} 
    \ar@(dr,dl)[ru] \ar[u]\\ 
    {#2} \ar[u] \ar[r] & {#4} \ar@(dr,dl)[ru] \ar@{..>}[u]\\
    {#1} \ar[u] \ar@(dr,dl)[ru]
  }
}

\newcommand{\oktaederalles}[9]
{
  \xymatrix@dr{
    & {#7} \ar[r] &  {#8} \ar[r] & {#9}\\
    {#3} \ar@(ur,ul)[ru] \ar[r] &
    {#5}
    \ar[r] \ar[u] & 
    {#6} 
    \ar@(dr,dl)[ru] \ar[u]\\ 
    {#2} \ar[u] \ar[r] & {#4} \ar@(dr,dl)[ru] \ar[u]\\
    {#1} \ar[u] \ar@(dr,dl)[ru]
  }
}

\newcommand{\oktaedertemplatedotted}
{
  \xymatrix@dr{
    && {G} \ar[r]^-{g} 
    & {H} \ar[r]^-{h} 
    & {I} 
    \\ 
    {T4:} 
    & {C} \ar@(ur,ul)[ru]^-{c_1}
    \ar@{..>}[r]^-{c_2} 
    & {E} \ar@{..>}[r]^-{e_1} \ar[u]^-{e_2} 
    & {F}  \ar@(dr,dl)[ru]^-{f_1} \ar[u]_-{f_2}
    \\ 
    {T3:} 
    & {B} \ar[u]^-{b_1} \ar[r]^-{b_2} 
    & {D} \ar@(dr,dl)[ru]^-{d_1} \ar[u]_-{d_2} 
    \\  
    {T2:} 
    & {A} \ar[u]^-{a_1} \ar@(dr,dl)[ru]^-{a_2}
    \ar@{}[ru]|{\triangle}
    \\
    & {T1:}
  }
}

\newcommand{\oktaedertemplate}
{
  \xymatrix@dr{
    && {G} \ar[r]^-{g} 
    & {H} \ar[r]^-{h} 
    & {I} 
    \\ 
    {T4:} 
    & {C} \ar@(ur,ul)[ru]^-{c_1}
    \ar[r]^-{c_2} 
    & {E} \ar[r]^-{e_1} \ar[u]^-{e_2} 
    & {F}  \ar@(dr,dl)[ru]^-{f_1} \ar[u]_-{f_2}
    \\ 
    {T3:} 
    & {B} \ar[u]^-{b_1} \ar[r]^-{b_2} 
    & {D} \ar@(dr,dl)[ru]^-{d_1} \ar[u]_-{d_2} 
    \\  
    {T2:} &{A} \ar[u]^-{a_1} \ar@(dr,dl)[ru]^-{a_2}
    \ar@{}[ru]|{\triangle}
    \\
    & {T1:}
  }
}

\newcommand{\symeighttransone}[9]
{
  { } \ar@{}[r]^{#9} & { }\\
  {#1} \ar@(u,d)[ruu] & {#2} \ar@(u,d)[luu] & {#3} \ar[uu] & {#4} \ar[uu] & 
  {#5} \ar[uu] & {#6} \ar[uu] & {#7} \ar[uu] & {#8} \ar[uu] \\ 
}

\newcommand{\symeighttranstwo}[9]
{
  & { } \ar@{}[r]^{#9} & { }\\
  {#1} \ar[uu] & {#2} \ar@(u,d)[ruu] & {#3} \ar@(u,d)[luu] & {#4} \ar[uu] & 
  {#5} \ar[uu] & {#6} \ar[uu] & {#7} \ar[uu] & {#8} \ar[uu] \\ 
}

\newcommand{\symeighttransthree}[9]
{
  && { } \ar@{}[r]^{#9} & { }\\
  {#1} \ar[uu] &  {#2} \ar[uu] & {#3} \ar@(u,d)[ruu] & {#4} \ar@(u,d)[luu] &
  {#5} \ar[uu] & {#6} \ar[uu] & {#7} \ar[uu] & {#8} \ar[uu] \\ 
}

\newcommand{\symeighttransfour}[9]
{
  &&& { } \ar@{}[r]^{#9} & { }\\
  {#1} \ar[uu] &  {#2} \ar[uu] &  {#3} \ar[uu] & {#4} \ar@(u,d)[ruu] & 
  {#5} \ar@(u,d)[luu] & {#6} \ar[uu] & {#7} \ar[uu] & {#8} \ar[uu] \\
}

\newcommand{\symeighttransfive}[9]
{
  &&&& { } \ar@{}[r]^{#9} & { }\\
  {#1} \ar[uu] &  {#2} \ar[uu] &  {#3} \ar[uu] & {#4} \ar[uu] & 
  {#5} \ar@(u,d)[ruu] & {#6} \ar@(u,d)[luu] & {#7} \ar[uu] & {#8}
  \ar[uu] \\
}

\newcommand{\symeighttranssix}[9]
{
  &&&&& { } \ar@{}[r]^{#9} & { }\\
  {#1} \ar[uu] & {#2} \ar[uu] &  {#3} \ar[uu] & {#4} \ar[uu] & 
  {#5} \ar[uu] & 
  {#6} \ar@(u,d)[ruu] & {#7} \ar@(u,d)[luu] & {#8} \ar[uu] \\
}

\newcommand{\symeighttransseven}[9]
{
  &&&&&& { } \ar@{}[r]^{#9} & { }\\
  {#1} \ar[uu] & {#2} \ar[uu] &  {#3} \ar[uu] & {#4} \ar[uu] & 
  {#5} \ar[uu] & 
  {#6} \ar[uu] & 
  {#7} \ar@(u,d)[ruu] & {#8} \ar@(u,d)[luu] \\
}

\newcommand{\symdowneighttransone}[9]
{
  {#1} \ar@(d,u)[rdd] & {#2} \ar@(d,u)[ldd] & {#3} \ar[dd] & {#4} \ar[dd] & 
  {#5} \ar[dd] & {#6} \ar[dd] & {#7} \ar[dd] & {#8} \ar[dd] \\ 
  { } \ar@{}[r]^{#9} & { }\\
}

\newcommand{\symdowneighttranstwo}[9]
{
  {#1} \ar[dd] & {#2} \ar@(d,u)[rdd] & {#3} \ar@(d,u)[ldd] & {#4} \ar[dd] & 
  {#5} \ar[dd] & {#6} \ar[dd] & {#7} \ar[dd] & {#8} \ar[dd] \\ 
  & { } \ar@{}[r]^{#9} & { }\\
}

\newcommand{\symdowneighttransthree}[9]
{
  {#1} \ar[dd] &  {#2} \ar[dd] & {#3} \ar@(d,u)[rdd] & {#4}
  \ar@(d,u)[ldd] &
  {#5} \ar[dd] & {#6} \ar[dd] & {#7} \ar[dd] & {#8} \ar[dd] \\ 
  && { } \ar@{}[r]^{#9} & { }\\
}

\newcommand{\symdowneighttransfour}[9]
{
  {#1} \ar[dd] &  {#2} \ar[dd] &  {#3} \ar[dd] & {#4} \ar@(d,u)[rdd] & 
  {#5} \ar@(d,u)[ldd] & {#6} \ar[dd] & {#7} \ar[dd] & {#8} \ar[dd] \\
  &&& { } \ar@{}[r]^{#9} & { }\\
}

\newcommand{\symdowneighttransfive}[9]
{
  {#1} \ar[dd] &  {#2} \ar[dd] &  {#3} \ar[dd] & {#4} \ar[dd] & 
  {#5} \ar@(d,u)[rdd] & {#6} \ar@(d,u)[ldd] & {#7} \ar[dd] & {#8}
  \ar[dd] \\
  &&&& { } \ar@{}[r]^{#9} & { }\\
}

\newcommand{\symdowneighttranssix}[9]
{
  {#1} \ar[dd] & {#2} \ar[dd] &  {#3} \ar[dd] & {#4} \ar[dd] & 
  {#5} \ar[dd] & 
  {#6} \ar@(d,u)[rdd] & {#7} \ar@(d,u)[ldd] & {#8} \ar[dd] \\
  &&&&& { } \ar@{}[r]^{#9} & { }\\
}

\newcommand{\symdowneighttransseven}[9]
{
  {#1} \ar[dd] & {#2} \ar[dd] &  {#3} \ar[dd] & {#4} \ar[dd] & 
  {#5} \ar[dd] & 
  {#6} \ar[dd] & 
  {#7} \ar@(d,u)[rdd] & {#8} \ar@(d,u)[ldd] \\
  &&&&&& { } \ar@{}[r]^{#9} & { }\\
}

\newcommand{\quiverAzwei}[7]
{
  \xymatrix@R0pc@C2pc{
    {{#2}} \ar[r]^{#4} \ar@{}[r]^>(0){#1}_>(0){#3} &
    {{#6}}  \ar@{}[l]_>(0){#5}^>(0){#7}
  }
}

\newcommand{\quiverAeins}[3]
{
  \xymatrix@R0pc@C0pt{
    {{#2}} \ar@{}[r]^>(0){#1}_>(0){#3} & {}
  }
}

\newcommand{\bruhatAzwei}[9]
{
  \xymatrix@C17.3pt@R10pt{
    & 
    {{#1}} \ar[rd]^-{{#9}} \ar[ld]_-{{#7}} \ar[dddd]|(0.25){{#8}} \\
    {{#2}} \ar[dd]_-{{#8}} \ar[rrdd]|(0.25){{{#9}}} &&
    {{#3}} \ar[dd]^-{{#8}} \ar[lldd]|(0.25){{#7}} \\ 
    \\
    {{#4}} \ar[rd]_-{{#9}} &&
    {{#5}} \ar[ld]^-{{#7}} \\
    &
    {{#6}}
  }
}

\newcommand{\weylAzwei}[6]
{
  \xymatrix@C17.3pt@R10pt{
    & 
    {{#1}} \ar[rd] \ar[ld] \ar[dddd] \\
    {{#2}} \ar[dd] \ar[rrdd] &&
    {{#3}} \ar[dd] \ar[lldd] \\ 
    \\
    {{#4}} \ar[rd] &&
    {{#5}} \ar[ld] \\
    &
    {{#6}}
  }
}

\newcommand{\pretr}{{\op{pre-tr}}}

\newcommand{\tria}{\op{tria}}
\newcommand{\thick}{\op{thick}}

\newcommand{\Sing}{{\op{Sing}}}
\newcommand{\sing}{{\op{sing}}}

\newcommand{\Crit}{{\op{Crit}}}

\author{Valery A.~Lunts \and Olaf M.~Schn{\"u}rer}

\address{
  Department of Mathematics\\
  Indiana University\\
  Rawles Hall\\
  831 East 3rd Street\\
  Bloomington, IN 47405\\
  USA
}

\email{vlunts@indiana.edu} 

\address{
  Mathematisches Institut\\ 
  Universit{\"a}t Bonn\\
  Endenicher Allee 60\\
  53115 Bonn\\
  Germany
}

\email{olaf.schnuerer@math.uni-bonn.de}

\title{Matrix factorizations and motivic measures}

\thanks{}

\begin{document}

\begin{abstract}
  This article is the continuation of
  \cite{valery-olaf-matfak-semi-orth-decomp}.  We use categories
  of matrix factorizations to define a morphism of rings (= a
  Landau-Ginzburg motivic measure) from the (motivic) Grothendieck ring of
  varieties over $\DA^1$ to the Grothendieck ring of saturated dg
  categories (with relations coming from semi-orthogonal
  decompositions into admissible subcategories).  Our Landau-Ginzburg motivic
  measure is the analog for matrix factorizations of the motivic
  measure in \cite{bondal-larsen-lunts-grothendieck-ring} whose
  definition involved bounded derived categories of coherent
  sheaves. On the way we
  prove smoothness and a Thom-Sebastiani theorem for enhancements
  of categories of matrix factorizations.
\end{abstract}

\maketitle

\tableofcontents

\section{Introduction}
\label{sec:introduction}

This article is the partner of
\cite{valery-olaf-matfak-semi-orth-decomp}.  The mutual goal of
these two articles is the construction of
% to use categories of matrix factorizations in order to
% construct  
an interesting Landau-Ginzburg motivic
measure: a ring morphism from the (motivic) Grothendieck ring of
varieties over $\DA^1$ to another ring. The terminology
Landau-Ginzburg comes from physics where a morphism $W \colon  X \ra
\DA^1$ is considered as a superpotential on a variety $X.$

% The idea is to send $W \colon  X \ra \DA^1,$ where $X$ is
% a smooth variety and $W$ is a proper morphism, to 
% the singularity category $\bfMF(W)$ of $W$ defined below using
% matrix factorization categories.
% the class of a
% dg (differential $\DZ/2\DZ$-graded) version of
% the in a defined below. 

Let $k$ be an algebraically closed field of characteristic zero.
Let $X$ be a smooth variety (over $k$) and $W \colon  X \ra \DA^1=\DA^1_k$ a morphism (also
viewed as an element of $\Gamma(X,\mathcal{O}_X)$). We denote the
category of matrix factorizations of $W$ by $\bfMF(X,W)$
(see \cite{valery-olaf-matfak-semi-orth-decomp}).
Taking the product over all the categories $\bfMF(X,W-a),$ for $a
\in k,$ defines the
singularity category 
$\bfMF(W)$ of $W,$
\begin{equation*}
  % \label{eq:D-Sg-W}
  \bfMF(W):=\prod_{a \in k} \textbf{MF}(X, W-a).
\end{equation*}
Only finitely many factors of this product are non-zero, and
$\bfMF(W)$ vanishes if and only if $W$ is a smooth morphism (see
Lemma~\ref{l:D-Sg-W-vanishes}).
Let $\bfMF(W)^\dg$ be a suitable enhancement
(in the differential $\DZ_2$-graded setting, where $\DZ_2=\DZ/2\DZ$)
of $\bfMF(W),$ and let $\bfMF(W)^{\dg,\natural}$ be the
corresponding enhancement 
of 
the Karoubi envelope of $\bfMF(W).$
%Let $\bfMF(W)^{\dg,\natural}$ be a suitable enhancement
%(in the differential $\DZ_2$-graded setting, where $\DZ_2=\DZ/2\DZ$%)
%of 
%the Karoubi envelope of $\bfMF(W).$

The (motivic) Grothendieck group $K_0(\Var_{\DA^1})$ of varieties
over $\DA^1$
% $=\DA^1_k$ 
is defined as the free abelian group on
isomorphism classes $[X]_{\DA^1}=[X,W]$ of varieties $W \colon  X \ra
\DA^1$ over 
$\DA^1$ subject to the relations $[X]_{\DA^1}=[X-Y]_{\DA^1} +
[Y]_{\DA^1}$ whenever $Y \subset X$ is a closed subvariety.
Given $W \colon  X \ra \DA^1$ and $V \colon Y \ra \DA^1$ we define $W * V \colon  X
\times Y \ra \DA^1$ by $(W*V)(x,y)=W(x)+V(y).$ This operation
turns $K_0(\Var_{\DA^1})$ into a commutative ring.

% VALERA HATTE: Since $\DA^1$ is an algebraic group, the Gothendieck group is
% naturally a ring\fotnote{
%   naja, auch Ring per fiber product over $\DA^1.$ so better:
%   using that $\DA^1$ is an algebraic group we can define a ring
%   structure on ... unit.. brauche nur monid!?
% } 
% with multiplication
% \begin{equation*}
%   [X \xra{W} \DA^1] \cdot [Y \xra{V} \DA^1] = [X \times Y
%   \xra{W*V} \DA^1]
% \end{equation*}
% where $(W*V)(x,y)=W(x)+V(y).$ 

We denote by $K_0(\sat_k^{\DZ_2})$ the
Grothendieck group of saturated dg categories
(see Definition~\ref{def:Grothendieck-group-of-saturated-dg-cats}),
i.\,e.\ 
the free abelian group on
quasi-equivalence classes of saturated
(= proper, smooth and triangulated)
dg (= differential
$\DZ_2$-graded)
% , where $\DZ_2=\DZ/2\DZ$) 
categories with
relations coming from semi-orthogonal decompositions into
admissible subcategories on the level of homotopy categories.
The tensor product of dg categories gives rise to a ring
structure on $K_0(\sat_k^{\DZ_2}).$ One may think of
$K_0(\sat_k^{\DZ_2})$ as a Grothendieck 
ring of suitable pretriangulated dg categories.
Now we can state our main result.

\begin{theorem}
  [{see Theorem~\ref{t:category-of-singularities-induces-ring-morphism}}]
  \label{t:category-of-singularities-induces-ring-morphism-intro}
  There is a unique morphism
  \begin{equation}
    \label{eq:MF-mot-meas}
    \mu \colon  K_0(\Var_{\DA^1})\ra K_0(\sat_k^{\DZ_2})
  \end{equation}
  of rings (= a Landau-Ginzburg motivic measure) that maps
  $[X,W]$ to 
  the class of $\bfMF(W)^{\dg,\natural}$ whenever $X$ is a smooth
  variety and
  $W \colon  X \ra \DA^1$ is a proper morphism. 

  In particular, $\mu$ is a morphism of abelian groups and maps 
  $[X,W]$ to 
  the class of $\bfMF(W)^{\dg,\natural}$ whenever $X$ is a smooth
  (connected) variety and
  $W \colon  X \ra \DA^1$ is a projective morphism. These two properties
  determine $\mu$ uniquely.
\end{theorem}

Let us sketch the main steps of the proof of this theorem.

We first show that $\bfMF(W)^\dg$ and $\bfMF(W)^{\dg,\natural}$
are smooth dg 
categories (Theorem~\ref{t:MF-DSgW-Cechobj-smooth}), and proper
if $W$ is proper (Proposition~\ref{p:PerfWdg-proper}). Hence, for proper $W,$ 
$\bfMF(W)^{\dg,\natural}$ is a saturated dg category and
defines an element of $K_0(\sat_k^{\DZ_2}).$
The proof of smoothness takes advantage of good properties of
object oriented \v{C}ech enhancements of matrix factorization
categories; for example, the standard duality and external tensor
products admit natural lifts to these enhancements.
On the way we show a Thom-Sebastiani Theorem
(Theorem~\ref{t:thom-sebastiani}); it says that 
given smooth varieties $X$ and $Y$ with morphisms $W \colon  X\ra
\DA^1$ and $V \colon  Y \ra \DA^1,$ the   
two dg categories
$\bfMF(W)^{\dg} \otimes \bfMF(V)^{\dg}$
and $\bfMF(W*V)^{\dg}$
are Morita equivalent. If $W$ is
proper, properness follows essentially from \cite[Cor.~1.24]{orlov-tri-cat-of-sings-and-d-branes}.

% Then we show additivity. 
According to
\cite[Theorem~5.1]{bittner-euler-characteristic}),
$K_0(\Var_{\DA^1})$ has a
presentation with generators
the isomorphism classes $[X, W],$ where $X$ is a smooth variety
and $W$ 
% $ \colon X \ra \DA^1$ 
is a proper (or projective) morphism, and relations coming from
blowing-ups.
Using this, the semi-orthogonal decompositions for projective
space bundles and blowing-ups we established in
\cite[Theorems~\ref{semi:t:semi-orthog-1-mf} and \ref{semi:t:semi-orthog-2-mf}]{valery-olaf-matfak-semi-orth-decomp}
imply that there is a morphism of abelian
groups $\mu \colon  K_0(\Var_{\DA^1}) \ra K_0(\sat_k^{\DZ_2})$
sending $[X,W]$ to the class of $\bfMF(W)^{\dg,\natural}$ for
smooth $X$
and proper $W,$ and that this morphism is already uniquely
determined by its values 
on $[X,W]$ for smooth $X$ and projective $W.$

It remains to show multiplicativity of $\mu.$
If $W \colon  X \ra \DA^1$ and $V \colon  Y \ra \DA^1$ are proper,
the product of the classes of 
$\bfMF(W)^{\dg,\natural}$ and $\bfMF(V)^{\dg,\natural}$
in 
$K_0(\sat_k^{\DZ_2})$
is isomorphic 
to the class of 
$\bfMF(W*V)^{\dg,\natural},$
by the Thom-Sebastiani Theorem~\ref{t:thom-sebastiani}.
However, $W * V$ is not proper in general, so it is a priori not
clear that $\mu$ maps $[X \times Y, W*V]$ to the
class of 
$\bfMF(W*V)^{\dg,\natural}.$
To ensure this we
furthermore have to compactify the morphism $W*V$ in a nice way
(Proposition~\ref{p:compactification}) in order to obtain
multiplicativity. This finishes the sketch of proof of Theorem~\ref{t:category-of-singularities-induces-ring-morphism-intro}.

% \begin{definition}
%   \label{d:landau-ginzburg-measure}
%   A \define{Landau-Ginzburg motivic measure} (or \define{LG
%     motivic measure} or \define{motivic measure}) \fotnote{
%     motivic LG measure?  }  is a morphism of unital rings from
%   $K_0(\Var_{\DA^1})$ to another ring.
% \end{definition}

The Landau-Ginzburg measure \eqref{eq:MF-mot-meas} 
sends $\DL_{(\DA^1,0)}:=[\DA^1,0]$ to 1. It
is the analog of the motivic measure constructed in
\cite{bondal-larsen-lunts-grothendieck-ring} using bounded
derived categories of coherent sheaves. 
We refer the reader to 
the introduction of \cite{valery-olaf-matfak-semi-orth-decomp}
for more details. There we discuss also our related work.

For the convenience of readers who are more familiar with
smooth and proper dg algebras than with saturated dg categories
let us explain that the Grothendieck group $K_0(\sat_k^{\DZ_2})$
can also be described using proper and smooth dg algebras.
We call two proper and smooth dg algebras "dg Morita equivalent"
if their derived categories are connected by a zig-zag of tensor
equivalences (cf.\ Remark~\ref{rem:dg-Morita-equivalence-and-isomorphic-in-Hmo}).
Define the Grothendieck group $K_0(\prsmalg_k^{\DZ_2})$ of
proper and smooth dg algebras
as the quotient of the free abelian group on dg Morita
equivalence classes $\ul{A}$ of proper and smooth dg algebras 
$A$ by the subgroup generated by
the elements 
% (the "lower-triangular matrix algebra relations")
% \begin{equation*}
%   \label{eq:semi-orthog-relation}
$\ul{R} - (\ul{A} + \ul{B})$
% \end{equation*}
whenever $R$ is a proper and smooth dg algebra such that
there are dg algebras $A$ and $B$ together with a
dg $A \otimes B^\opp$-module 
$N=\leftidx{_B}{N}{_A}$ such that 
$R=\roundtzmat A0NB$ (see
Def.~\ref{def:Grothendieck-group-of-proper-smooth-dg-alg}).
The tensor product of dg algebras turns
$K_0(\prsmalg_k^{\DZ_2})$ into a ring.
Under our assumption that $k$ is a field
we show 
% in sections~\ref{sec:modif-groth-ring}, \ref{sec:groth-ring-prop-smooth-cats}
% and \ref{sec:groth-ring-prop-smooth-alg} 
that mapping a proper
and smooth dg algebra $A$ to its triangulated envelope $\Perf(A)$
induces an isomorphism
\begin{equation*}
  K_0(\prsmalg_k^{\DZ_2})
  \sira K_0(\sat_k^{\DZ_2})
\end{equation*}
of rings (Proposition~\ref{p:modified-K0} and Remark~\ref{rem:K0s-combined}).
Using this isomorphism, the Landau-Ginzburg motivic measure 
\eqref{eq:MF-mot-meas} may be viewed as a morphism of rings
% Landau-Ginzburg motivic measure 
\begin{equation*}
  \mu \colon  K_0(\Var_{\DA^1}) \ra K_0(\prsmalg_k^{\DZ_2}).
\end{equation*}
In this interpretation, $\mu$ can be described more concretely as
follows. Given a smooth variety $X$ and a proper morphism
$W\colon X \ra \DA^1,$ choose a classical generator in each
category $\bfMF(X, W-a)$ and let $A_a$ be its endomorphism dg
algebra (computed in a suitable enhancement).  Then the image of
$[X,W]$ under $\mu$ is the class of $\prod_{a \in k} A_a.$

Section~\ref{sec:groth-ring-saturated-dg-cats} contains the
definitions of various Grothendieck rings of dg categories. We
even work over an arbitrary commutative ground ring $\groundring$ 
there. Besides the Grothendieck rings
$K_0(\sat_\groundring^{\DZ_2})$ and
$K_0(\prsmalg_\groundring^{\DZ_2})$ explained above we also
introduce the modified Grothendieck ring
$K'_0(\sat_\groundring^{\DZ_2})$ of saturated dg categories and
the Grothendieck ring $K_0(\prsm_\groundring^{\DZ_2})$ of proper
and smooth dg categories.  There are canonical morphisms
\begin{equation*}
  K_0(\sat_\groundring^{\DZ_2})
  \sra
  K'_0(\sat_\groundring^{\DZ_2})
  \sila
  K_0(\prsm_\groundring^{\DZ_2})
  \sila
  K_0(\prsmalg_\groundring^{\DZ_2})
\end{equation*}
of rings. The first morphism is the obvious surjection:
the definition of $K'_0(\sat_\groundring^{\DZ_2})$ is obtained from that of
$K_0(\sat_\groundring^{\DZ_2})$ by dropping the words "into admissible
subcategories". It is an isomorphism if $\groundring$ is a field (Proposition~\ref{p:modified-K0}).
The other two morphisms are isomorphisms (see
Remark~\ref{rem:K0s-combined}).

\subsection*{Acknowledgments}
\label{sec:acknowledgements}

We thank Maxim Kontsevich, J{\'a}nos Koll{\'a}r, 
Gon\c{c}alo Tabuada, 
Daniel Pomerleano, and Anatoly Preygel for their help and useful
discussions.

The second author was
supported by a postdoctoral fellowship of 
the German Academic Exchange Service (DAAD)  
and partially supported by
the Collaborative Research Center SFB Transregio 45
and 
the priority program SPP 1388 
of the
German Science foundation (DFG). He thanks these institutions.

\section{Grothendieck rings of dg categories}
\label{sec:groth-ring-saturated-dg-cats}

Our aim is to define and compare some Grothendieck rings of saturated dg
$\groundring$-categories (or proper and smooth dg
$\groundring$-algebras) where $\groundring$ is a commutative
ring.
Initially we follow
\cite{bondal-larsen-lunts-grothendieck-ring} but pay 
more attention to finiteness conditions and work
over an arbitrary commutative ring.
We use notions and results from
\cite{keller-on-dg-categories-ICM,
  toen-finitude-homotopique-propre-lisse,toen-vaquie-moduli-objects-dg-cats}.
For the convenience of the reader we repeat some proofs.
In this section dg stands for "differential $\DZ$-graded".

\begin{remark}
  \label{rem:more-general}
  All results of this section (and the results we cite) are also
  true in the 
  differential $\DZ_n$-graded setting (where $\DZ_n:=\DZ/n\DZ$), for
  any $n \in \DZ,$ 
  unless said otherwise (the
  standard notions we use have their obvious counterpart in this
  setting).
  The proofs are easily adapted.  
%  In fact everything we do in this section is valid in the
%  differential 
%  $\DZ_N$-graded setting with $N \in \DZ_{\geq 0},$ and 
  We could
  even work over a graded 
  commutative differential $\DZ_n$-graded $\groundring$-algebra
  $K$ as in 
  \cite{valery-olaf-smoothness-equivariant}.
%\end{remark}

%\begin{remark}
%  \label{rem:dZ2g}
  In fact, in the rest of this article, we only need the
  differential $\DZ_2$-graded setting for $\groundring$ a field.
  To exclude misunderstandings, a dg module in this setting is a
  $\DZ_2$-graded $\groundring$-module $V=V_0 \oplus V_1$ together
  with a differential $d \colon V \ra V$ of degree 1, i.\,e.\
  $\groundring$-linear maps $d_i \colon V_i \ra V_{i+1}$ for $i \in \DZ_2$
  such that $d_{i+1}d_i=0.$ 

  We choose to explain
  the case of an arbitrary commutative ring $\groundring$
  since it is only 
  slightly more difficult than that of a field. 
\end{remark}

% In this section dg stands either for
% "differential $\DZ$-graded" or 
% "differential $\DZ_2$-graded".

% BLABLA einfuegen... 

% When considering the differential $\DZ_2$-graded setting

% Our notation coincides with that of
% \cite[sections~2 and 3]{valery-olaf-smoothness-equivariant} 
% (if one puts
% $K=\groundring$ there and possibly replaces "differential
% $\DZ$-graded" by "differential $\DZ_2$-graded"). Many of the
% results 
% there and elsewhere in the literature in the differential
% $\DZ$-graded setting also hold in the  
% differential $\DZ_2$-graded setting with essentially the same
% proof. When refering to these results we tacitly mean the
% corresponding differential $\DZ_2$-graded version.
% Similarly we transfer standard terminology (semi-free,
% pretriangulated, enhancement, quasi-equivalence...) from the 
% differential $\DZ$-graded setting to the 
% differential $\DZ_2$-graded setting.

% \begin{remark}
%   \label{rem:more-general}
%   In fact everything we do in this section is valid in the
%   differential 
%   $\DZ_N$-graded setting with $N \in \DZ_{\geq 0}$ (where
%   $\DZ_0=\DZ$), and we could also work over a graded
%   commutative differential $\DZ_N$-graded $\groundring$-algebra
%   $K$ as in 
%   \cite{valery-olaf-smoothness-equivariant}.
% \end{remark}

\subsection{Dg categories and their module categories}
\label{sec:dg-categories-their}

Our notation coincides with that of
\cite[sections~2 and 3]{valery-olaf-smoothness-equivariant} 
(if one puts
$K=\groundring$ there).
% and possibly replaces "differential
% $\DZ$-graded" by "differential $\DZ_2$-graded").

Let $\groundring$ be a commutative ring.
% We set up our notation and state some basic results.
We denote by
$C(\groundring)$ the 
category whose objects are dg
\mbox{($\groundring$-)modules}.
%, i.\,e.\ 
%$\DZ_2$-graded $\groundring$-modules $V=V_0 \oplus V_1$ together
%with a differential $d \colon V \ra V$ of degree 1, i.\,e.\
%$\groundring$-linear maps $d_i \colon V_i \ra V_{i+1}$ for $i \in \DZ_2$
%such that $d_{i+1}d_i=0.$ 
Morphisms in $C(\groundring)$ are
degree zero morphisms that commute with the respective
differentials. 
Note that the category $C(\groundring)$ is abelian and closed
symmetric monoidal with the obvious tensor product
$\otimes:=\otimes_\groundring.$

If $\mathcal{A}$ is a dg category (= a category enriched in
$C(\groundring)$) we denote the category with the same objects and closed
degree zero morphism (resp.\ closed degree zero morphisms up to
homotopy) by $Z_0(\mathcal{A})$ (resp.\ $[\mathcal{A}].$)
For example, there is an obvious dg category
$\calMod(\groundring)$ such that
$Z_0(\calMod(\groundring))=C(\groundring).$
We denote the category of small dg categories
by $\dgcat_\groundring.$

Let $\mathcal{A}$ be a small dg category. We denote by
$\calMod(\mathcal{A})$ the dg category of (right) dg
$\mathcal{A}$-modules (= dg functors
$\mathcal{A}^\opp \ra \calMod(\groundring)$).
The dg functor $Y \colon  \mathcal{A}
\ra \calMod(\mathcal{A}),$ $A \mapsto
Y(A):=\Yoneda{A}:=\mathcal{A}(-,A),$ is full and faithful and
called Yoneda embedding. 
We write $C(\mathcal{A}):=Z_0(\calMod(\mathcal{A}))$ and
$\mathcal{H}(\mathcal{A}):=[\calMod(\mathcal{A})]$ (resp.\
$D(\mathcal{A})$) for the
homotopy (resp.\ derived) category of dg
$\mathcal{\mathcal{A}}$-modules. 
We equip $C(\mathcal{A})$ with the (cofibrantly generated)
projective model structure 
(cf.\ \cite[Thm.~2.2]{valery-olaf-smoothness-equivariant}).
Its weak equivalences are the quasi-isomorphisms, and its
fibrations are the epimorphisms.

Let $\calMod(\mathcal{A})_{cf} \subset \calMod(\mathcal{A})$ be the
full dg subcategory of all cofibrant objects 
in $C(\mathcal{A})$ (all objects
are fibrant).
% Let $\calMod(\mathcal{A})_{sf}
% \subset \calMod(\mathcal{A})$ be the full dg subcategory of semi-free dg
% $\mathcal{A}$-modules.
We denote by $\ol{\mathcal{A}} \subset
\calMod(\mathcal{A})$ the smallest strict full dg subcategory which
contains the zero module, all $\Yoneda{A},$ for $A \in \mathcal{A},$
and is closed under cones of closed degree zero morphisms
(and then also under all shifts).
% :
% If $M, N$ are objects of $\ol{\mathcal{A}}$ and $f$ is in
% $Z^0(\calMod \mathcal{A})(M,N))=(C(\mathcal{A}))(M,N),$
% then $\Cone(f) \in \ol{\mathcal{A}}.$ (It is then also closed under
% all shifts.)
Any object of $\ol{\mathcal{A}}$ is a semi-free
dg $\mathcal{A}$-module and hence cofibrant.
The situation is illustrated by
the diagram
\begin{equation*}
  {\mathcal{A}} \xra{Y}
  {\ol{\mathcal{A}} \subset} 
  % \calMod(\mathcal{A})_{sf} \ar@{}[r]|-{\subset} &
  {\calMod(\mathcal{A})_{cf}} \subset
  \calMod(\mathcal{A}).
  % \xymatrix{
  %   {\mathcal{A}} \iar[r]^-{Y} &
  %   {\ol{\mathcal{A}}} \ar@{}[r]|-{\subset} &
  %   % \calMod(\mathcal{A})_{sf} \ar@{}[r]|-{\subset} &
  %   {\calMod(\mathcal{A})_{cf}} \ar@{}[r]|-{\subset} &
  %   {\calMod(\mathcal{A}),}
  % }
\end{equation*}
% in which $Y$ is the full and faithful Yoneda dg embedding.
The three categories on the right are 
% stable under all shifts and
% under taking cones of closed degree zero morphisms; in particular they
(strongly) pretriangulated.
The canonical full and faithful
dg functor
$\mathcal{A}^{\pretr} \ra \calMod(\mathcal{A})$
from \cite[\S1, \S4]{bondal-kapranov-enhancements}
(an extended version of the Yoneda embedding)
has precisely $\ol{\mathcal{A}}$ as its essential image, so
$\mathcal{A}^\pretr$ and $\ol{\mathcal{A}}$ are dg
equivalent. So $\ol{\mathcal{A}}$ is a pretriangulated envelope
of $\mathcal{A}.$
We pass to the respective homotopy categories and obtain the first row
in the commutative diagram
\begin{equation*}
  \xymatrix{
    {[\mathcal{A}]} \iar[r]^-{[Y]} &
    {[\ol{\mathcal{A}}]} \ar@{}[r]|-{\subset}
    \iar[d]^{\sim}%|-{\cap}
    &
    % \calMod(\mathcal{A})_{sf} \ar@{}[r]|-{\subset} &
    {[\calMod(\mathcal{A})_{cf}]} \ar@{}[r]|-{\subset} \ar[rd]^{\sim}&
    {\mathcal{H}(\mathcal{A})} \ar[d] \\
    &
    {\tria(\mathcal{A})} \ar@{}[r]|-{\subset} &
    {\thick(\mathcal{A})} \ar@{}[u]|-{\cup} \ar[rd]^{\sim} &
    {D(\mathcal{A})} \\
    &&&
    {\per(\mathcal{A}).} \ar@{}[u]|-{\cup}
  }
\end{equation*}
Here $\tria(\mathcal{A})$ is defined to be the smallest strict
full triangulated subcategory of $[\calMod(\mathcal{A})_{cf}]$
that contains all $\Yoneda{A},$ for $A \in \mathcal{A},$ 
$\thick(\mathcal{A})$ is in addition required to be closed under
direct summands in $[\calMod(\mathcal{A})_{cf}],$
and 
$\per(\mathcal{A})$ is defined to be the thick envelope of
$\{\Yoneda{A}\mid A \in \mathcal{A}\}$ in $D(A).$
The three
indicated triangulated equivalences are well-known (or obvious).
They show that $\ol{\mathcal{A}}$ (together with equivalence
$[\ol{\mathcal{A}}] \ra \tria(\mathcal{A})$) is an enhancement of
$\tria(\mathcal{A}),$ and that $\calMod(\mathcal{A})_{cf}$ is an
enhancement of $[\calMod(\mathcal{A})_{cf}]$ and
$D(\mathcal{A}).$ We define $\Perf(\mathcal{A})$ to be the full
subcategory of $\calMod(\mathcal{A})_{cf}$ whose objects coincide
with those of $\thick(\mathcal{A}).$ Then $\Perf(\mathcal{A})$ is
(strongly) pretriangulated and an enhancement of
$\thick(\mathcal{A})$ and $\per(\mathcal{A}).$

% It is well known that the obvious composition defines a triangulated
% equivalence between
% $[\calMod(\mathcal{A})_{cf}]$ and the derived category
% $D(\mathcal{A})$ of dg $\mathcal{A}$-modules.
% This equivalence restricts to the indicated equivalence
% $\thick(\mathcal{A}) \sira \per(\mathcal{A}),$ where

The categories $[\calMod(\mathcal{A})_{cf}]$ and $D(\mathcal{A})$ have
arbitrary (in particular countable) coproducts. Hence they are Karoubian
(= idempotent complete), and 
% = all idempotent endomorphism of objects
% split). 
so are
$\thick(\mathcal{A})$ and $\per(\mathcal{A}).$
In particular $\thick(\mathcal{A})$ can be viewed as the
Karoubi envelope (= idempotent completion) of
$\tria(\mathcal{A}).$ 
Note also that
$D(\mathcal{A})$ is compactly generated and that the subcategory
$D(\mathcal{A})^c$ of compact
objects in $D(\mathcal{A})$ is precisely 
$\per(\mathcal{A}),$ i.\,e.\
$D(\mathcal{A})^c=\per(\mathcal{A})$ (cf.\ the discussion
around equation (2.4) in 
\cite{valery-olaf-smoothness-equivariant}).

\subsection{Triangulated dg categories}
\label{sec:triang-dg-categories}

Recall that a dg functor $F \colon \mathcal{A} \ra \mathcal{B}$ is a
\define{quasi-equivalence} if
\begin{enumerate}[label=(qe{\arabic*})]
\item
  \label{enum:quasi-equi-qiso-homs}
  for all objects $a_1, a_2 \in \mathcal{A},$ the morphism
  $F \colon  \mathcal{A}(a_1, a_2) \ra \mathcal{B}(Fa_1, Fa_2)$
  is a quasi-isomorphism, and
\item
  \label{enum:quasi-equi-H0-essentially-epi}
  the induced functor $[F] \colon  [\mathcal{A}] \ra[\mathcal{B}]$ on
  homotopy categories is
  essentially surjective.
  %(i.\,e.\ surjective on isoclasses of objects).
\end{enumerate}
If \ref{enum:quasi-equi-qiso-homs} holds, then
\ref{enum:quasi-equi-H0-essentially-epi} is equivalent to the condition
that $[F] \colon  [\mathcal{A}] \ra[\mathcal{B}]$ is an equivalence.

\begin{definition}
  [{\cite[Def.~2.4.5]{toen-vaquie-moduli-objects-dg-cats}}]
  \label{d:triangulated-dg-category}
  A dg category $\mathcal{A}$ is \define{triangulated}
  if the
  Yoneda embedding induces a quasi-equivalence
  $\mathcal{A} \ra \Perf(\mathcal{A}).$
  % \begin{equation*}
  %   \mathcal{A} \ra \Perf(\mathcal{A}).
  % \end{equation*}
  (It is enough to require that
  $[\mathcal{A}] \ra [\Perf(\mathcal{A})]$ is essentially surjective.)
\end{definition}

\begin{lemma}
  \label{l:triang-versus-pretriang-and-karoubi}
  A dg category $\mathcal{A}$ is triangulated if and only if it is
  pretriangulated and $[\mathcal{A}]$ is Karoubian.
  % (= idempotent complete).
\end{lemma}

\begin{proof}
  Note that $\mathcal{A} \ra
  \Perf(\mathcal{A})$ factors as $\mathcal{A} \hra \ol{\mathcal{A}}
  \subset \Perf(\mathcal{A}).$
  % Consider the diagram
  % \begin{equation*}
  %   \xymatrix{
  %   {\mathcal{A}} \iar[r]^-{Y} &
  %   {\ol{\mathcal{A}}} \ar@{}[r]|-{\subset} &
  %   {\Perf(\mathcal{A})}\\
  %   {[\mathcal{A}]} \iar[r]^-{[Y]} &
  %   {[\ol{\mathcal{A}}]} \ar@{}[r]|-{\subset} \iar[d]^-{\sim}&
  %   {[\Perf(\mathcal{A})]} \gar[d] \\
  %   &
  %   {\tria(\mathcal{A})} \ar@{}[r]|-{\subset} &
  %   {\thick(\mathcal{A})}
  % }
  % \end{equation*}

  Assume that $\mathcal{A}$ is triangulated. Then it is clear
  that $\mathcal{A} \ra \ol{\mathcal{A}}$ is a quasi-equivalence
  which precisely means that $\mathcal{A}$ is pretriangulated.
  Since $[\mathcal{A}] \ra [\Perf(\mathcal{A})]$ is an
  equivalence and $[\Perf(\mathcal{A})]=\thick(\mathcal{A})$ is
  Karoubian, the same is true for $[\mathcal{A}].$

  Conversely, if $\mathcal{A}$ is pretriangulated and
  $[\mathcal{A}]$ is Karoubian, then
  $[\mathcal{A}] \sira [\ol{\mathcal{A}}] \sira \tria(\mathcal{A}),$
  so $\tria(\mathcal{A})$ is Karoubian and coincides with its
  Karoubi envelope $\thick(\mathcal{A})=[\Perf(\mathcal{A})].$
  This implies that $[\mathcal{A}] \ra [\Perf(\mathcal{A})]$ is an
  equivalence, so $\mathcal{A}$ is triangulated.
\end{proof}

\begin{corollary}
  [{\cite[Lemma.~2.6]{toen-vaquie-moduli-objects-dg-cats}}]
  \label{c:triang-versus-pretriang-and-karoubi}
  Let $\mathcal{A}$ be a dg category. Then
  $\Perf(\mathcal{A})$ is a triangulated dg category,
  i.\,e.\ the morphism
  % \begin{equation*}
  $\Perf(\mathcal{A}) \ra \Perf(\Perf(\mathcal{A}))$
  % \end{equation*}
  induced by the Yoneda functor is a quasi-equivalence.
\end{corollary}

\begin{proof}
  We have observed above
  that $\Perf(\mathcal{A})$ is pretriangulated, and that
  $[\Perf(\mathcal{A})]=\thick(\mathcal{A})$ is Karoubian.
\end{proof}

So passing from a dg category $\mathcal{A}$ to
$\Perf(\mathcal{A})$ means taking a triangulated envelope.

\subsection{Some general results}
\label{sec:some-general-results}

\begin{lemma}
  \label{l:dg-functor-qequi-test}
  Let $F \colon  \mathcal{A} \ra \mathcal{B}$ be a dg functor and
  assume that the following condition on $\mathcal{A}$ holds: For
  all $X \in \mathcal{A}$ and $r \in \DZ$ there is an object $Z
  \in \mathcal{A}$ and morphisms $f \in \mathcal{A}(X,Z)^{-r}$
  and $g \in \mathcal{A}(Z,X)^{r}$ such that $df=0,$ $dg=0,$ and
  $fg$ is homotopic to $\id_Z$ and $gf$ is homotopic to $\id_X.$
  In other words, the essential image of $[\mathcal{A}]$ in
  $[\calMod(\mathcal{A})]$ is closed under all shifts.
  (This condition is for example
  satisfied if $\mathcal{A}$ is pretriangulated or closed under
  all shifts.)

  Then $F$ is a quasi-equivalence if and only if
  $[F] \colon [\mathcal{A}] \ra [\mathcal{B}]$ is an equivalence.
\end{lemma}

\begin{proof}
  We prove the non-trivial implication.
  Assume that $[F]$ is an equivalence. Then obviously
  \ref{enum:quasi-equi-H0-essentially-epi} is satisfied.
  Let $A, X \in \mathcal{A}.$
  Let $r \in \DZ$ and let $Z,$ $f,$ $g$ be as above.
  Consider the commutative diagram
  \begin{equation*}
    \xymatrix{
      {[r]\mathcal{A}(A,X)}
      \ar[r]^-{f \comp ?} \ar[d]^{[r]F}
      &
      {\mathcal{A}(A,Z)} \ar[d]^F
      \\
      {[r]\mathcal{B}(FA,FX)}
      \ar[r]^-{F(f) \comp ?}
      &
      {\mathcal{B}(FA,FZ)}
    }
  \end{equation*}
  in $C(\groundring).$ If we apply $H^0$ to this diagram, the
  horizontal arrows and the vertical arrow on the right become
  isomorphisms. Hence the same is true for the vertical arrow on
  the left, i.\,e.\
%  \begin{equation*}
  $H^r(F) \colon  H^r(\mathcal{A}(A,X)) \ra H^r(\mathcal{A}(FA,FX))$
%  \end{equation*}
  is an isomorphism. This proves \ref{enum:quasi-equi-qiso-homs}.
\end{proof}

Any dg functor $f \colon  \mathcal{A} \ra \mathcal{B}$ gives rise
to the dg functor $f^*= \pro_\mathcal{A}^\mathcal{B} = (- \otimes_\mathcal{A}
\mathcal{B}) \colon \calMod(\mathcal{A}) \ra \calMod(\mathcal{B})$
called extension of scalars functor, and we
have a commutative diagram
\begin{equation}
  \label{eq:yoneda-extension-of-scalars}
  \xymatrix{
    {\mathcal{A}} \ar[d]^-{f} \ar[r]^-{Y} &
    {\calMod(\mathcal{A})} \ar[d]^-{f^*}\\
    {\mathcal{B}} \ar[r]^-{Y} &
    {\calMod(\mathcal{B})}
  }
\end{equation}
since $f^*(\Yoneda{A}) =\Yoneda{f(A)}$ for all $A \in \mathcal{A}.$

\begin{lemma}
  \label{l:dg-functors-restricts-to-perfects-embedding-and-qequiv}
  Let $f \colon  \mathcal{A} \ra \mathcal{B}$ be a morphism in
  $\dgcat_\groundring.$ Then:
  \begin{enumerate}%[label=({\arabic*})]
  \item
    \label{enum:restricts-to-Perf}
    The extension of scalars functor $f^*$
    % $f^* \colon \calMod(\mathcal{A}) \ra \calMod(\mathcal{B})$
    % (a morphism in $dgcat_K$)
    induces dg
    functors
    $f^* \colon \calMod(\mathcal{A})_{cf} \ra \calMod(\mathcal{B})_{cf},$
    $f^* \colon \Perf(\mathcal{A}) \ra \Perf(\mathcal{B}),$
    and
    $f^* \colon \ol{\mathcal{A}} \ra \ol{\mathcal{B}}.$
  \item
    \label{enum:embeddings}
    If $f$ is full and faithful, then all these functors $f^*$ are
    full and faithful.
  \item
    \label{enum:qequi-induces-qequi-between-perfs}
    If $f$ is a quasi-equivalence, then
    $f^* \colon \calMod(\mathcal{A})_{cf} \ra \calMod(\mathcal{B})_{cf},$
    $f^* \colon \Perf(\mathcal{A}) \ra \Perf(\mathcal{B}),$
    and
    $f^* \colon \ol{\mathcal{A}} \ra \ol{\mathcal{B}}$ are
    quasi-equivalences.
 \end{enumerate}
\end{lemma}

\begin{proof}
  We prove \ref{enum:restricts-to-Perf}.
  Note first that $f^*$ (viewed as a functor $C(\mathcal{A}) \ra
  C(\mathcal{B})$ where both categories are viewed as model categories
  with the projective model structure) maps cofibrations to cofibrations since its right
  adjoint $f_*=\res^\mathcal{B}_\mathcal{A}$ maps trivial fibrations (=
  epimorphic quasi-isomorphisms)
  to trivial
  fibrations. In particular $f^*$ induces a dg functor
  $f^* \colon \calMod(\mathcal{A})_{cf} \ra \calMod(\mathcal{B})_{cf}.$ For the
  remaining statements of
  \ref{enum:restricts-to-Perf} use
  $f^*(\Yoneda{A}) =\Yoneda{f(A)}$ for all $A \in \mathcal{A}$
  and the fact that a dg functor preserves shifts and cones of
  closed degree zero morphisms (see
  \cite[4.3]{bondal-larsen-lunts-grothendieck-ring}).

  In order to prove \ref{enum:embeddings} assume that $f$ is full and faithful.
  The right adjoint of
  \begin{equation}
    \label{eq:extension-of-scalars-embedding}
    f^*=\pro_\mathcal{A}^\mathcal{B} \colon  \calMod(\mathcal{A}) \ra \calMod(\mathcal{B})
  \end{equation}
  is
  restriction $f_*=\res_\mathcal{A}^\mathcal{B}.$ Hence $f^*$ is full and
  faithful if and only if the unit
  $X \ra
  \res_\mathcal{A}^\mathcal{B}(\pro_\mathcal{A}^\mathcal{B}(X))$ of
  this adjunction is an isomorphism for all $X \in
  \calMod(\mathcal{A}).$
  But
  \begin{multline*}
    (\res_\mathcal{A}^\mathcal{B}(\pro_\mathcal{A}^\mathcal{B} (X))) (A)
    = (\pro_\mathcal{A}^\mathcal{B} (X))(f(A))\\
    = \cokern\Big(
    \bigoplus_{A''', A'' \in \mathcal{A}}X(A''')\otimes \mathcal{A}(A'', A''') \otimes
    \mathcal{B}(f(A), f(A''))
    \ra
    \bigoplus_{A' \in \mathcal{A}} X(A')\otimes \mathcal{B}(f(A), f(A'))
    \Big)\\
    \xsila{f} \cokern\Big(
    \bigoplus_{A''', A'' \in \mathcal{A} }X(A''')\otimes \mathcal{A}(A'', A''') \otimes
    \mathcal{A}(A, A'')
    \ra
    \bigoplus_{A' \in \mathcal{A}} X(A')\otimes \mathcal{A}(A, A')
    \Big)\\
    \sira X(A),
  \end{multline*}
  where the first arrow is an isomorphism since $f$ is full and
  faithful, and the second arrow is the obvious evaluation
  morphism. 
  Under this identification the unit becomes the identity which shows
  that the functor $f^*$ in
  \eqref{eq:extension-of-scalars-embedding}
  is full and faithful. Then
  $f^*$ is obviously also full and faithful on
  all full subcategories of $\calMod(\mathcal{A}).$

  Let us prove
  \ref{enum:qequi-induces-qequi-between-perfs}. Assume that $f$
  is a quasi-equivalence.  View $\mathcal{A} \subset
  \calMod(\mathcal{A})$ and $\mathcal{B} \subset
  \calMod(\mathcal{B})$ as full dg subcategories via the Yoneda
  embedding. 
  It is easy to prove that $f^* \colon \ol{\mathcal{A}} \ra
  \ol{\mathcal{B}}$ is a quasi-equivalence (this statement
  corresponds to \cite[Prop.~2.5]{drinfeld-dg-quotients} under
  the dg equivalences $\mathcal{A}^\pretr \sira \ol{\mathcal{A}}$
  and $\mathcal{B}^\pretr \sira \ol{\mathcal{B}}$).  
  % In order to
  % prove that $f^*$ induces quasi-isomorphisms on morphism
  % spaces
  % one starts with objects of the form $\Yoneda{A},$ for $A \in
  % \mathcal{A},$ where this statement is true by the Yoneda
  % embedding; then one extends this statement to shifts and
  % cones
  % of closed degree zero morphisms (use
  % pre-distinguished fussnote
  % { spreche normal einfach von triangles statt distinguished
  % triangles } triangles and the five Lemma).  A similar
  % inductive approach shows that $[f^*]:[\ol{\mathcal{A}}] \ra
  % [\ol{\mathcal{B}}]$ is essentially surjective.
  In particular, $[f^*] \colon [\ol{\mathcal{A}}] \ra
  [\ol{\mathcal{B}}]$ and hence $[f^*] \colon \tria(\mathcal{A})\ra
  \tria(\mathcal{B})$ are equivalences. It extends to an
  equivalence between the corresponding Karoubi 
  envelopes given by $[f^*] \colon \thick(\mathcal{A})\ra
  \thick(\mathcal{B}).$ 
  This also implies that $\bL f^*=(- \otimes^\bL _\mathcal{A}
  \mathcal{B}) \colon  D(\mathcal{A}) \ra D(\mathcal{B})$ is 
  an equivalence (\cite[4.2, Lemma]{Keller-deriving-dg-cat}). 
  Now
  Lemma~\ref{l:dg-functor-qequi-test} shows that $f^* \colon 
  \Perf(\mathcal{A}) \ra \Perf(\mathcal{B})$ and
  $f^* \colon \calMod(\mathcal{A})_{cf} \ra
  \calMod(\mathcal{B})_{cf}$ are quasi-equivalences.
\end{proof}

\begin{lemma}
  [{\cite[Lemma~4.16]{bondal-larsen-lunts-grothendieck-ring}}]
  \label{l:dense-in-homotopy-cat}
  Let $\mathcal{A}$ be a full dg subcategory
  of a dg category $\mathcal{B}.$ Assume that
  $[\mathcal{A}] \subset [\mathcal{B}]$ is dense in the sense that any object of
  $[\mathcal{B}]$ is a direct summand of an object of
  $[\mathcal{A}].$
  Then $[\ol{\mathcal{A}}]$ is dense in $[\ol{\mathcal{B}}]$
  where we view $\ol{\mathcal{A}}$ as a full dg subcategory of
  $\ol{\mathcal{B}}$ via
  Lemma~\ref{l:dg-functors-restricts-to-perfects-embedding-and-qequiv}.\ref{enum:embeddings}.
  % Die Inklusion wirkt nat{\"u}rlicher, wenn man mit $(?)^\pretr$ arbeitet: $\mathcal{A}^\pretr
  % \subset \mathcal{B}^\pretr.$
\end{lemma}

\begin{proof}
  We view $\mathcal{A} \subset \ol{\mathcal{A}}$ and $\mathcal{B}
  \subset \ol{\mathcal{B}}$ as full dg
  subcategories via the Yoneda embedding.
  Any object of $[\mathcal{B}]$ is a direct summand of an object of
  $[\mathcal{A}]$ and hence also of an object of
  $[\ol{\mathcal{A}}].$
  If an object of $[\ol{\mathcal{B}}]$ is a direct summand of an
  object of $[\ol{\mathcal{A}}],$ then all its shifts have the same
  property.

  Assume that $f \colon  X \ra Y$ is a closed degree zero morphism in
  $\ol{\mathcal{B}},$ and that
  $X \oplus X' \cong M$
  and
  $Y\oplus Y' \cong N$ in $[\ol{\mathcal{B}}]$
  with $M \in [\ol{\mathcal{A}}]$ and $N \in [\ol{\mathcal{A}}].$
  Consider $f \oplus 0 \colon  X\oplus X' \ra Y \oplus Y'$ and let $g \colon M \ra
  N$ be a closed degree zero morphism 
  in $\ol{\mathcal{A}}$
  corresponding to $f \oplus 0$ 
  in $[\ol{\mathcal{B}}].$
  Then
  \begin{equation*}
    \Cone(g) \cong \Cone(f\oplus 0) \cong \Cone(f) \oplus
    Y' \oplus [1]X'
  \end{equation*}
  in $[\ol{\mathcal{B}}].$ Hence $\Cone(f)$ is a direct summand of
  the object $\Cone(g) \in [\ol{\mathcal{A}}].$

  Now use that any object of $[\ol{\mathcal{B}}]$ is built up from
  the objects of 
  $\mathcal{B}$ using
  shifts and cones of closed degree zero morphisms.
\end{proof}

\subsection{Perfection of tensor products}
\label{sec:perf-tens-prod}

% We prove the following strengthening of
% Corollary~\ref{c:triang-versus-pretriang-and-karoubi}.

\begin{proposition}
  [{cf.\ \cite[Prop.~4.17]{bondal-larsen-lunts-grothendieck-ring}}]
  \label{p:perf-to-perf-perf}
  Let $\mathcal{A}$ and $\mathcal{B}$ be dg categories.
  Then the dg functor
  % \begin{equation*}
  $f \colon \mathcal{A} \otimes
  \mathcal{B} \ra \Perf(\mathcal{A}) \otimes \mathcal{B}$
  % \end{equation*}
  obtained from
  % the Yoneda embedding
  $\mathcal{A} \ra \Perf(\mathcal{A})$
  is full and faithful and
  induces (by extension of scalars along $f$) a quasi-equivalence
  % \begin{equation*}
  $f^* \colon \Perf(\mathcal{A} \otimes \mathcal{B}) \ra
  \Perf(\Perf(\mathcal{A}) \otimes \mathcal{B})$
  % \end{equation*}
  of dg categories.
\end{proposition}

\begin{proof}
  The sequence
  $\mathcal{A} \hra \ol{\mathcal{A}} \subset \Perf(\mathcal{A})$
  of full and faithful dg functors yields a sequence
  \begin{equation*}
    \mathcal{A} \otimes \mathcal{B} \hra \ol{\mathcal{A}} \otimes
    \mathcal{B} \subset \Perf(\mathcal{A}) \otimes \mathcal{B}
  \end{equation*}
  of full and
  faithful dg functors whose composition is $f.$
  By
  Lemma~\ref{l:dg-functors-restricts-to-perfects-embedding-and-qequiv}.\ref{enum:embeddings}
  we obtain full and faithful
  dg functors
  \begin{equation*}
    \ol{\mathcal{A} \otimes \mathcal{B}}
    \hra \ol{\ol{\mathcal{A}} \otimes \mathcal{B}}
    \subset \ol{\Perf(\mathcal{A}) \otimes \mathcal{B}}.
  \end{equation*}
  The functor on the left is an equivalence of dg
  categories (and in particular a quasi-equivalence) since both
  categories are built up from the objects of $\mathcal{A}
  \otimes \mathcal{B}$ using shifts and cones of
  closed degree zero 
  morphisms. 
  The first row of the commutative diagram
  \begin{equation*}
    \xymatrix@R=3ex{
      {[\ol{\mathcal{A} \otimes \mathcal{B}}]}
      \ar[d]^-\sim
      \ar[r]^-\sim
      &
      {[\ol{\ol{\mathcal{A}} \otimes \mathcal{B}}]}
      \ar[d]^-\sim
      \ar@{}[r]|-{\subset}
      &
      {[\ol{\Perf(\mathcal{A}) \otimes \mathcal{B}}]}
      \ar[d]^-\sim
      \\
      {\tria(\mathcal{A} \otimes \mathcal{B})} \iar[d] \ar[r]^-\sim
      &
      {\tria(\ol{\mathcal{A}} \otimes \mathcal{B})} \iar[d]
      \ar@{}[r]|-{\subset}
      % \ar[r]
      &
      {\tria(\Perf(\mathcal{A}) \otimes \mathcal{B})} \iar[d]
      \\
      {\thick(\mathcal{A} \otimes \mathcal{B})} \gar[d] \ar[r]^-\sim
      &
      {\thick(\ol{\mathcal{A}} \otimes \mathcal{B})}
      \gar[d] 
      \ar@{}[r]|-{=}
      % \ar[r]
      &
      {\thick(\Perf(\mathcal{A}) \otimes \mathcal{B})} \gar[d]
      \\
      {[\Perf(\mathcal{A} \otimes \mathcal{B})]} \ar[r]^-\sim
      &
      {[\Perf(\ol{\mathcal{A}} \otimes \mathcal{B})]}
      \ar@{}[r]|-{=}
      % \ar[r]
      &
      {[\Perf(\Perf(\mathcal{A}) \otimes \mathcal{B})]}
    }
  \end{equation*}
  is obtained by passing to the respective homotopy
  categories. Its left arrow is an equivalence, and we claim that
  its inclusion is dense: Since $[\ol{\mathcal{A}}] \sira
  \tria(\mathcal{A}),$ the inclusion $[\ol{\mathcal{A}}] \subset
  [\Perf(\mathcal{A})]=\thick(\mathcal{A})$ is dense; then
  $[\ol{\mathcal{A}} \otimes \mathcal{B}] \subset
  [\Perf(\mathcal{A}) \otimes \mathcal{B}]$ is dense, too, and
  Lemma~\ref{l:dense-in-homotopy-cat} shows our claim.  The
  second row is in the obvious way equivalent to the first one,
  and passing to the third row means taking the respective
  Karoubi envelopes; in particular, the dense inclusion becomes
  an equality.  The fourth row is equal to the third row, and
  Lemma~\ref{l:dg-functor-qequi-test} proves that both arrows in
  % \begin{equation*}
  $\Perf(\mathcal{A} \otimes \mathcal{B}) \ra
  \Perf(\ol{\mathcal{A}} \otimes 
  \mathcal{B}) \ra  \Perf(\Perf(\mathcal{A}) \otimes \mathcal{B})$
  % \end{equation*}
  are quasi-equivalences. The composition of these arrows is
  $f^*.$ 
\end{proof}

We
equip $\dgcat_\groundring$ with the (cofibrantly
generated) model structure\footnote{
  In case that 
  $\groundring$
  is a field (which is all we need in this article), the rest of
  this section can be simplified: we don't need this model
  structure and can assume that $Q(\mathcal{A})=\mathcal{A}$ for
  any dg category $\mathcal{A}$ (since all we need is that
  $Q(\mathcal{A})$ is h-flat; but any
  dg module over a field $\groundring$ is $\groundring$-h-flat).
} 
from \cite{tabuada-model-structure-on-cat-of-dg-cats} (cf.\
\cite[section~2.7]{valery-olaf-smoothness-equivariant}). Its weak
equivalences are the quasi-equivalences, and the cofibrant dg
categories are precisely the retracts of semi-free dg categories.
We denote the homotopy category of
$\dgcat_\groundring$ with respect to these weak
equivalences by $\Heq_\groundring.$
We fix a cofibrant replacement functor $Q.$
If $\mathcal{A}$ and $\mathcal{B}$ are dg categories we define
$\mathcal{A} \otimes^\bL \mathcal{B}:= Q(\mathcal{A}) \otimes
Q(\mathcal{B})$ and 
\begin{equation}
  \label{eq:def-triang-tensor-product-of-dg-cats}
  \mathcal{A} 
  \Perfotimes \mathcal{B} := \Perf(\mathcal{A} \otimes^\bL
  \mathcal{B}). 
\end{equation}
One may consider 
$\mathcal{A} \Perfotimes \mathcal{B}$ as a triangulated
envelope of $\mathcal{A} \otimes^\bL
  \mathcal{B}$ (cf.\ Corollary~\ref{c:triang-versus-pretriang-and-karoubi}).

\begin{lemma}
  \label{l:Perfotimes-respects-qequivalences}
  Quasi-equivalences $\mathcal{A} \ra \mathcal{A}'$ and
  $\mathcal{B} \ra \mathcal{B}'$ give rise to a quasi-equivalence
  $\mathcal{A} \Perfotimes \mathcal{B} \ra \mathcal{A}'
  \Perfotimes \mathcal{B}'.$
\end{lemma}

\begin{proof}
  Observe that
  $\mathcal{A} \otimes^\bL \mathcal{B} = Q(\mathcal{A}) \otimes
  Q(\mathcal{B}) \ra Q(\mathcal{A}') \otimes Q(\mathcal{B})
  = \mathcal{A}' \otimes^\bL \mathcal{B}$ is a quasi-equivalence
  since the cofibrant dg 
  category $Q(\mathcal{B})$ is $\groundring$-h-flat
  (by \cite[Lemmata~2.14 and 2.15]{valery-olaf-smoothness-equivariant}).
  Hence we obtain a quasi-equivalence
  $\mathcal{A} \Perfotimes \mathcal{B} \ra \mathcal{A}'
  \Perfotimes \mathcal{B}$ 
  by
  Lemma~\ref{l:dg-functors-restricts-to-perfects-embedding-and-qequiv}.\ref{enum:qequi-induces-qequi-between-perfs}.
\end{proof}

\begin{proposition}
  \label{p:triangulated-tensor-product-of-cats-and-their-Perfs}
  Let $\mathcal{A}$ and $\mathcal{B}$ be dg categories.
  Then the natural morphism
  \begin{equation*}
    % \label{eq:perf-perf-perf-qe-perf}
    \mathcal{A} \Perfotimes \mathcal{B}
    = \Perf(\mathcal{A} \otimes^\bL \mathcal{B})
    \ra \Perf((\Perf(\mathcal{A})) \otimes^\bL
    (\Perf(\mathcal{B}))) 
    = \Perf(\mathcal{A}) \Perfotimes \Perf(\mathcal{B})
  \end{equation*}
  in $\dgcat_\groundring$ is a quasi-equivalence (and becomes an isomorphism
  in $\Heq_\groundring$).
\end{proposition}

\begin{proof}
  Let $Y \colon  \mathcal{A} \ra \Perf(\mathcal{A})$ be the full and faithful Yoneda dg
  functor.
  The cofibrant replacement functor $Q$ comes with a natural
  transformation $Q \ra \id$ and yields the
  commutative square
  \begin{equation*}
    \xymatrix{
      {Q(\mathcal{A})} \ar[r]^-{Q(Y)}
      \ar[d]%|-{\trFib}
      &
      {Q(\Perf(\mathcal{A}))}
      \ar[d]%|-{\trFib}
      \\
      {\mathcal{A}} \ar[r]^-{Y} &
      {\Perf(\mathcal{A})}
    }
  \end{equation*}
  whose vertical arrows are trivial fibrations.
  We tensor this diagram with $Q(\mathcal{B})$ and obtain the commutative square
  \begin{equation*}
    \xymatrix{
      {Q(\mathcal{A}) \otimes Q(\mathcal{B})} \ar[rr]^-{Q(Y) \otimes \id}
      \ar[d]%|-{\trFib}
      & &
      {Q(\Perf(\mathcal{A})) \otimes Q(\mathcal{B})}
      \ar[d]%|-{\trFib}
      \\
      {\mathcal{A} \otimes Q(\mathcal{B})} \ar[rr]^-{Y \otimes \id} & &
      {\Perf(\mathcal{A}) \otimes Q(\mathcal{B})}
    }
  \end{equation*}
  whose vertical arrows are still quasi-equivalences
  since
  $Q(\mathcal{B})$ is $\groundring$-h-flat
  (they are even
  trivial fibrations by the characterization of the trivial fibrations,
  see \cite[after Thm.~2.11]{valery-olaf-smoothness-equivariant}). 

  These morphisms of dg categories induce by extension of scalars
  a commutative diagram
  \begin{equation}
    \label{eq:triangulated-tensor-product-and-Perf-of-algebras}
    \xymatrix{
      {\Perf(Q(\mathcal{A}) \otimes Q(\mathcal{B}))} \ar[rr]^-{(Q(Y) \otimes \id)^*}
      \ar[d] & &
      {\Perf(Q(\Perf(\mathcal{A})) \otimes Q(\mathcal{B}))} \ar[d] \\
      {\Perf(\mathcal{A} \otimes Q(\mathcal{B}))} \ar[rr]^-{(Y \otimes \id)^*} & &
      {\Perf(\Perf(\mathcal{A}) \otimes Q(\mathcal{B}))}
    }
  \end{equation}
  whose vertical arrows are quasi-equivalences
  (Lemma~\ref{l:dg-functors-restricts-to-perfects-embedding-and-qequiv}.\ref{enum:qequi-induces-qequi-between-perfs}).
  The lower horizontal arrow is a quasi-equivalence by
  Proposition~\ref{p:perf-to-perf-perf}.
  This implies that the upper
  horizontal arrow
  \begin{equation*}
    {\Perf(\mathcal{A} \otimes^\bL \mathcal{B})} \xra{(Y \otimes^\bL \id)^*}
    {\Perf(\Perf(\mathcal{A}) \otimes^\bL \mathcal{B})}
  \end{equation*}
  is a quasi-equivalence. 
  The same reasoning shows that
  the morphism $(\id \otimes^\bL Y)^*$ from the right-hand side
  to $\Perf((\Perf(\mathcal{A})) \otimes^\bL \Perf(\mathcal{B}))$
  is a quasi-equivalence.
  % \begin{equation*}
  %   {\Perf(\Perf(\mathcal{A}) \otimes^\bL \mathcal{B})}
  %   \xra{(\id \otimes^\bL Y)^*}
  %   {\Perf((\Perf(\mathcal{A})) \otimes^\bL \Perf(\mathcal{B}))}
  % \end{equation*}
  % is a quasi-equivalence. The composition of these two quasi-equivalences
  % is the quasi-equivalence
  % $(Y \otimes^\bL Y)^*$ which is the map in
  % \eqref{eq:perf-perf-perf-qe-perf}.
\end{proof}

\subsection{Proper, smooth, and saturated dg categories}
\label{sec:prop-smooth-satur}

\begin{definition}
  [{cf.\ \cite[Def.~2.4]{toen-vaquie-moduli-objects-dg-cats}, \cite[Def.~2.3]{toen-finitude-homotopique-propre-lisse}}]
  \label{d:loc-perf-cpt-gen-proper-smooth-(triang-)sat(-fin-type)}
  Let $\mathcal{A}$ be a dg category.
  \begin{enumerate}
  \item
    $\mathcal{A}$ is \textbf{locally ($\groundring$-)perfect}
    (or \textbf{locally ($\groundring$-)proper}) if $\mathcal{A}(A,A')$ is a
    perfect dg $\groundring$-module (i.\,e.\ in $\per(\groundring)$) for
    all $A, A' \in \mathcal{A}.$
  \item 
    $\mathcal{A}$ \textbf{has a compact generator} if the
    triangulated category 
    % $[\calMod(\mathcal{A})_{cf}]$ which is equivalent to
    $D(\mathcal{A})$ has a compact generator. An equivalent
    condition is that
    % $[\Perf(\mathcal{A})]\sira
    $\per(\mathcal{A})$ has a classical generator
    (use
    \cite[Thm.~2.1.2]{bondal-vdbergh-generators}).
    % and the fact
    % that
    % $D(\mathcal{A})$ is compactly generated and satisfies
    % $D(\mathcal{A})^c=\per(\mathcal{A})$, cf.\ the discussion
    % around equation (2.4) in 
    % \cite{valery-olaf-smoothness-equivariant}).
  \item
    $\mathcal{A}$ is \textbf{($\groundring$-)proper} if it is locally
    perfect and has 
    a compact generator.
  \item
    $\mathcal{A}$ is \textbf{($\groundring$-)smooth} if $\mathcal{A}$ considered
    as a dg
    $Q(\mathcal{A}) \otimes Q(\mathcal{A})^\opp$-module, is in
    $\per(Q(\mathcal{A}) \otimes Q(\mathcal{A})^\opp).$
  \item
    $\mathcal{A}$ is \textbf{($\groundring$-)saturated} if it is
    ($\groundring$-)proper, ($\groundring$-)smooth and 
    triangulated (see Def.~\ref{d:triangulated-dg-category}).
    % \item
    %   \textbf{of finite type}
  \end{enumerate}
\end{definition}

If $A$ is a dg algebra, then $\Yoneda{A}$ is a compact generator
of $D(A),$ hence $A$ has a compact generator.
Hence $A$ is proper if and only if it is locally perfect, i.\,e.\
if $A$
is a perfect dg $\groundring$-module.
The same statement is true for $\mathcal{A}$ a dg category with
finitely many isoclasses of objects in $[\mathcal{A}]:$ $\bigoplus_{A \in [\mathcal{A}]/\cong} \Yoneda{A}$ is a
compact generator of $D(\mathcal{A}).$

\begin{lemma}
  \label{l:properties-dg-K-cat-preserved-under-quasi-equivalence}
  The notions introduced in
  Definitions~\ref{d:loc-perf-cpt-gen-proper-smooth-(triang-)sat(-fin-type)}
  and \ref{d:triangulated-dg-category}
  are all invariant under quasi-equivalences.
\end{lemma}

\begin{proof}
  Let  $F \colon \mathcal{A} \ra \mathcal{B}$ be a quasi-equivalence.

  Locally perfect: If all $\mathcal{B}(B, B')$ are perfect dg
  $\groundring$-modules, the same is true for all $\mathcal{A}(A,A').$
  If all $\mathcal{A}(A,A')$ are in $\per(\groundring),$ then all
  $\mathcal{B}(F(A), F(A'))$ are in $\per(\groundring).$ In order to show that
  all $\mathcal{B}(B, B')$ are perfect use that $[F]$ is an
  equivalence.

  Has a compact generator: It is well-known
  (cf.\ proof of
  Lemma~\ref{l:dg-functors-restricts-to-perfects-embedding-and-qequiv}.\ref{enum:qequi-induces-qequi-between-perfs})
  that $F$ induces an equivalence $\bL F^* \colon D(\mathcal{A}) \ra
  D(\mathcal{B})$ of triangulated categories.

  Proper: Clear from above.

  Smooth: See
  \cite[Lemma~3.12]{valery-olaf-smoothness-equivariant}. 

  Triangulated:
  By
  Lemma~\ref{l:dg-functors-restricts-to-perfects-embedding-and-qequiv}
  we have a commutative diagram
  \begin{equation*}
    \xymatrix{
      {\mathcal{A}} \ar[d]^-{f} \ar[r]^-{Y} &
      {\Perf(\mathcal{A})} \ar[d]^-{f^*}\\
      {\mathcal{B}} \ar[r]^-{Y} &
      {\Perf(\mathcal{B})}
    }
  \end{equation*}
  whose vertical arrows are quasi-equivalences.
  Hence the upper horizontal arrow is a quasi-equivalence if and only
  if the lower horizontal arrow is a quasi-equivalence.

  Saturated: Clear from above.
\end{proof}

\begin{lemma}
  [{\cite[Lemma~2.6]{toen-vaquie-moduli-objects-dg-cats}}]
  \label{l:A-and-PerfA}
  Let $\mathcal{A}$ be a dg category.
  Then $\mathcal{A}$ is locally perfect (resp.\ has a compact
  generator resp.\ is proper resp.\ is smooth) if and only if
  $\Perf(\mathcal{A})$ has the corresponding property.
\end{lemma}

\begin{proof}
  Locally perfect: The Yoneda functor $Y \colon  \mathcal{A} \ra
  \Perf(\mathcal{A})$ is full and faithful. Hence $\mathcal{A}$
  is certainly locally perfect if $\Perf(\mathcal{A})$ is locally
  perfect. Conversely assume that $\mathcal{A}$ is locally
  perfect. It is easy to see that $\ol{\mathcal{A}}(A,A')$ is a
  perfect dg module for all $A, A' \in \ol{\mathcal{A}}.$ If
  $U$ is in $\Perf(\mathcal{A}),$ then there is an object $U' \in
  \Perf(\mathcal{A})$ and an object $X \in \ol{\mathcal{A}}$ such
  that $U\oplus U' \cong X$ in $[\Perf(\mathcal{A})].$ Let $Y \in
  \ol{\mathcal{A}}.$ Then $\Perf(\mathcal{A})(U,Y)$ is a direct
  summand of $\Perf(\mathcal{A})(U\oplus U', Y)$ which is in
  $D(\groundring)$ (even in $[\calMod(\groundring)]$) isomorphic to
  $\Perf(\mathcal{A})(X, Y)=\ol{\mathcal{\mathcal{A}}}(X,Y).$
  Hence $\Perf(\mathcal{A})(U,Y)$ is a perfect dg $\groundring$-module.
  Similarly we show that $\Perf(\mathcal{A})(U,V)$ is a perfect
  dg $\groundring$-module for $V$ in $\Perf(\mathcal{A}).$ This implies
  that $\Perf(\mathcal{A})$ is locally perfect.

  Let us prove the remaining claims.
  The Yoneda functor $Y \colon  \mathcal{A} \ra \Perf(\mathcal{A})$
  gives rise to the commutative diagram
  \begin{equation*}
    \xymatrix{
      {\mathcal{A}} \ar[d]^-{Y} \ar[r]^-{Y} &
      {\Perf(\mathcal{A})} \ar[d]^-{Y^*}\\
      {\Perf(\mathcal{A})} \ar[r]^-{Y} &
      {\Perf(\Perf(\mathcal{A})),}
    }
  \end{equation*}
  by
  Lemma~\ref{l:dg-functors-restricts-to-perfects-embedding-and-qequiv}.\ref{enum:restricts-to-Perf}. 
  Since $\Perf(\mathcal{A})$ is triangulated (Corollary~\ref{c:triang-versus-pretriang-and-karoubi}), the lower
  horizontal arrow is a quasi-equivalence.
  Note that $[\Perf(\mathcal{A})] \sira \per(\mathcal{A})$ is
  classicaly generated by the objects in the image of the Yoneda
  functor $[Y] \colon [\mathcal{A}] \ra [\Perf(\mathcal{A})].$ These
  statements 
  imply that $Y^*$ induces an equivalence on homotopy
  categories and hence is a quasi-equivalence by
  Lemma~\ref{l:dg-functor-qequi-test}.
  In particular 
  $\bL Y^*=(- \otimes_\mathcal{A}^\bL \Perf(\mathcal{A})) \colon 
  \per(\mathcal{A}) \ra 
  \per(\Perf(\mathcal{A}))$
  is an equivalence, and hence also
  %\begin{equation*}
    % \label{eq:scalar-extension-A-PerfA}
    $\bL Y^*= (- \otimes_\mathcal{A}^\bL \Perf(\mathcal{A})) \colon 
    D(\mathcal{A}) \ra 
    D(\Perf(\mathcal{A}))$
  %\end{equation*}
  (use \cite[4.2, Lemma]{Keller-deriving-dg-cat}).

  This immediately implies the claims concerning compact
  generators and properness, and also the claim concering
  smoothness (by 
  \cite[Thm.~3.17]{valery-olaf-smoothness-equivariant}).
\end{proof}

\subsection{Smoothness and properness of tensor products}
\label{sec:smoothness-properness-tensor-products}

We start with some observations.
Let $\mathcal{R}$ and $\mathcal{S}$ be dg categories, and assume
that $\mathcal{R}$ is $\groundring$-h-flat (i.\,e.\ all morphism
spaces $\mathcal{R}(R,R')$ are $\groundring$-h-flat).
Then the obvious dg bifunctor
% \begin{equation*}
$\calMod(\mathcal{R}) \times \calMod(\mathcal{S}) \ra
\calMod(\mathcal{R} 
\otimes \mathcal{S}),$
$(X,Y) \mapsto X \otimes Y,$
  % \end{align*}
induces the left derived functor
\begin{align}
  \label{eq:derived-tensor-over-K}
  D(\mathcal{R}) \times D(\mathcal{S}) & \ra D(\mathcal{R}
  \otimes \mathcal{S}),\\
  \notag  (X,Y) & \mapsto X\otimes^\bL Y := c(X) \otimes Y,
\end{align}
where $c(X) \ra X$ 
% and $c(Y) \ra Y$ are 
is a cofibrant resolution; note for this that
a cofibrant dg $\mathcal{R}$-module is a retract of a
semi-free dg $\mathcal{R}$-module 
\cite[Lemma~2.7]{valery-olaf-smoothness-equivariant}
and hence $\groundring$-h-flat
by our assumption on $\mathcal{R}.$
If $Y$ is $\groundring$-h-flat, then the obvious morphism 
$X\otimes^\bL Y \ra X \otimes Y$ is an isomorphism.
It is easy to see that the bifunctor
\eqref{eq:derived-tensor-over-K} induces a bifunctor
\begin{equation*}
  \per(\mathcal{R}) \times \per(\mathcal{S}) \ra
  \per(\mathcal{R} \otimes \mathcal{S}).
\end{equation*}
In particular (for $\mathcal{R}=\mathcal{S} =\groundring$), if
$X$ and $Y$ are perfect dg modules, then
$X\otimes^\bL Y \in \per(\groundring);$ 
if they are perfect and $Y$ is
$\groundring$-h-flat, then $X \otimes Y$ is in
$\per(\groundring).$

\begin{lemma}
  \label{l:tensor-product-of-smooth}
  Let $\mathcal{A}$ and $\mathcal{B}$ be smooth dg
  categories. Then
  $\mathcal{A} \otimes^\bL \mathcal{B}$ is smooth.
\end{lemma}

\begin{proof}
  Let $Q(\mathcal{A}) \ra \mathcal{A}$ and $Q(\mathcal{B}) \ra
  \mathcal{B}$ be cofibrant resolutions. Then
  $Q(\mathcal{A}) \in \per(Q(\mathcal{A}) \otimes
  Q(\mathcal{A})^\opp)$
  and
  $Q(\mathcal{B}) \in \per(Q(\mathcal{B}) \otimes
  Q(\mathcal{B})^\opp)$ by assumption.

  Note that
  both diagonal dg bimodules $Q(\mathcal{A})$ and
  $Q(\mathcal{B})$ are 
  $\groundring$-h-flat, 
  and that both dg categories
  $\mathcal{R}=Q(\mathcal{A}) \otimes Q(\mathcal{A})^\opp$ and
  $\mathcal{S}=Q(\mathcal{B}) \otimes Q(\mathcal{B})^\opp$
  are $\groundring$-h-flat, 
  by
  \cite[Lemma~2.14]{valery-olaf-smoothness-equivariant}.
  Then, using the obvious isomorphism
  % \begin{equation*}
  $\mathcal{R} \otimes \mathcal{S}
  \sira
  (Q(\mathcal{A}) \otimes Q(\mathcal{B})) \otimes
  (Q(\mathcal{A}) \otimes Q(\mathcal{B}))^\opp,$
  % \end{equation*}
  the above discussion shows that
  \begin{equation*}
    Q(\mathcal{A})\otimes Q(\mathcal{B}) \in
    \per(
    (Q(\mathcal{A}) \otimes Q(\mathcal{B})) \otimes
    (Q(\mathcal{A}) \otimes Q(\mathcal{B}))^\opp),
  \end{equation*}
  and this dg bimodule is the diagonal bimodule. Since
  $Q(\mathcal{A})\otimes Q(\mathcal{B})$ is $\groundring$-h-flat
  this implies that $\mathcal{A} \otimes^\bL \mathcal{B}=
  Q(\mathcal{A})\otimes Q(\mathcal{B})$ is smooth
  (since smoothness can be checked using a $\groundring$-h-flat
  resolution, by
  \cite[Lemma~3.6]{valery-olaf-smoothness-equivariant}).
\end{proof}

\begin{lemma}
  \label{l:tensor-product-of-locally-proper}
  Let $\mathcal{A}$ and $\mathcal{B}$ be locally perfect dg
  categories.  Then $\mathcal{A} \otimes^\bL \mathcal{B}$ is a
  locally perfect dg category.  In particular, if $A$ and $B$
  are proper dg algebras, then $A \otimes^\bL B$ is
  a proper dg algebra.
\end{lemma}

\begin{proof}
  If $Q(\mathcal{A}) \ra \mathcal{A}$ and $Q(\mathcal{B}) \ra
  \mathcal{B}$ are cofibrant resolutions, both $Q(\mathcal{A})$
  and $Q(\mathcal{B})$ are locally perfect
  (Lemma~\ref{l:properties-dg-K-cat-preserved-under-quasi-equivalence}).
  Since both $Q(\mathcal{A})$ and $Q(\mathcal{B})$ are
  $\groundring$-h-flat
  (\cite[Lemma~2.14]{valery-olaf-smoothness-equivariant}), the
  above discussion shows that $\mathcal{A} \otimes^\bL
  \mathcal{B}= Q(\mathcal{A})\otimes Q(\mathcal{B})$ is locally
  perfect.  The last claim is immediate since a dg algebra is
  proper if and only if it is locally perfect.
\end{proof}

\subsection{Back to saturated dg categories}
\label{sec:back-saturated-dg}

Recall
from
Definition~\ref{d:loc-perf-cpt-gen-proper-smooth-(triang-)sat(-fin-type)}
that a dg category $T$ is saturated if it is
triangulated, smooth and proper.

\begin{proposition}
  [{cf.~\cite[Prop.~13]{toen-lectures-dg-cats}, and
    \cite[Lemma~2.6]{toen-vaquie-moduli-objects-dg-cats}}]
  \label{p:saturated-easy-characterization}
  A dg category $\mathcal{T}$ has a compact generator if and
  only if there is a
  dg algebra $A$ such that $\Perf(\mathcal{T})$ and
  $\Perf(A)$ are isomorphic in $\Heq_\groundring.$
  If such an $A$ is given, then $\Perf(\mathcal{T})$ is smooth (resp.\
  proper) 
  if and only if $A$ is smooth (resp.~proper).
  
  In particular, a dg category $\mathcal{T}$ is saturated if and only if
  there is a smooth and proper dg algebra $A$ such that $\mathcal{T}$ and
  $\Perf(A)$ are isomorphic in
  $\Heq_\groundring.$
\end{proposition}

\begin{proof}
  If $A$ is a 
  dg algebra, then 
  $\Yoneda{A}$ is a classical generator of $[\Perf(A)] \sira
  \per(A).$ If $\Perf(\mathcal{T})$ and $\Perf(A)$ are isomorphic
  in $\Heq_\groundring$ then $[\Perf(\mathcal{T})]\cong
  [\Perf(A)],$ so $\mathcal{T}$ has a compact generator.

  Conversely, assume that $\mathcal{T}$ has a compact generator.
  Let $E \in \Perf(\mathcal{T})$ be such that $E$ is a classical 
  generator of $[\Perf(\mathcal{T})] \sira \per(\mathcal{T}).$
  Let $A:=(\Perf(\mathcal{T}))(E,E).$ We consider the dg algebra $A$ also
  as a 
  dg category with one object $\star.$ The obvious inclusion $i \colon A
  \ra \Perf(\mathcal{T}),$ $\star \mapsto E,$ gives
  by Lemma~\ref{l:dg-functors-restricts-to-perfects-embedding-and-qequiv}.\ref{enum:embeddings}
  rise to the commutative
  diagram
  \begin{equation*}
    \xymatrix{
      {A} \ar[d]^-{i} \ar[r]^-{Y} &
      {\Perf(A)} \ar[d]^-{i^*}\\
      {\Perf(\mathcal{T})} \ar[r]^-{Y} &
      {\Perf(\Perf(\mathcal{T}))}
    }
  \end{equation*}
  whose vertical arrows are full and faithful.
  The lower horizontal arrow is a quasi-isomorphism
  (Corollary~\ref{c:triang-versus-pretriang-and-karoubi});
  in particular, it induces a triangulated equivalence
  \begin{equation*}
    % \thick(\mathcal{T})=
    [\Perf(\mathcal{T})]
    \xsira{[Y]}
    % \thick(\Perf(\mathcal{T})) =
    [\Perf(\Perf(\mathcal{T}))].
  \end{equation*}
  This implies that $[\Perf(\Perf(\mathcal{T}))]$ is the thick envelope of
  $\Yoneda{E}=Y(E).$
  Note
  that $[\Perf(A)]=\thick(A)$ is the thick envelope of
  $\Yoneda{\star}$
  and
  that
  $i^*(\Yoneda{\star})=i^*(Y(\star)) =Y(i(\star))=\Yoneda{E}.$
  Since $[\Perf(A)]$ is Karoubian this implies that
  $[i^*] \colon [\Perf(A)] \ra [\Perf(\Perf(\mathcal{T}))]$ is a triangulated
  equivalence.
  Then Lemma~\ref{l:dg-functor-qequi-test} shows that
  the vertical arrow $i^*$ in the above commutative square is a
  quasi-equivalence. This shows that $\Perf(A)$ and $\Perf(\mathcal{T})$ are
  connected by a zig-zag of quasi-equivalences.

  Lemmata~\ref{l:properties-dg-K-cat-preserved-under-quasi-equivalence}
  and
  \ref{l:A-and-PerfA}
  then yield the second claim, and for the last claim 
  use additionally Corollary~\ref{c:triang-versus-pretriang-and-karoubi}.
\end{proof}

\begin{proposition}
  \label{p:triang-tensor-product-of-saturated-dg-K-cats}
  Let $\mathcal{S},$ $\mathcal{T}$ be saturated dg categories.
  Then $\mathcal{S} \Perfotimes \mathcal{T}$
  (defined in \eqref{eq:def-triang-tensor-product-of-dg-cats})
  is a saturated dg category.
\end{proposition}

Note that Lemmata~ \ref{l:A-and-PerfA},
\ref{l:tensor-product-of-smooth}, 
\ref{l:tensor-product-of-locally-proper}
and
Corollary~\ref{c:triang-versus-pretriang-and-karoubi}
show that $\mathcal{S} \Perfotimes \mathcal{T}$
is locally perfect, smooth, and triangulated.

\begin{proof}
  By Proposition~\ref{p:saturated-easy-characterization} there are
  smooth and proper dg algebras $A$ and $B$ such that
  $\mathcal{S} \cong \Perf(A)$ and $\mathcal{T} \cong \Perf(B)$
  in $\Heq_\groundring.$ 
  Then we have isomorphisms
  \begin{equation*}
    \mathcal{S} \Perfotimes \mathcal{T} \cong (\Perf(A)) \Perfotimes
    (\Perf(B)) \sila \Perf(A\otimes^\bL B)=
    A \Perfotimes B
  \end{equation*}
  in $\Heq_\groundring$
  by
  Lemma~\ref{l:Perfotimes-respects-qequivalences}
  and Proposition~\ref{p:triangulated-tensor-product-of-cats-and-their-Perfs}.
  Lemmata~\ref{l:tensor-product-of-smooth} and
  \ref{l:tensor-product-of-locally-proper}
  show that $A\otimes^\bL B$ is smooth and proper,
  so Proposition~\ref{p:saturated-easy-characterization}
  again proves the claim.
\end{proof}

For later use we include the following result which is 
similar to Proposition~\ref{p:saturated-easy-characterization}.
 
\begin{proposition}
  \label{p:pretriang-classical-generator}
  Let $\mathcal{E}$ be a pretriangulated dg category and let 
  $E \in \mathcal{E}$ be an object that becomes a classical 
  generator of $[\mathcal{E}].$ Consider the dg
  algebra 
  $A:=\mathcal{E}(E,E).$ Then there is a
  quasi-equivalence
  % \begin{equation*}
  $\Perf(A) \ra \Perf(\mathcal{E})$
  % \end{equation*}
  of dg categories, and the dg functor $\mathcal{E}(E,-) \colon 
  \mathcal{E} \ra \calMod(A)$ induces a full and faithful
  triangulated functor $\mathcal{E}(E,-) \colon  [\mathcal{E}] \ra
  \per(A)$ that extends to an equivalence between the Karoubi
  envelope of $[\mathcal{E}]$ and $\per(A).$

  Moreover, $A$ is smooth (resp.\ proper) 
  if and only if $\mathcal{E}$ has this property 
  if and only if $\Perf(\mathcal{E})$ has this property.
  In particular, $A$ is smooth and proper if and only if
  $\Perf(\mathcal{E})$ is saturated.
\end{proposition}

\begin{proof}
  % Let $A$ be a dg ($\groundring$-)algebra. 
  % We denote the dg
  % category with one object $\bullet$ whose endomorphisms are $A$
  % by $A_\bullet.$ 
  % We sometimes write $\Perf(A)$ instead of
  % $\Perf(A_\bullet).$ 
  We consider $A$ as a dg category. Mapping its unique object to
  $E$ defines a dg functor $i \colon  A \ra \mathcal{E}.$
  Lemma~\ref{l:dg-functors-restricts-to-perfects-embedding-and-qequiv}.\ref{enum:embeddings}
  shows that the induced extension of scalars functor 
  $i^* \colon  \Perf(A) \ra \Perf(\mathcal{E})$ is full and faithful.
  It maps $A$ to $\Yoneda{E}.$
  The induced functor $i^* \colon  \thick(A) \ra \thick(\mathcal{E})$ on
  homotopy categories is full and faithful, and moreover
  essentially surjective since $\Yoneda{E}$ is a classical
  generator of $\thick(\mathcal{E})$ and $\thick(A)$ is
  idempotent complete.
  Now Lemma~\ref{l:dg-functor-qequi-test} shows that
  $i^* \colon  \Perf(A) \ra \Perf(\mathcal{E})$ is a quasi-equivalence.

  % This also implies that $(- \otimes_A^\bL \mathcal{E}) \colon  D(A) \ra
  % D(\mathcal{E})$ (cf.\
  % \cite[2.12]{lunts-categorical-resolution})
  % is an equivalence. Then
  % \cite[Thm.~3.17]{valery-olaf-smoothness-equivariant} 
  % shows that $A$ is $K$-smooth if and only $\mathcal{E}$ is
  % $K$-smooth. 

  A quasi-inverse of 
  $i^* \colon  \thick(A) \ra \thick(\mathcal{E})$
  is given by restriction along $i.$ 
  This restriction composed with $[\mathcal{E}] \ra
  \thick(\mathcal{E})$ is given by $\mathcal{E}(E,-).$
  Moreover, $\thick(\mathcal{E})$ is the
  Karoubi envelope of 
  $[\ol{\mathcal{E}}],$ $[\mathcal{E}]
  \ra 
  [\ol{\mathcal{E}}]$ is an equivalence since $\mathcal{E}$ is
  pretriangulated, and
  $\thick(A) \sira
  \per(A).$
  
  The remaining claims follow from the
  Lemmata~\ref{l:A-and-PerfA}
  and
  \ref{l:properties-dg-K-cat-preserved-under-quasi-equivalence}
  and
  Corollary~\ref{c:triang-versus-pretriang-and-karoubi}.
\end{proof}

\subsection{Semi-orthogonal decompositions}
\label{sec:semi-orth-decomp}

We refer the reader to
% \cite[App.~A]{valery-olaf-matfak-semi-orth-decomp} 
% Appendix A in 
\cite[appendix~\ref{semi:sec:app:remind-admiss-subc}]{valery-olaf-matfak-semi-orth-decomp}
for the
definition and elementary properties of semi-orthogonal
decompositions.

The first part of the following result
says that a
semi-orthogonal decomposition  
of the homotopy category $[\mathcal{T}]$ of a pretriangulated dg
category $\mathcal{T}$ induces a semi-orthogonal decomposition of
$[\Perf(\mathcal{T})].$ This may be viewed as a dg lift of 
\cite[Cor.~\ref{semi:c:semi-od-and-karoubi-envelope}]{valery-olaf-matfak-semi-orth-decomp}.
Its formulation is a bit technical since
the components of a semi-orthogonal decomposition are required to
be strict subcategories.  
A related result is given in
Lemma~\ref{l:directed-induces-semi-od-on-Perf} below.
 
\begin{proposition}
  \label{p:semi-od-goes-to-Perf}
  Assume that $\mathcal{T}$ is a pretriangulated dg category with
  full dg subcategories $\mathcal{U}$ and $\mathcal{V}$ such that
  $[\mathcal{T}]=\langle [\mathcal{U}],[\mathcal{V}]\rangle$ is a
  semi-orthogonal decomposition (resp.\ a semi-orthogonal
  decomposition into admissible subcategories).
  
  Then there is an induced semi-orthogonal decomposition
  $[\Perf(\mathcal{T})]=\langle
  [\Perf(\mathcal{U})'],[\Perf(\mathcal{V})']\rangle$ 
  (into admissible subcategories). Here $\Perf(\mathcal{U})'$ is
  the full dg subcategory of $\Perf(\mathcal{T})$ such that 
  $[\Perf(\mathcal{U})']$ is the strict closure of
  $[\Perf(\mathcal{U})]$ in $[\Perf(\mathcal{T})].$ In
  particular, there is an obvious quasi-equivalence 
  $\Perf(\mathcal{U}) \ra \Perf(\mathcal{U})'.$
  The dg category $\Perf(\mathcal{V})'$ is defined similarly. 

  More generally, let $\mathcal{R}$ be a $\groundring$-h-flat dg
  category. Then there is a semi-orthogonal decomposition
  $[\Perf(\mathcal{R} \otimes \mathcal{T})]=\langle
  [\Perf(\mathcal{R} \otimes \mathcal{U})'],[\Perf(\mathcal{R}
  \otimes \mathcal{V})']\rangle$ 
  (into admissible subcategories) where the involved dg
  subcategories are defined in the obvious way.
\end{proposition}

\begin{proof}
  The first claim is the special case $\mathcal{R}=\groundring$
  of the second claim which we prove now.
  Assume that
  $[\mathcal{T}]=\langle [\mathcal{U}],[\mathcal{V}]\rangle$ is a
  semi-orthogonal decomposition.
  The inclusions $\mathcal{U} \subset \mathcal{T}$ and $\mathcal{V} \subset \mathcal{T}$ give rise to full and
  faithful dg functors
  $\mathcal{R} \otimes \mathcal{U} \ra \mathcal{R} \otimes
    \mathcal{T},$
  $\mathcal{R} \otimes \mathcal{V} \ra \mathcal{R} \otimes
    \mathcal{T}.$
  % \begin{align*}
  %   \mathcal{R} \otimes \mathcal{U} & \ra \mathcal{R} \otimes
  %   \mathcal{T},\\ 
  %   \mathcal{R} \otimes \mathcal{V} & \ra \mathcal{R} \otimes
  %   \mathcal{T}.
  % \end{align*}
  Lemma~\ref{l:dg-functors-restricts-to-perfects-embedding-and-qequiv}.\ref{enum:embeddings}
  shows that the induced dg functors
  \begin{align*}
    \tildew{\mathcal{U}} := \Perf(\mathcal{R} \otimes \mathcal{U}) & \ra \tildew{\mathcal{T}} :=\Perf(\mathcal{R} \otimes \mathcal{T}),\\
    \tildew{\mathcal{V}} := \Perf(\mathcal{R} \otimes \mathcal{V}) & \ra \tildew{\mathcal{T}} = \Perf(\mathcal{R} \otimes \mathcal{T})
  \end{align*}
  are full and faithful. We view $\tildew{\mathcal{U}}$ and $\tildew{\mathcal{V}}$ as full dg
  subcategories of $\tildew{\mathcal{T}}.$
  From 
  $[\mathcal{T}]([\mathcal{V}],[\mathcal{U}])=0$
  we see that $\mathcal{T}(v,u)$ is acyclic for all
  $v \in \mathcal{V}$ and $u \in \mathcal{U}.$

  Let $r,r' \in \mathcal{R},$ $u \in \mathcal{U}$ and $v \in
  \mathcal{V}.$ Since $\mathcal{R}(r,r')$ is $\groundring$-h-flat,
  % ($\mathcal{R}=Q(\mathcal{S})$ being cofibrant),
  $\mathcal{R}(r,r')\otimes \mathcal{T}(v,u)$ is acyclic. This
  implies that
  \begin{equation*}
    [\tildew{\mathcal{T}}]\big(\Yoneda{(r,v)}, \Yoneda{(r',u)}\big)=
    [\mathcal{R} \otimes \mathcal{T}]((r,v),(r',u))= H^0(\mathcal{R}(r,r')\otimes \mathcal{T}(v,u))=0.
  \end{equation*}
  Since $[\tildew{\mathcal{U}}]=\thick(\mathcal{R}\otimes \mathcal{U})$ is classically generated
  by the objects $\Yoneda{(r,u)},$ for $r \in \mathcal{R}$ and $u \in \mathcal{U},$ and
  similarly for $[\tildew{\mathcal{V}}],$ we see that
%  \begin{equation*}
  $[\tildew{\mathcal{T}}]([\tildew{\mathcal{V}}], [\tildew{\mathcal{U}}]) = 0.$
%  \end{equation*}

  Let $r \in \mathcal{R}$ and $t \in \mathcal{T}.$ There are $v \in \mathcal{V}$ and $u \in \mathcal{U}$ such
  that there is a triangle
  % \begin{equation*}
  $v \ra t \ra u \ra [1]v$
  % \end{equation*}
  in $[\mathcal{T}].$ Consider the dg functor $i_r \colon  \mathcal{T} \ra \mathcal{R}
  \otimes \mathcal{T},$ $t' \mapsto (r,t').$ 
  It induces a
  commutative diagram
  \begin{equation*}
    \xymatrix{
      {\mathcal{T}} \ar[d]^-{i_r} \iar[r]^-{Y} &
      {\Perf(\mathcal{T})} \ar[d]^-{i_r^*}\\
      {\mathcal{R} \otimes \mathcal{T}} \iar[r]^-{Y} &
      {\tildew{\mathcal{T}}=\Perf(\mathcal{R}\otimes \mathcal{T})}
    }
  \end{equation*}
  (see Lemma~\ref{l:dg-functors-restricts-to-perfects-embedding-and-qequiv}.\ref{enum:restricts-to-Perf}). If we pass to homotopy categories,
  the upper horizontal and the right vertical functor are triangulated
  functors; they map the above triangle to the triangle
  \begin{equation*}
    \Yoneda{(r,v)} \ra \Yoneda{(r,t)} \ra \Yoneda{(r,u)} \ra [1]\Yoneda{(r,v)}
  \end{equation*}
  in $[\tildew{\mathcal{T}}].$
  Note that $[\tildew{\mathcal{T}}]$ is classically generated by
  the objects $\Yoneda{(r,t)},$ for $r \in \mathcal{R}$ and $t \in \mathcal{T},$ and that
  both $[\tildew{\mathcal{V}}]$ and $[\tildew{\mathcal{U}}]$ are Karoubian
  % (by Corollary~\ref{c:triang-versus-pretriang-and-karoubi}
  % and Lemma~\ref{l:triang-versus-pretriang-and-karoubi}).
  (by Lemma~\ref{l:triang-versus-pretriang-and-karoubi} and
  Corollary~\ref{c:triang-versus-pretriang-and-karoubi}).  Define
  $\tildew{\mathcal{V}}'$ to be the full dg subcategory of
  $\tildew{\mathcal{T}}$ such that $[\tildew{\mathcal{V}}']$ is
  the closure under isomorphisms of $[\tildew{\mathcal{V}}]$ in
  $[\tildew{\mathcal{T}}];$ define $\tildew{\mathcal{U}}'$
  similarly.  (Note that $\tildew{\mathcal{V}} \ra
  \tildew{\mathcal{V}}'$ is a quasi-equivalence by
  Lemma~\ref{l:dg-functor-qequi-test}.)
  % Lemma~\ref{l:right-admissible-by-classical-generators}
  
  From
  \cite[Lemma~\ref{semi:l:right-admissible-by-triangulated-and-classical-generators}.\ref{semi:enum:T-is-thick-envelope-of-E}]{valery-olaf-matfak-semi-orth-decomp}
  we see that 
  $[\tildew{\mathcal{T}}]=\langle [\tildew{U}'], [\tildew{V}']
  \rangle$ 
  is a semi-orthogonal decomposition.
  In particular, 
  $[\tildew{\mathcal{V}}']$
  is
  right admissible and
  $[\tildew{\mathcal{U}}']$ is left admissible in
  $[\tildew{\mathcal{T}}].$

  Assume now in addition that $[\mathcal{U}]$ 
  % and $[\mathcal{V}]$
  is right admissible in $[\mathcal{T}].$
  Then \cite[Lemma~\ref{semi:l:first-properties-semi-orthog-decomp}.\ref{semi:enum:admissible-and-semi-orth}]{valery-olaf-matfak-semi-orth-decomp}
  says that $[\mathcal{T}]=\langle [\mathcal{U}]^\perp,
  [\mathcal{U}] \rangle$ is a semi-orthogonal decomposition. 
  Let $\mathcal{U}^\dag$ be the full dg subcategory of $\mathcal{T}$ that has the same
  objects as $[\mathcal{U}]^\perp,$ so $[\mathcal{U}^\dag]=[\mathcal{U}]^\perp.$ 
  Then the above argument shows that $[\tildew{\mathcal{U}}']$ is
  right admissible.
  % We need to prove that $[\tildew{\mathcal{U}}']$ is right admissible.
  % Let $\mathcal{U}^\dag$ be the full dg subcategory of $\mathcal{T}$ that has the same
  % objects as $[\mathcal{U}]^\perp.$ 
  % % Note that $\mathcal{U}^\dag$ is saturated.
  % Then $\mathcal{T}(u,u^\dag)$ is acyclic for all $u \in
  % \mathcal{U}$ and $u^\perp \in \mathcal{U}^\perp.$ Let
  % $\tildew{\mathcal{U}^\dag}:=\Perf(\mathcal{R} \otimes
  % \mathcal{U}^\dag).$ As above we deduce that
  % $[\tildew{\mathcal{T}}]([\tildew{\mathcal{U}}],
  % [\tildew{\mathcal{U}^\dag}])=0,$ and that any object of the
  % form $\hatw{(r,t)}$ for $r \in \mathcal{R}$ and $t \in
  % \mathcal{T}$ fits into a triangle
  % % \begin{equation*}
  % $\Yoneda{(r,u)} \ra \Yoneda{(r,t)} \ra \Yoneda{(r,u^\dag)} \ra
  % [1]\Yoneda{(r,u)}$ 
  % % \end{equation*} 
  % in $[\tildew{\mathcal{T}}]$ with $u \in \mathcal{U}$ and
  % $u^\dag \in \mathcal{U}^\dag.$ 
  % % Lemma~\ref{l:right-admissible-by-classical-generators}
  % Then
  % \cite[Lemma~\ref{semi:l:right-admissible-by-triangulated-and-classical-generators}.\ref{semi:enum:T-is-thick-envelope-of-E}]{valery-olaf-matfak-semi-orth-decomp}
  % again shows that $[\tildew{\mathcal{U}}']$ is right admissible
  % in $[\tildew{\mathcal{T}}].$
  Similarly, left admissibility of $[\mathcal{V}]$ implies that
  $[\tildew{\mathcal{V}}']$ is left admissible.
\end{proof}

\begin{proposition}
  \label{p:semi-od-heredity}
  Let $\mathcal{T}$ be a pretriangulated dg category with full dg
  subcategories $\mathcal{U}$ and $\mathcal{V}$ such that
  $[\mathcal{T}]=\langle [\mathcal{U}],[\mathcal{V}]\rangle$ is a
  semi-orthogonal decomposition.  Then $\mathcal{U}$ and
  $\mathcal{V}$ are pretriangulated as well. Moreover, if
  $\mathcal{T}$ is triangulated (resp.\ is locally perfect resp.\
  has a compact generator resp.\ is smooth resp.\ is proper
  resp.\ is saturated) then $\mathcal{U}$ and
  $\mathcal{V}$ have 
  the same property.
\end{proposition}

\begin{proof}
  It is clear that $\mathcal{U}$ and $\mathcal{V}$ are
  pretriangulated. 

  Triangulated: 
  If $[\mathcal{T}]$ is Karoubian, so are $[\mathcal{U}]$ and
  $[\mathcal{V}]$ since $[\mathcal{U}]=[\mathcal{V}]^\perp$ and
  $[\mathcal{V}]=\leftidx{^\perp}{[\mathcal{U}]}{},$ and we can
  apply Lemma~\ref{l:triang-versus-pretriang-and-karoubi}.

  Locally perfect: this is obviously passed to any full dg subcategory.

  Has a compact generator. 
  The first claim of
  Proposition~\ref{p:semi-od-goes-to-Perf} 
  (together with
  Lemmata~\ref{l:properties-dg-K-cat-preserved-under-quasi-equivalence}
  and \ref{l:A-and-PerfA})
  shows that we can assume that $\mathcal{T}$ is
  triangulated.
  Then by
  assumption $[\mathcal{T}] \sira [\Perf(\mathcal{T})] \sira
  \per(\mathcal{T})$ has a classical generator.  The obvious
  functors $[\mathcal{U}] \ra [\mathcal{T}]/[\mathcal{V}]$ and
  $[\mathcal{V}] \ra [\mathcal{T}]/[\mathcal{U}]$ are
  equivalences of triangulated categories. This implies that
  $[\mathcal{U}] \sira \per(\mathcal{U})$ and $[\mathcal{V}]
  \sira \per(V)$ have classical generators, i.\,e.\ $\mathcal{U}$
  and $\mathcal{V}$ have compact generators.

  Smooth:
  Let $\mathcal{E} \subset \mathcal{T}$ be the full dg
  subcategory of $\mathcal{T}$ whose objects are the union of the
  objects of $\mathcal{U}$ and $\mathcal{V}.$ Let $\mathcal{V}'$
  be the full dg subcategory of $\mathcal{V}$ obtained by
  ignoring all objects that are also in $\mathcal{U}.$
  Let $\mathcal{E}' \subset \mathcal{E}$ be the (in general
  non-full) dg subcategory with the same objects and morphism
  spaces as $\mathcal{E}$ except that we set
  $\mathcal{E}'(V',U)=0$ for all $V' \in
  \mathcal{V}'$ and $U \in \mathcal{U}.$ 
  Then $\mathcal{E}' \ra \mathcal{E}$ is a quasi-equivalence
  since all $\mathcal{T}(V',U)$ are acyclic.
  Symbolically, this inclusion may be written as
  $
  \mathcal{E}'
  =
  \Big[\begin{smallmatrix}
    \mathcal{U} & 0 \\
    \leftidx{_{\mathcal{V}'}}{\mathcal{T}}{_\mathcal{U}} &
    \mathcal{V}'
  \end{smallmatrix}\Big] \subset
  \mathcal{E}
  =
  \Big[\begin{smallmatrix}
    \mathcal{U} &
    \leftidx{_\mathcal{U}}{\mathcal{T}}{_{\mathcal{V}'}} \\ 
    \leftidx{_{\mathcal{V}'}}{\mathcal{T}}{_\mathcal{U}} &
    \mathcal{V}'
  \end{smallmatrix}\Big].$
  Lemma~\ref{l:dg-functors-restricts-to-perfects-embedding-and-qequiv} 
  implies that $\Perf(\mathcal{E}') \ra \Perf(\mathcal{E})$ is a
  quasi-equivalence and that $\Perf(\mathcal{E}) \ra
  \Perf(\mathcal{T})$ is full and faithful; but in fact this last
  arrow is
  also a quasi-equivalence: on homotopy categories it induces an
  equivalence since each object of $[\Perf(\mathcal{T})]$ is an
  extension of an object of $[\Perf(\mathcal{U})]$ by an object
  of $[\Perf(\mathcal{V})])$ (by
  Proposition~\ref{p:semi-od-goes-to-Perf}),
  so we can use Lemma~\ref{l:dg-functor-qequi-test}.
  Hence smoothness of $\mathcal{T}$ implies smoothness of
  $\mathcal{E}'$ (using Lemmata~\ref{l:properties-dg-K-cat-preserved-under-quasi-equivalence}
  and \ref{l:A-and-PerfA} again), and then
  \cite[Thm.~3.24]{valery-olaf-smoothness-equivariant}
  implies smoothness of both $\mathcal{U}$ and $\mathcal{V}.$
  % As above we can assume that $\mathcal{T}$ is
  % triangulated. We also know that $[\mathcal{T}] \sira
  % \per(\mathcal{T})$ has a classical generator $G$ by
  % Lemma~\ref{l:smooth-implies-cpt-generator} below.
  % Let $V \ra G \ra U \ra [1]V$ be a triangle in $[\mathcal{T}].$
  % Then $U \oplus V$ (resp.\ $V$ resp.\ $U$) is a classical
  % generator of $[\mathcal{T}]$ (resp.\ $[\mathcal{V}]$ resp.\
  % $[\mathcal{U}]$).  
  % Let $E=\mathcal{T}(U \oplus V, U \oplus V).$ 
  % The proof of Proposition~\ref{p:saturated-easy-characterization}
  % shows that $[\mathcal{T}] \sira [\Perf(\mathcal{T})]$ is smooth
  % if and only if $E$ is smooth. 
  % Since $\mathcal{T}(V,U)$ is acyclic, the inclusion
  % $\Big[\begin{smallmatrix}
  %   \mathcal{T}(U,U) & 0 \\
  %   \mathcal{T}(U,V) & \mathcal{T}(V,V)
  % \end{smallmatrix}\Big] \subset E$ is a quasi-isomorphism.
  % Hence smoothness of $E$ implies smoothness of 
  % $\mathcal{T}(U,U)$ and $\mathcal{T}(V,V)$
  % (by \cite[Thm.~3.24]{valery-olaf-smoothness-equivariant})
  % and hence smoothness of $\mathcal{U}$ and $\mathcal{V}.$

  Proper, saturated: Clear from above.
\end{proof}

\begin{corollary}
  \label{c:semi-od-enhancement}
  Let $\mathcal{D}$ be a triangulated category with a
  semi-orthogonal decomposition $\mathcal{D}= \langle
  \mathcal{A}, \mathcal{B}\rangle.$
  Then the following
  properties of an enhancement of $\mathcal{D}$ are passed on to
  the induced 
  enhancements of $\mathcal{A}$ and $\mathcal{B}$:
  being triangulated, being locally perfect,
  having a compact generator, smoothness, properness,
  being saturated.
\end{corollary}

\subsection{Grothendieck ring of saturated dg categories}
\label{sec:groth-ring-satur}

Let $\sat_\groundring$ be the full subcategory of
$\dgcat_\groundring$ consisting of saturated dg categories.
By $\ul{\sat_\groundring}$ 
we
denote the set of isomorphism 
classes in $\Heq_\groundring$ of these
categories
(cf.\
Lemma~\ref{l:properties-dg-K-cat-preserved-under-quasi-equivalence}).
Given a saturated dg category $\mathcal{T},$ we write $\ul{\mathcal{T}}$ for its class
in $\ul{\sat_\groundring}.$

\begin{proposition}
  \label{p:ulsat-as-a-commutative-monoid}
  The map $(\mathcal{S},\mathcal{T}) \mapsto \mathcal{S} \Perfotimes \mathcal{T}$ induces a
  multiplication $\bullet$ on $\ul{\sat_\groundring}$ that
  turns $\ul{\sat_\groundring}$ into a commutative monoid
  with unit $\ul{\Perf(\groundring)}.$
\end{proposition}

\begin{proof}
  Lemma~\ref{l:Perfotimes-respects-qequivalences} 
  and Proposition~\ref{p:triang-tensor-product-of-saturated-dg-K-cats}
  show that
  $\bullet$ is well defined. 
  Let $\mathcal{S}_1,$ $\mathcal{S}_2,$ and $\mathcal{S}_3$ be
  saturated dg categories. We know that 
  $\mathcal{S}_i \ra \Perf(\mathcal{S}_i)$ is a
  quasi-equivalence.
  Hence to obtain associativity of $\bullet$ it is enough to
  prove that $(\Perf(\mathcal{S}_1) \Perfotimes
  \Perf(\mathcal{S}_2)) \Perfotimes  
  \Perf(\mathcal{S}_3)$ and
  $(\Perf(\mathcal{S}_1) \Perfotimes
  \Perf(\mathcal{S}_2)) \Perfotimes  
  \Perf(\mathcal{S}_3)$ are isomorphic  
  in $\Heq_\groundring.$
  Proposition~\ref{p:triangulated-tensor-product-of-cats-and-their-Perfs}
  and Lemma~\ref{l:dg-functors-restricts-to-perfects-embedding-and-qequiv}.\ref{enum:qequi-induces-qequi-between-perfs}
  reduce this to showing that
  $(\mathcal{S}_1 \otimes^\bL
  \mathcal{S}_2) \otimes^\bL  
  \mathcal{S}_3$
  and $\mathcal{S}_1 \otimes^\bL
  (\mathcal{S}_2 \otimes^\bL  
  \mathcal{S}_3)$
  are isomorphic  
  in $\Heq_\groundring.$
  But this is easy to see since
  cofibrant dg
  categories are $\groundring$-h-flat.
  Similarly, commutativity follows from
  $\mathcal{S}_1 \otimes^\bL \mathcal{S}_2
  \cong \mathcal{S}_2 \otimes^\bL \mathcal{S}_1,$ 
  and $\mathcal{S}_1 \otimes^\bL \groundring
  \cong \mathcal{S}_1$ proves that $\ul{\Perf(\groundring)}$ is
  the unit.
\end{proof}

Denote by $\DZ\ul{\sat_\groundring}$ the (commutative
associative unital) monoid ring of
$\ul{\sat_\groundring},$ i.\,e.\ the free abelian group
on $\ul{\sat_\groundring}$ with $\DZ$-bilinear
multiplication induced by $\bullet.$

\begin{definition}
  [{cf.\ \cite[Def.~5.1,
    8.1]{bondal-larsen-lunts-grothendieck-ring}}]
  \label{def:Grothendieck-group-of-saturated-dg-cats}
  The \define{Grothendieck group $K_0(\sat_\groundring)$ of
    saturated dg categories} 
  is defined to be the quotient of
  $\DZ\ul{\sat_\groundring}$ by the subgroup generated by
  the elements (the "semi-ortogonal relations")
  % \begin{equation*}
  %   \label{eq:semi-orthog-relation}
  $\ul{\mathcal{T}} - (\ul{\mathcal{U}} + \ul{\mathcal{V}})$
  % \end{equation*}
  whenever there is a saturated dg category $\mathcal{T}$ with
  full dg subcategories $\mathcal{U}$ and $\mathcal{V}$ 
  such that
  $[\mathcal{T}]=\langle [\mathcal{U}],[\mathcal{V}]\rangle$
  is a semi-orthogonal decomposition into admissible
  subcategories.
  (We do not require that $\mathcal{U}$ and
  $\mathcal{V}$ are saturated; this is automatic by
  Proposition~\ref{p:semi-od-heredity}.)
\end{definition}

If $0$ is the trivial dg algebra (considered as a dg category)
and if $\emptyset$ is the empty dg category, then
$0=\Perf(\emptyset),$ and we have
$\ul{0}=\ul{\Perf(\emptyset))}=0$ in $K_0(\sat_\groundring).$
% A priori $K_0(\sat_\groundring)$ is only an abelian group.

\begin{proposition}
  \label{p:multiplication-of-Zsat-descends}
  The multiplication $\bullet$ on $\DZ\ul{\sat_\groundring}$
  induces a multiplication on $K_0(\sat_\groundring)$ such that $\DZ\ul{\sat_\groundring} \ra
  K_0(\sat_\groundring)$ is a ring morphism.
  Equipped with this multiplication, we call
  $K_0(\sat_\groundring)$ the \define{Grothendieck ring of
    saturated dg categories}. 
\end{proposition}

\begin{proof}
  Let $I \subset \DZ\ul{\sat_\groundring}$ be the
  subgroup generated by the "semi-orthogonal relations". We need
  to show that $I$ is an ideal in
  $\DZ\ul{\sat_\groundring}.$
  Assume that $\mathcal{T}$ is a saturated dg category with
  (saturated) dg 
  subcategories $\mathcal{U}$ and $\mathcal{V}$ such that
  $[\mathcal{T}]=\langle [\mathcal{U}],[\mathcal{V}]\rangle$ is a
  semi-orthogonal decomposition into admissible subcategories.

  Let $\mathcal{S}$ be any saturated dg category. We need to
  prove that 
  \begin{equation*}
    \ul{\mathcal{S}} \bullet \ul{\mathcal{T}} - (\ul{\mathcal{S}} \bullet \ul{\mathcal{U}} + \ul{\mathcal{S}} \bullet
    \ul{\mathcal{V}})
    =
    \ul{\Perf(\mathcal{S} \otimes^\bL \mathcal{T})}
    -
    (\ul{\Perf(\mathcal{S} \otimes^\bL \mathcal{U})}
    +
    \ul{\Perf(\mathcal{S} \otimes^\bL \mathcal{V})})
  \end{equation*}
  is an element of $I.$
  Observe that
  $\mathcal{S}\otimes^\bL \mathcal{A} = Q(\mathcal{S}) \otimes Q(\mathcal{A}) \ra Q(\mathcal{S})
    \otimes \mathcal{A}$
  is a quasi-equivalence (for $\mathcal{A}$ an arbitrary dg
  category) 
  since $Q(\mathcal{S})$ is $\groundring$-h-flat.
  Lemma~\ref{l:dg-functors-restricts-to-perfects-embedding-and-qequiv}.\ref{enum:qequi-induces-qequi-between-perfs}
  shows that
  the above element is equal to
  \begin{equation*}
    \ul{\Perf(Q(\mathcal{S}) \otimes \mathcal{T})}
    -
    \ul{\Perf(Q(\mathcal{S}) \otimes \mathcal{U})}
    +
    \ul{\Perf(Q(\mathcal{S}) \otimes \mathcal{V})}.
  \end{equation*}
  But this element lies in $I$ by
  Proposition~\ref{p:semi-od-goes-to-Perf}.
\end{proof}

\begin{remark}
  Gon\c{c}alo Tabuada shows in
  \cite[section~7]{tabuada-inv-additifs} that for $\groundring$ a
  field (and in the differential $\DZ$-graded situation) 
  there is a surjective morphism
  $K_0(\sat_\groundring) \ra K_0(\Hmo_0^{cl})$
  of commutative rings. We refer the reader to  
  \cite{tabuada-inv-additifs} for the definition of 
  $K_0(\Hmo_0^{cl}).$ This ring is non-zero, so the same is
  true for $K_0(\sat_\groundring).$
\end{remark}

The results of the following subsections~\ref{sec:modif-groth-ring},
\ref{sec:groth-ring-prop-smooth-cats} and
\ref{sec:groth-ring-prop-smooth-alg} are dispensable for
sections~\ref{sec:groth-ring-vari},
\ref{sec:thom-sebast-smoothness},
\ref{sec:land-ginzb-motiv} and
\ref{sec:compactification}.

\subsection{Modified Grothendieck ring of saturated dg categories}
\label{sec:modif-groth-ring}

By omitting the words "into admissible subcategories" in
Definition~\ref{def:Grothendieck-group-of-saturated-dg-cats}
we define the \define{modified Grothendieck group
  $K'_0(\sat_\groundring)$ of saturated dg categories.}
The proof of Proposition~\ref{p:multiplication-of-Zsat-descends}
shows that $K'_0(\sat_\groundring)$ becomes a ring with
multiplication induced by $\bullet.$
There is an obvious surjective morphism  
\begin{equation}
  \label{eq:K0-to-K0prime1}
  K_0(\sat_\groundring) \ra K'_0(\sat_\groundring)  
\end{equation}
of rings.

\begin{proposition}
  \label{p:modified-K0}
  If $\groundring$ is a field, then the map
  \eqref{eq:K0-to-K0prime1} is an isomorphism.
\end{proposition}

The proof of this result requires some additional care in the
differential $\DZ_n$-graded setting.  

\begin{proof}[Proof of Prop.~\ref{p:modified-K0} in the differential $\DZ$-graded setting]  
  Let $\mathcal{T}$ be
  a saturated dg category $\mathcal{T}$ with full dg
  subcategories $\mathcal{U}$ and $\mathcal{V}$ such that
  $[\mathcal{T}]=\langle [\mathcal{U}],[\mathcal{U}]\rangle$ is a
  semi-orthogonal decomposition. We have seen that this already
  implies that both $\mathcal{U}$ and $\mathcal{V}$ are saturated
  dg categories. Then
  \cite[Thm.~3.1]{shklyarov-serre-duality-cpt-smooth-arXiv} (and
  Proposition~\ref{p:saturated-easy-characterization}) show that
  $[\mathcal{U}]$ and $[\mathcal{V}]$ are "saturated" in the
  sense that they are $\Ext$-finite and all covariant and
  contravariant cohomological functors $[\mathcal{T}] \ra
  \Vect(\groundring)$ of finite type are representable. Here
  $\Vect(\groundring)$ denotes the category of vector spaces
  over $\groundring.$ Then
  \cite[Prop.~2.6]{bondal-kapranov-representable-functors} tells
  us that $[\mathcal{U}]$ and $[\mathcal{V}]$ are both admissible
  in $[\mathcal{T}].$
\end{proof}

The following proposition is a variant of
\cite[Thm.~1.3]{bondal-vdbergh-generators}.
We denote the category of finite dimensional vector spaces
over $\groundring$ by $\Vect_\fd(\groundring).$

\begin{proposition}
  \label{p:represent-functor-on-periodic-cat}
  Let $\groundring$ be a field and
  $\mathcal{D}$ a triangulated (in the usual Verdier sense)
  $\groundring$-linear
  category. Assume that 
  \begin{enumerate}
%  \item there is some $n \in \DZ$ such that $\mathcal{D}$ is $n$-periodic in the
%    sense that for each object $A$ of $\mathcal{D}$ there is an
%    isomorphism $A \cong [n]A;$
  \item
    \label{enum:fin-dim-hom}
    $\dim_\groundring \mathcal{D}(A,B) < \infty$ for all
    objects $A,$ $B$ of $\mathcal{D};$
  \item 
    \label{enum:strong-gen-shiftfinite}
    $\mathcal{D}$ has a strong generator $E$ such that 
    there is some $n > 0$ such that $E \cong [n]E$;
    % the set $\{[m]E \mid m \in \DZ\}$ of objects of $\mathcal{D}$
    % consists of finitely many isomorphism classes; 
    and 
  \item 
    \label{enum:karoubian}
    $\mathcal{D}$ is Karoubian.
  \end{enumerate}
  Then  every $\groundring$-linear cohomological functor $\mathcal{D}^\opp \ra
  \Vect_\fd(\groundring)$ is representable.
\end{proposition}

If $\mathcal{D}$ is $n$-periodic (for some $n \in \DZ$) in the
sense that each object $A \in \mathcal{D}$ is isomorphic to
$[n]A,$ then condition~\ref{enum:strong-gen-shiftfinite} is
satisfied if and only if $\mathcal{D}$ has a strong generator.

\begin{proof}
  Let $Z=\{0,1, \dots, n-1\}.$
  % Condition~\ref{enum:strong-gen-shiftfinite} implies  
  % that 
  Then for each $m \in
  \DZ$ there is an $z \in Z$ such that $[m]E \cong [z]E.$ 
  % Condition~\ref{enum:strong-gen-shiftfinite} says that there
  % is a finite subset $Z \subset \DZ$ such that for each $m \in
  % \DZ$ there is an $z \in Z$ such that $[m]E \cong [z]E.$ 
  Now observe that
  the proof of \cite[Thm.~1.3]{bondal-vdbergh-generators}
  contains all the ideas needed to prove this proposition.  
  Only
  subsection "2.4 Construction of resolutions" there needs to be
  modified:
  one essentially replaces all direct sums
  indexed by $\DZ$ by direct sums indexed by $Z.$
\end{proof}

We deduce a variant of \cite[Thm.~3.1]{shklyarov-serre-duality-cpt-smooth-arXiv}.

\begin{proposition}
  \label{p:saturated-dZng-cat-functors-rep}
  Let $\groundring$ be a field and
  $n \in \DZ.$ Let $\mathcal{T}$ be a
  saturated d$\DZ_n$g (= differential $\DZ_n$-graded)
  category. 
  Then all covariant and contravariant $\groundring$-linear cohomological functors $[\mathcal{T}] \ra
  \Vect_\fd(\groundring)$ are representable.
\end{proposition}

\begin{proof}
  We follow the proof of
  \cite[Thm.~3.1]{shklyarov-serre-duality-cpt-smooth-arXiv} 
  but use 
  Proposition~\ref{p:represent-functor-on-periodic-cat} 
  instead of
  \cite[Thm.~1.3]{bondal-vdbergh-generators}.
  By Proposition~\ref{p:saturated-easy-characterization} we can
  assume that $\mathcal{T}=\Perf(A)$ for a smooth and proper
  d$\DZ_n$g algebra $A.$ The argument from the proof of
  \cite[Thm.~3.1]{shklyarov-serre-duality-cpt-smooth-arXiv}
  shows that $[\mathcal{T}]$ has a strong generator. 
  Hence we can apply
  Proposition~\ref{p:represent-functor-on-periodic-cat} and
  obtain 
  that every $\groundring$-linear cohomological functor $[\mathcal{T}]^\opp \ra
  \Vect_\fd(\groundring)$ is representable.
  
  We claim that $\mathcal{T}^\opp$ is also saturated. It
  is certainly pretriangulated, locally proper, and smooth (see for example
  \cite[Remark~3.11]{valery-olaf-smoothness-equivariant}).
  Observe that $[\mathcal{T}^\opp] = [\mathcal{T}]^\opp \cong
  \per(A)^\opp \cong \per(A^\opp)$ where the last equivalence
  comes from the proof of
  \cite[Thm.~3.1]{shklyarov-serre-duality-cpt-smooth-arXiv}.
  This shows that $[\mathcal{T}^\opp]$ is Karoubian, so
  $\mathcal{T}^\opp$ is triangulated, and that $\mathcal{T}^\opp$
  has a compact generator. Hence the above argument applied to
  the saturated d$\DZ_n$g category $\mathcal{T}^\opp$ shows that
  any 
  $\groundring$-linear cohomological functor $[\mathcal{T}]=[\mathcal{T}^\opp]^\opp \ra
  \Vect_\fd(\groundring)$ is representable.
\end{proof}

\begin{proof}[Proof of Prop.~\ref{p:modified-K0} in the
  differential $\DZ_n$-graded setting (for some $n \in \DZ$)]  

  Let $\mathcal{T}$ be a saturated d$\DZ_n$g 
  category $\mathcal{T}$ with full d$\DZ_n$g
  subcategories $\mathcal{U}$ and $\mathcal{V}$ such that
  $[\mathcal{T}]=\langle [\mathcal{U}],[\mathcal{U}]\rangle$ is a
  semi-orthogonal decomposition. We already know that both $\mathcal{U}$ and $\mathcal{V}$ are saturated
  d$\DZ_n$g categories, so we can apply Proposition~\ref{p:saturated-dZng-cat-functors-rep}
  to $\mathcal{U}$ and $\mathcal{V}.$
  The proof of
  \cite[Prop.~2.6]{bondal-kapranov-representable-functors} then tells
  us that $[\mathcal{U}]$ and $[\mathcal{V}]$ are both admissible
  in $[\mathcal{T}].$
\end{proof}
  
\subsection{Grothendieck ring of proper and smooth dg categories}
\label{sec:groth-ring-prop-smooth-cats}

The aim of this section is to give an alternative description of 
the modified Grothendieck ring $K'_0(\sat_\groundring)$ of saturated dg categories using 
smooth and proper dg categories (which are not
necessarily triangulated). 
% $K'_0(\sat_\groundring).$

Recall from \cite{tabuada-model-structure-on-cat-of-dg-cats}
and \cite{tabuada-inv-additifs} (and corrections) that there are
three model category structures on $\dgcat(\groundring).$ They
are all 
cofibrantly generated by the same set of generating
cofibrations. In particular they have the same cofibrations,
cofibrant objects and trivial fibrations, and hence we can use
the same cofibrant replacement functor.

Above we have used the model category structure whose weak
equivalences are the quasi-equivalences and have denoted the
corresponding homotopy category by $\Heq_\groundring.$ Now we
will work with the model structure whose weak equivalences are
the Morita equivalences (= dg foncteurs de Morita). The
corresponding homotopy category will be denoted $\Hmo_\groundring.$

Recall that a dg functor $f \colon  \mathcal{A} \ra \mathcal{B}$ is a
Morita equivalence if the restriction of scalars functor
$D(\mathcal{B}) \ra D(\mathcal{A})$ is an equivalence of
triangulated categories.
It is easy to see that $f$ is a Morita equivalence if and only if
$f^* \colon  \Perf(\mathcal{A}) \ra \Perf(\mathcal{B})$ is a
quasi-equivalence (use 
Lemmata~\ref{l:dg-functor-qequi-test}
and
\ref{l:dg-functors-restricts-to-perfects-embedding-and-qequiv},
and \cite[Lemma~2.12]{lunts-categorical-resolution}).
For example, if $\mathcal{A}$ is any dg category, the Yoneda
morphism $\mathcal{A} \ra \Perf(\mathcal{A})$ is a Morita
equivalence by 
Proposition~\ref{p:perf-to-perf-perf}.

\begin{example}
  \label{exam:endo-dg-algebra-Morita-equivalence}
  Let $\mathcal{T}$ be a dg category with a compact generator.
  Let $E \in \Perf(\mathcal{T})$ be a classical generator of
  $[\Perf(\mathcal{T})],$ and let $A:= (\Perf(\mathcal{T}))(E,E)$
  be its endomorphism dg algebra. Then the proof of 
  Proposition~\ref{p:saturated-easy-characterization}
  shows that the obvious dg functor $A \ra \Perf(\mathcal{T})$ is
  a Morita equivalence.
  Moreover,
  Lemma~\ref{l:properties-dg-K-cat-preserved-under-Morita-equivalence} below
  shows that $\mathcal{T}$ is proper (resp.\ smooth) if and only
  if $A$ is proper (resp.\ smooth).
\end{example}

\begin{lemma}
  \label{l:properties-dg-K-cat-preserved-under-Morita-equivalence}
  The following properties of dg categories are invariant under
  Morita 
  equivalences:
  being
  locally perfect, having a compact
  generator, properness, smoothness.
\end{lemma}

\begin{proof}
  This follows from the above and
  Lemmata~\ref{l:properties-dg-K-cat-preserved-under-quasi-equivalence}
  and \ref{l:A-and-PerfA}.
\end{proof}

\begin{lemma}
  \label{l:L-otimes-respects-Morita-equivalences}
  Morita equivalences $\mathcal{A} \ra \mathcal{A}'$ and
  $\mathcal{B} \ra \mathcal{B}'$ give rise to a Morita equivalence
  $\mathcal{A} \otimes^\bL \mathcal{B} \ra \mathcal{A}'
  \otimes^\bL \mathcal{B}'.$ 
\end{lemma}

\begin{proof}
  Clearly $Q(\mathcal{A}) \ra Q(\mathcal{A}')$ is a Morita
  equivalence, so
  $\Perf(Q(\mathcal{A})) \ra
  \Perf(Q(\mathcal{A}'))$
  and
  $\Perf(Q(\mathcal{A})) \otimes Q(\mathcal{B}) \ra
  \Perf(Q(\mathcal{A}')) \otimes Q(\mathcal{B})$ 
  (by \cite[Lemma~2.15]{valery-olaf-smoothness-equivariant})
  and 
  $\Perf(\Perf(Q(\mathcal{A})) \otimes Q(\mathcal{B})) \ra
  \Perf(\Perf(Q(\mathcal{A}')) \otimes Q(\mathcal{B}))$ 
  (by Lemma~\ref{l:dg-functors-restricts-to-perfects-embedding-and-qequiv}.\ref{enum:qequi-induces-qequi-between-perfs})
  are
  quasi-equivalences. Then Proposition~\ref{p:perf-to-perf-perf}
  shows that 
  $\Perf(Q(\mathcal{A}) \otimes Q(\mathcal{B})) \ra
  \Perf(Q(\mathcal{A}') \otimes Q(\mathcal{B}))$ is a
  quasi-equivalence, so 
  $Q(\mathcal{A}) \otimes Q(\mathcal{B}) \ra
  Q(\mathcal{A}') \otimes Q(\mathcal{B})$ is a Morita
  equivalence.
\end{proof}

\begin{lemma}
  \label{l:tensor-product-of-proper}
  Let $\mathcal{A}$ and $\mathcal{B}$ be proper dg
  categories.  Then $\mathcal{A} \otimes^\bL \mathcal{B}$ is a
  proper dg category.  
\end{lemma}

\begin{proof}
  Example~\ref{exam:endo-dg-algebra-Morita-equivalence}
  % The proof of
  % Proposition~\ref{p:saturated-easy-characterization} 
  shows that
  there are Morita equivalences $A \ra \Perf(\mathcal{A})$ and $B
  \ra \Perf(\mathcal{B})$ for proper dg algebras $A$ and $B.$
  Lemma~\ref{l:tensor-product-of-locally-proper} implies that $A
  \otimes^\bL B$ is a proper dg algebra/category.
  Lemma~\ref{l:L-otimes-respects-Morita-equivalences} 
  shows that $A \otimes^\bL B \ra \Perf(\mathcal{A}) \otimes^\bL
  \Perf(\mathcal{B}) \la \mathcal{A} \otimes^\bL \mathcal{B}$
  consists of Morita equivalences, and hence
  Lemma~\ref{l:properties-dg-K-cat-preserved-under-Morita-equivalence}
  shows that $\mathcal{A} \otimes^\bL \mathcal{B}$ is proper.
\end{proof}

Let $\prsm_\groundring$ be the full subcategory of
$\dgcat_\groundring$ consisting of proper and smooth dg categories.
By $\ul{\prsm_\groundring}$ 
we
denote the set of isomorphism 
classes in $\Hmo_\groundring$ of these
categories
(cf.\
Lemma~\ref{l:properties-dg-K-cat-preserved-under-Morita-equivalence}).
Given a proper and smooth dg category $\mathcal{T},$ we write $\ul{\mathcal{T}}$ for its class
in $\ul{\prsm_\groundring}.$

\begin{proposition}
  \label{p:ulprsm-as-a-commutative-monoid}
  The map $(\mathcal{S},\mathcal{T}) \mapsto \mathcal{S} \otimes^\bL \mathcal{T}$ induces a
  multiplication $\bullet$ on $\ul{\prsm_\groundring}$ that
  turns $\ul{\prsm_\groundring}$ into a commutative monoid
  with unit $\ul{\groundring}.$
\end{proposition}

\begin{proof}
  Lemmata~\ref{l:L-otimes-respects-Morita-equivalences},
  \ref{l:tensor-product-of-proper} and
  \ref{l:tensor-product-of-smooth}
  show that
  $\bullet$ is well defined. 
  We leave the easy proofs of associativity, commutativity, and of
  $\ul{\groundring}$ being the unit to the reader.
\end{proof}

Denote by $\DZ\ul{\prsm_\groundring}$ the (commutative
associative unital) monoid ring of
$\ul{\prsm_\groundring},$ i.\,e.\ the free abelian group
on $\ul{\prsm_\groundring}$ with $\DZ$-bilinear
multiplication induced by $\bullet.$

Let $\mathcal{T}$ be a dg category. Following
\cite[section~3.2]{valery-olaf-smoothness-equivariant} we write
$\mathcal{T}= \tzmat{\mathcal{U}}{0}{*}{\mathcal{V}}$ if
$\mathcal{U}$ and $\mathcal{V}$ are full dg subcategories of
$\mathcal{T}$ such that $\mathcal{T}(\mathcal{V}, \mathcal{U})=0$
and such that the set of objects of $\mathcal{T}$ is the disjoint
union of the sets of objects of $\mathcal{U}$ and of
$\mathcal{V}.$ (Conversely, any two dg categories $\mathcal{U}$
and $\mathcal{V}$ together with a dg $\mathcal{U}\otimes
\mathcal{V}^\opp$-module
$N=\leftidx{_\mathcal{V}}{N}{_\mathcal{U}}$ give rise to such a
"directed" or "lower triangular" dg category
$\tzmat{\mathcal{U}}{0}{N}{\mathcal{V}}.$)

\begin{lemma}
  \label{l:directed-induces-semi-od-on-Perf}
  Assume that $\mathcal{U}$ and $\mathcal{V}$ are full dg
  subcategories of a dg category $\mathcal{T}$ such that
  $\mathcal{T}= \tzmat{\mathcal{U}}{0}{*}{\mathcal{V}}.$
  Then there is an induced semi-orthogonal decomposition
  $[\Perf(\mathcal{T})]=\langle
  [\Perf(\mathcal{U})'],[\Perf(\mathcal{V})']\rangle$ 
  where $\Perf(\mathcal{U})'$ is
  the full dg subcategory of $\Perf(\mathcal{T})$ such that 
  $[\Perf(\mathcal{U})']$ is the strict closure of
  $[\Perf(\mathcal{U})]$ in $[\Perf(\mathcal{T})],$ 
  and $\Perf(\mathcal{V})'$ is defined similarly. 
  In particular, there are obvious quasi-equivalences 
  $\Perf(\mathcal{U}) \ra \Perf(\mathcal{U})'$
  and $\Perf(\mathcal{V}) \ra \Perf(\mathcal{V})'.$
\end{lemma}

\begin{proof}
  The proof is similar to (but easier than) the proof of
  Proposition~\ref{p:semi-od-goes-to-Perf}
  and also based on
  \cite[Lemma~\ref{semi:l:right-admissible-by-triangulated-and-classical-generators}.\ref{semi:enum:T-is-thick-envelope-of-E}]{valery-olaf-matfak-semi-orth-decomp}
\end{proof}

\begin{definition}
  \label{def:Grothendieck-group-of-proper-smooth-dg-cats}
  The \define{Grothendieck group $K_0(\prsm_\groundring)$ of
    proper and smooth dg categories} 
  is defined to be the quotient of
  $\DZ\ul{\prsm_\groundring}$ by the subgroup generated by
  the elements (the "directed relations")
  % \begin{equation*}
  %   \label{eq:semi-orthog-relation}
  $\ul{\mathcal{T}} - (\ul{\mathcal{U}} + \ul{\mathcal{V}})$
  % \end{equation*}
  whenever there is a smooth and proper dg category $\mathcal{T}$
  with 
  full dg subcategories $\mathcal{U}$ and $\mathcal{V}$ 
  such that
  $\mathcal{T}= \tzmat{\mathcal{U}}{0}{*}{\mathcal{V}}.$
  (We do not require that $\mathcal{U}$ and
  $\mathcal{V}$ are smooth and proper since this is automatic:
%  use
%  \cite[Thm.~3.24]{valery-olaf-smoothness-equivariant}
%  for smoothness, and
  use Lemma~\ref{l:directed-induces-semi-od-on-Perf},
  Proposition~\ref{p:semi-od-heredity}, and
  Lemma~\ref{l:properties-dg-K-cat-preserved-under-Morita-equivalence}
  (applied to $\mathcal{U} \ra \Perf(\mathcal{U})$ and
  $\mathcal{V} \ra \Perf(\mathcal{V})$).)
%  for properness.)
\end{definition}

If $0$ is the trivial dg algebra (considered as a dg category)
and if $\emptyset$ is the empty dg category, then $\emptyset \ra
0$ is a Morita equivalence and we have
$\ul{\emptyset}=\ul{0}=0$ in $K_0(\prsm_\groundring).$
% A priori $K_0(\prsm_\groundring)$ is only an abelian group.

\begin{proposition}
  \label{p:multiplication-of-Zprsm-descends}
  The multiplication $\bullet$ on $\DZ\ul{\prsm_\groundring}$
  induces a multiplication on $K_0(\prsm_\groundring)$ such that $\DZ\ul{\prsm_\groundring} \ra
  K_0(\prsm_\groundring)$ is a ring morphism.
  Equipped with this multiplication, we call
  $K_0(\prsm_\groundring)$ the \define{Grothendieck ring of
    proper and smooth dg categories}. 
\end{proposition}

\begin{proof}
  Let $I \subset \DZ\ul{\prsm_\groundring}$ be the
  subgroup generated by the "directed relations". We need 
  to show that $I$ is an ideal in
  $\DZ\ul{\prsm_\groundring}.$
  Assume that $\mathcal{T}$ is a smooth and proper dg category
  with 
  full dg subcategories $\mathcal{U}$ and $\mathcal{V}$ 
  such that
  $\mathcal{T}= \tzmat{\mathcal{U}}{0}{*}{\mathcal{V}}.$
  Let $\mathcal{S}$ be any saturated dg category. 
  Then
  \begin{equation*}
    \ul{\mathcal{S}} \bullet \ul{\mathcal{T}} - (\ul{\mathcal{S}} \bullet \ul{\mathcal{U}} + \ul{\mathcal{S}} \bullet
    \ul{\mathcal{V}})
    =
    \ul{\mathcal{S} \otimes^\bL \mathcal{T}}
    -
    (\ul{\mathcal{S} \otimes^\bL \mathcal{U}}
    +
    \ul{\mathcal{S} \otimes^\bL \mathcal{V}})
    =
    \ul{Q(\mathcal{S}) \otimes \mathcal{T}}
    -
    (\ul{Q(\mathcal{S}) \otimes \mathcal{U}}
    +
    \ul{Q(\mathcal{S}) \otimes \mathcal{V}})
  \end{equation*}
  is an element of $I$ since 
  $Q(S) \otimes \mathcal{T}= \tzmat{Q(S) \otimes
    \mathcal{U}}{0}{*}{Q(S) \otimes \mathcal{V}}.$
\end{proof}

\begin{proposition}
  \label{p:compare-K-prime-sat-and-K-prsm}
  The map $\mathcal{T} \mapsto \Perf(\mathcal{T})$ (for
  proper and smooth $\mathcal{T}$)
  induces an isomorphism 
  \begin{equation*}
    K_0(\prsm_\groundring)
    \sira 
    K'_0(\sat_\groundring)
  \end{equation*}
  of rings with inverse morphism induced by
  $\mathcal{S} \mapsto \mathcal{S}$ (for saturated $\mathcal{S}$).
\end{proposition}

\begin{proof}
  Morita equivalences $\mathcal{T} \ra \mathcal{T}'$
  between proper and smooth dg categories induce 
  quasi-equivalences $\Perf(\mathcal{T}) \ra \Perf(\mathcal{T}')$
  between saturated dg categories, 
  and $\mathcal{T} \ra \Perf(\mathcal{T})$ is a Morita
  equivalence. 
  Quasi-equivalences
  are certainly Morita equivalences, and $\mathcal{S}
  \ra \Perf(\mathcal{S})$ is a quasi-equivalence for saturated
  $\mathcal{S}.$ 
  Hence we get isomorphisms 
  $\ul{\prsm_\groundring} \ra \ul{\sat_\groundring}$
  of monoids (multiplicativity is obvious for the
  inverse $\ul{\mathcal{S}} \mapsto \ul{\mathcal{S}}$)
  and
  $\DZ\ul{\prsm_\groundring} \ra \DZ\ul{\sat_\groundring}$ of
  unital rings.
  Lemma~\ref{l:directed-induces-semi-od-on-Perf} shows that the
  "directed relations" in
  $\DZ\ul{\prsm_\groundring}$ 
  go to zero in $K'_0(\sat_\groundring).$ 

  We claim that the "semi-orthogonal relations" 
  in $\DZ\ul{\sat_\groundring}$
  go to zero in
  $K_0(\prsm_\groundring)$ under $\ul{\mathcal{S}} \mapsto
  \ul{\mathcal{S}}.$  
  Namely, let $\mathcal{S}$ be a saturated dg category with
  full dg subcategories $\mathcal{U}$ and $\mathcal{V}$ 
  such that
  $[\mathcal{S}]=\langle [\mathcal{U}],[\mathcal{V}]\rangle$
  is a semi-orthogonal decomposition.
  Let 
  $\mathcal{V}'$ be the the full dg subcategory of $\mathcal{V}$
  consisting of all objects that are not in $\mathcal{U}.$
  From the proof of Proposition~\ref{p:semi-od-heredity}
  we see that the obvious dg functor 
  $
  \Big[\begin{smallmatrix}
    \mathcal{U} & 0 \\
    \leftidx{_{\mathcal{V}'}}{\mathcal{T}}{_\mathcal{U}} &
    \mathcal{V}'
  \end{smallmatrix}\Big]
  \ra \mathcal{S}$ is a Morita equivalence.
  Obviously $\mathcal{V}' \ra \mathcal{V}$ is a Morita equivalence
  ($\mathcal{V}'$ may be empty). We obtain
  $\ul{\mathcal{S}}= 
  \ul{\Big[
    \begin{smallmatrix}
      \mathcal{U} & 0 \\
      \leftidx{_{\mathcal{V}'}}{\mathcal{T}}{_\mathcal{U}} &
      \mathcal{V}'
    \end{smallmatrix}
    \Big]}
  = \ul{\mathcal{U}}+\ul{\mathcal{V}'} 
  = \ul{\mathcal{U}}+\ul{\mathcal{V}}$ in
  $K_0(\prsm_\groundring).$
  % then $\mathcal{S}(v,u)$ is acyclic for all $v \in \mathcal{V}$ and
  % $u \in \mathcal{U}.$ Let 
  % $\mathcal{V}'$ be the the full dg subcategory of $\mathcal{V}$
  % consisting of all objects that are not in $\mathcal{U},$ and let
  % $\mathcal{S}'$ be the 
  % dg subcategory obtained from $\mathcal{S}$ by replacing all
  % morphism 
  % spaces from objects of $\mathcal{V}'$ to objects of
  % $\mathcal{U}$ by zero.
  % Then 
  % $\mathcal{S}' \ra \mathcal{S}$ is a quasi-equivalence, 
  % $\mathcal{V}' \ra \mathcal{V}$ is a Morita equivalence
  % ($\mathcal{V}'$ may be empty), and we have
  % $\mathcal{S'}= \tzmat{\mathcal{U}}{0}{*}{\mathcal{V}'}.$
  This shows the claim and proves the proposition.
\end{proof}

\subsection{Grothendieck ring of proper and smooth dg algebras}
\label{sec:groth-ring-prop-smooth-alg}

The aim of this section is to give an alternative
description of the ring $K_0(\prsm_\groundring)$
using smooth and proper dg algebras (instead of categories).

\begin{lemma}
  \label{l:tensor-equivalence-and-isomorphic-in-Hmo}
  Let $\mathcal{A},$ $\mathcal{B}$ be dg categories, and let 
  $X=\leftidx{_\mathcal{B}}{X}{_\mathcal{A}}$
  be a dg $\mathcal{A} \otimes \mathcal{B}^\opp$-module.
  Assume that $(- \otimes_\mathcal{B}^\bL X) \colon  D(\mathcal{B}) \ra
  D(\mathcal{A})$ is an equivalence of triangulated categories.
  Then $\mathcal{A}$ and $\mathcal{B}$ are isomorphic in
  $\Hmo_\groundring.$
\end{lemma}

\begin{remark}
  \label{rem:dg-Morita-equivalence-and-isomorphic-in-Hmo}
  Lemma~\ref{l:tensor-equivalence-and-isomorphic-in-Hmo}
  shows that two dg categories $\mathcal{A}$ and $\mathcal{B}$
  are "dg Morita equivalent" in the sense of
  \cite[Def.~3.13]{valery-olaf-smoothness-equivariant} if and
  only if they are isomorphic in $\Hmo_\groundring.$ 
\end{remark}

\begin{proof}
  Let $\mathcal{A}' := Q(\mathcal{A}) \ra \mathcal{A}$
  and $\mathcal{B}' := Q(\mathcal{B}) \ra \mathcal{B}$ be
  cofibrant resolutions.  Consider $X$ by restriction of scalars
  as a dg $\mathcal{A}' \otimes
  \mathcal{B}'^\opp$-module, and let $X' \ra X$
  be a cofibrant resolution in $C(\mathcal{A}' \otimes
  \mathcal{B}'^\opp).$ Corollary 3.15 and the beginning
  of the proof of Proposition 3.16 in
  \cite{valery-olaf-smoothness-equivariant} show that
  the dg functor $T_{X'}:=(- \otimes_{\mathcal{B}'} X') \colon 
  \calMod(\mathcal{B}') \ra 
  \calMod(\mathcal{A}')$
  directly descends to an equivalence
  $T_{X'} \colon 
  D(\mathcal{B}') \ra 
  D(\mathcal{A}')$ of triangulated
  categories. On compact objects we obtain an equivalence
  $T_{X'} \colon  \per(\mathcal{B}') \ra \per(\mathcal{A}').$
  
  For $B' \in
  \mathcal{B}'$ 
  note that $T_{X'}(\Yoneda{B'})=X'(-,B')$ is a cofibrant dg
  $\mathcal{A}'$-module by 
  \cite[Prop.~2.10.(b) and
  Lemma~2.14]{valery-olaf-smoothness-equivariant}.
  Hence the dg functor $T_{X'}$ maps cofibrant 
  dg $\mathcal{B}'$-modules to cofibrant dg
  $\mathcal{A}'$-modules (use
  \cite[Lemma~2.7]{valery-olaf-smoothness-equivariant}). 
  
  This shows that the dg functor $T_{X'} \colon 
  \Perf(\mathcal{B}') \ra \Perf(\mathcal{A}')$ 
  induced by $T_{X'}$ is a
  quasi-equivalence. 
  Hence $\mathcal{B} \la \mathcal{B}' \ra \Perf(\mathcal{B}')
  \xra{T_{X'}} \Perf(\mathcal{A}') \la \mathcal{A}' \ra
  \mathcal{A}$ consists of Morita equivalences, so $\mathcal{A}$
  and $\mathcal{B}$ are isomorphic in $\Hmo_\groundring.$
\end{proof}

\begin{lemma}
  \label{l:reduce-Morita-equivalence-prop-smooth-to-algebras}
  Let $f \colon  \mathcal{S} \ra \mathcal{T}$ be a Morita equivalence
  between dg categories having a
  compact generator. Let $A$ and $B$ be endomorphism dg algebras
  of objects $E \in \Perf(\mathcal{S})$ and
  $F \in \Perf(\mathcal{T}),$ respectively, that become classical
  generators of $[\Perf(\mathcal{S})]$ and $[\Perf(\mathcal{T})],$
  respectively. If $\mathcal{S}$ and $\mathcal{T}$ are proper and
  smooth, the same is true for $A$ and $B.$ 
  Then there is a dg $B \otimes A^\opp$-module
  $X=\leftidx{_A}{X}{_B}$
  such that
  $(- \otimes_A^\bL X) \colon  D(A) \ra
  D(B)$ is an equivalence of triangulated categories.
\end{lemma}

\begin{proof}
  Example~\ref{exam:endo-dg-algebra-Morita-equivalence}
  shows that the obvious dg functors $A \ra \Perf(\mathcal{S})$
  and $B \ra \Perf(\mathcal{T})$ are Morita equivalences and
  that 
  properness and smoothness is passed on from $\mathcal{S}$ and
  $\mathcal{T}$ to $A$ and $B.$
  The Morita equivalences $A \ra \Perf(\mathcal{S}) \xra{f^*}
  \Perf(\mathcal{T}) \la B$ (where $f^*$ is even a
  quasi-equivalence) define equivalences
  \begin{equation*}
    [\calMod(A)_{cf}] \sira 
    [\calMod(\Perf(\mathcal{S}))_{cf}] \xsira{[(f^*)^*]} 
    [\calMod(\Perf(\mathcal{T}))_{cf}] 
    \sira D(\Perf(\mathcal{T})) 
    \xsira{\res^{\Perf(\mathcal{T})}_B} D(B)
  \end{equation*}
  mapping
  $A$ to $(\Perf(\mathcal{S}))(-,E)$ to
  $(\Perf(\mathcal{T}))(-,f^*(E))$ to
  $(\Perf(\mathcal{T}))(F,f^*(E)).$
  Hence we can take $X$ to be the
  dg $B \otimes A^\opp$-module $(\Perf(\mathcal{T}))(F,
  f^*(E)).$
\end{proof}

Let $\prsmalg_\groundring$ be the full subcategory of
$\dgcat_\groundring$ consisting of proper and smooth dg algebras
(= dg categories with one object).
We consider the equivalence relation on the objects of
$\prsmalg_\groundring$ 
generated by $A \sim B$ if there is a
dg $B \otimes A^\opp$-module
$X=\leftidx{_A}{X}{_B}$
such that
$(- \otimes_A^\bL X) \colon  D(A) \ra
D(B)$ is an equivalence of triangulated categories.
Let $\ul{\prsmalg_\groundring}$ be the set of equivalence
classes. Given a proper and smooth dg algebra $A$ we denote its
class by $\ul{A}.$

\begin{lemma}
  \label{l:prsmalg-prsmcat}
  The inclusion $\prsmalg_\groundring \ra \prsm_\groundring$
  induces an isomorphism 
  $\ul{\prsmalg_\groundring} \sira \ul{\prsm_\groundring}$ of
  sets and then an
  isomorphism $\DZ\ul{\prsmalg_\groundring} \sira
  \DZ\ul{\prsm_\groundring}$ of abelian groups.
\end{lemma}

\begin{proof}
  The first isomorphism trivially yields the second one which we
  prove now.
  Lemma~\ref{l:tensor-equivalence-and-isomorphic-in-Hmo} shows
  that the map 
  $\ul{\prsmalg_\groundring} \ra \ul{\prsm_\groundring}$ is
  well-defined.
  If $\mathcal{T}$ is any proper and smooth dg category, take any
  $E \in \Perf(\mathcal{T})$ that is a classical generator of
  $[\Perf(\mathcal{T})],$ and let $A=\Perf(\mathcal{T})(E,E).$
  Then $A \ra \Perf(\mathcal{T})$ is a Morita equivalence
  and $A$ is proper and smooth (see
  Example~\ref{exam:endo-dg-algebra-Morita-equivalence}). This 
  shows surjectivity.
  Injectivity follows from
  Lemma~\ref{l:reduce-Morita-equivalence-prop-smooth-to-algebras}.
\end{proof} 

Let $A$ and $B$ be dg algebras, and let 
$N=\leftidx{_B}{N}{_A}$ be a dg $A \otimes B^\opp$-module.
Then we can form the dg algebra 
$\roundtzmat A0NB.$ We use round brackets in order to distinguish 
$\roundtzmat A0NB$ from the dg category
$\tzmat A0NB$ with two objects.

\begin{definition}
  \label{def:Grothendieck-group-of-proper-smooth-dg-alg}
  The \define{Grothendieck group $K_0(\prsmalg_\groundring)$ of
    proper and smooth dg algebras}
  is defined to be the quotient of the abelian group
  $\DZ\ul{\prsmalg_\groundring}$ by the subgroup generated by
  the elements (the "lower-triangular matrix algebra relations")
  % \begin{equation*}
  %   \label{eq:semi-orthog-relation}
  $\ul{R} - (\ul{A} + \ul{B})$
  % \end{equation*}
  whenever $R$ is a proper and smooth dg algebra such that
  there are dg algebras $A$ and $B$ together with a
  dg $A \otimes B^\opp$-module 
  $N=\leftidx{_B}{N}{_A}$ such that 
  $R=\roundtzmat A0NB.$ 
  (We do not require that $A$ and $B$ are smooth and proper since
  this is automatic: properness of $A$ and $B$ is obvious, and
  smoothness follows from
  \cite[Thm.~3.24 and
  Rem.~3.25]{valery-olaf-smoothness-equivariant}.) 
\end{definition}

\begin{proposition}
  \label{p:compare-K-prsm-and-K-prsmalg}
  The isomorphism $\DZ\ul{\prsmalg_\groundring} \sira
  \DZ\ul{\prsm_\groundring}$ of abelian groups induces an
  isomorphism 
  \begin{equation*}
    K_0(\prsmalg_\groundring)
    \sira 
    K_0(\prsm_\groundring)
  \end{equation*}
   of abelian groups.
\end{proposition}

\begin{proof}
  It is easy to see that 
  the dg algebra $\roundtzmat A0NB$ and the dg category
  $\tzmat A0NB$ are Morita equivalent, cf.\
  \cite[Rem.~3.25]{valery-olaf-smoothness-equivariant}. 
  Hence the "lower-triangular matrix algebra relations" go to
  zero in  
  $K_0(\prsm_\groundring)$ and we obtain a morphism 
  $K_0(\prsmalg_\groundring)
  \ra
  K_0(\prsm_\groundring)$ of groups.
  
  Let $\mathcal{T}$ be a smooth and proper dg category with full
  dg subcategories $\mathcal{U}$ and $\mathcal{V}$ such that
  $\mathcal{T}= \tzmat{\mathcal{U}}{0}{*}{\mathcal{V}}.$
  Then we have a semi-orthogonal decomposition
  $[\Perf(\mathcal{T})] = \langle [\Perf(\mathcal{U})],
  [\Perf(\mathcal{V})]\rangle$ by
  Lemma~\ref{l:directed-induces-semi-od-on-Perf}.
  Choose $u \in \Perf(\mathcal{U})$ and $v \in
  \Perf(\mathcal{V})$ that become classical generators of 
  $[\Perf(\mathcal{U})]$ and $[\Perf(\mathcal{V})]$ respectively.
  Then $u \oplus v$ is a classical generator of
  $[\Perf(\mathcal{T})].$ Let $A,$ $B,$ $R$ be the
  endomorphism dg algebras of $u,$ $v,$ $u\oplus v,$
  respectively. Then 
  the "directed relation"
  $\ul{\mathcal{T}} - (\ul{\mathcal{U}} + 
  \ul{\mathcal{V}})$ in 
  $\DZ\ul{\prsm_\groundring}$ is mapped to 
  $\ul{R} - (\ul{A} + \ul{B})$ in
  $\DZ\ul{\prsmalg_\groundring}$
  under the inverse of the isomorphism of
  Lemma~\ref{l:prsmalg-prsmcat}. 
  Note that $R=\roundtzmat
  A{(\Perf(\mathcal{T}))(v,u)}{(\Perf(\mathcal{T}))(u,v)}B.$
  But since ${(\Perf(\mathcal{T}))(v,u)}$ is acyclic the obvious
  morphism 
  $R':=\roundtzmat
  A0
  {(\Perf(\mathcal{T}))(u,v)}B \ra R$ is a quasi-isomorphism of dg
  algebras. Hence
  $\ul{R} - (\ul{A} + \ul{B})=\ul{R'} - (\ul{A} + \ul{B})$ goes
  to zero in 
  $K_0(\prsmalg_\groundring),$ and this implies the proposition.
\end{proof}

Note that up to now 
$\ul{\prsmalg_\groundring}$ is only a set and
$\DZ\ul{\prsmalg_\groundring}$ and
$K_0(\prsmalg_\groundring)$ are only abelian groups.
However the isomorphisms from 
Lemma~\ref{l:prsmalg-prsmcat}
and Proposition~\ref{p:compare-K-prsm-and-K-prsmalg}
enable us to equip these structures with multiplication maps
$\bullet$ which are obviously induced by $(A,B)
\mapsto A \otimes^\bL B.$
So 
$\ul{\prsmalg_\groundring}$ is a monoid with unit
$\ul{\groundring},$ and 
$\DZ\ul{\prsmalg_\groundring}$ and
$K_0(\prsmalg_\groundring)$ are commutative rings with unit 
$\ul{\groundring}.$

\begin{definition}
  \label{def:Grothendieck-ring-of-proper-smooth-dg-alg}
  We call $K_0(\prsmalg_\groundring)$ with the multiplication
  $\bullet$ induced by $(A,B) \mapsto A \otimes^\bL B.$ the
  \define{Grothendieck ring of proper and smooth dg algebras}.
\end{definition}

\begin{remark}
  \label{rem:K0s-combined}  
  If we combine 
  Propositions~\ref{p:compare-K-prsm-and-K-prsmalg} and
  \ref{p:compare-K-prime-sat-and-K-prsm}
  we obtain ring isomorphisms
  \begin{equation*}
    K_0(\prsmalg_\groundring)
    \sira
    K_0(\prsm_\groundring)
    \sira 
    K'_0(\sat_\groundring)
  \end{equation*}
  induced by $A \mapsto A$ (for $A$ a proper and smooth dg
  algebra) and $\mathcal{T} \mapsto \Perf(\mathcal{T})$ (for
  $\mathcal{T}$ a proper and smooth dg category).
  The inverse map $K'_0(\sat_\groundring) \sira
  K_0(\prsmalg_\groundring)$ is induced by mapping a saturated dg
  category $\mathcal{S}$ to the endomorphism dg algebra of an
  arbitrary classical generator of $[\mathcal{S}].$
\end{remark}

\section{Grothendieck ring of varieties over \texorpdfstring{$\DA^1$}{A1}}
\label{sec:groth-ring-vari}

A variety is a reduced separated
scheme of finite type over a field $k$ (not necessarily irreducible).
In this section we assume that $k$ has characteristic zero. 
If $X$ and $Y$ are schemes over $k$ we abbreviate $X \times Y := X \times_{\Spec k} Y.$
% Smoothness means smoothness over $k$ if not indicated otherwise.
% We
Denote by $\DA^1=\DA^1_k$ the affine line over $k.$ An
$\DA^1$-variety is a variety $X$ together with a morphism $X \ra
\DA^1.$

\begin{definition}
  [\cite{bittner-euler-characteristic}]
  \label{d:groth-group-var-over-A1}
  The \define{(motivic) Grothendieck group $K_0(\Var_{\DA^1})$ of
    varieties 
    over $\DA^1$} is 
  the free abelian group on isomorphism classes
  $[X]_{\DA^1}$ of varieties $X \ra \DA^1$ over $\DA^1$
  subject 
  to the 
  relations $[X]_{\DA^1}=[X-Y]_{\DA^1} + [Y]_{\DA^1}$ whenever
  $Y \subset X$ is a closed subvariety. 
\end{definition}

Sometimes we write $[X,W]$ instead of $[X]_{\DA^1}$ if we want to
emphasize the morphism $W \colon X \ra \DA^1.$
The following theorem describes two  
alternative presentations
of the Grothendieck group
$K_0(\Var_{\DA^1})$ of varieties
over $\DA^1.$

\begin{theorem}
  [{\cite[Thm.~5.1]{bittner-euler-characteristic}}]
  \label{t:groth-group-var-over-A1-smooth-and-blowup}
  The obvious morphisms from the following two abelian groups to 
  $K_0(\Var_{\DA^1})$ are isomorphisms.
  % \fotnote{
  %   Eventuell sinnvoll, urspr\"ungliche Bittner-Aussage
  %   anzugeben, ohne quasi-projektiv \"uber $\DA^1$
  %   bzw. proper statt projektiv...
  % }
  \begin{enumerate}[label=(sm)]
  \item
    \label{enum:smooth-presentation}
    $K_0^\sm(\Var_{\DA^1}),$ the free abelian group on
    isomorphism classes 
    $[X]_{\DA^1}$ of 
    $\DA^1$-varieties which are smooth over $k,$ 
    % connected,
    % \fotnote{
    %   habe connected gel\"oscht, da Produkt zshgder nicht
    %   notw. zshgd. Will aber Produkt mit Hilfe dieser
    %   Pr\"asentation definieren...
%
    %   Aus Gleichberechtigungsgr\"unden habe dann connected auch
    %   in (bl) gel\"oscht.
%
    %   Jedoch brauchen wir es nur f\"ur $k$ algebraisch
    %   abgeschlossen von Charakteristik Null. Also k\"onnte
    %   durchaus connected annehmen...
    % }
    %   
    %   and
    %   quasi-projective over $\DA^1,$ 
    subject to the
    relations
    $[X]_{\DA^1}=[X-Y]_{\DA^1}+[Y]_{\DA^1},$ where 
    $X$ is smooth over $k,$ 
    % , connected, 
    % and quasi-projective over
    % $\DA^1,$  
    and $Y\subset X$ is a $k$-smooth 
    % connected 
    closed subvariety.
  \end{enumerate}
  \begin{enumerate}[label=(bl)]
  \item
    \label{enum:blowup-presentation}
    $K_0^\bl(\Var_{\DA^1}),$
    the free abelian group on isomorphism classes
    $[X]_{\DA^1}$ of 
    $\DA^1$-varieties which are smooth over $k$
    % , connected, 
    and proper
    over $\DA^1$ subject to relations $[\emptyset]_{\DA^1}=0$ 
    % (if
    % $\emptyset$ is considered as a connected set) 
    and
    $[\Blow_Y(X)]_{\DA^1}-[E]_{\DA^1}=[X]_{\DA^1}-[Y]_{\DA^1},$
    where $X$ is smooth over $k$
    % , connected, 
    and proper over
    $\DA^1,$ $Y\subset X$ is a $k$-smooth 
    % connected 
    closed
    subvariety, $\Blow_Y(X)$ is the blowing-up of $X$ along $Y,$
    and $E$ is the exceptional divisor of this blowing-up.
  \end{enumerate}
  In case
  \ref{enum:smooth-presentation} we can restrict to
  varieties which are in addition quasi-projective over
  $\DA^1$
  (and hence quasi-projective over $k$),
  and in case 
  \ref{enum:blowup-presentation} to varieties which are
  projective over
  $\DA^1$ (and hence quasi-projective over $k$).
  In both cases we can restrict to connected varieties.
\end{theorem}

The presentation \ref{enum:blowup-presentation} of
$K_0(\Var_{\DA^1})$ is very important for us whereas the
presentation \ref{enum:smooth-presentation} is not used in the
rest of this article.

Using that $\DA^1$ is an abelian algebraic group we now turn
$K_0(\Var_{\DA^1})$ into a commutative ring with unit. Given
varieties $W \colon X \ra \DA^1$ and $V \colon  Y \ra \DA^1$ define $W*V$ to be
the composition
\begin{equation*}
  W*V \colon  X \times Y \xra{W \times V}
  \DA^1\times \DA^1\xra{+} \DA^1. 
\end{equation*}
%Let $[X \times Y]_{\DA^1}$ be the class of $W*V \colon  X \times Y \ra
%\DA^1.$ 
From Definition~\ref{d:groth-group-var-over-A1} it is
clear that $[X,W] \cdot [Y,V]:= [X \times Y, W*V]$
turns the abelian group $K_0(\Var_{\DA^1})$ into a commutative
ring with unit $[\Spec k, 0],$ the class of the zero
function $\Spec k \xra{0} \DA^1$.

The same recipe turns $K_0^\sm(\Var_{\DA^1})$ 
into a ring such that $K_0^\sm(\Var_{\DA^1}) \ra
K_0(\Var_{\DA^1})$ is an isomorphism of rings.
Note however that this recipe does not work for
$K_0^\bl(\Var_{\DA^1})$: if $W \colon  X \ra \DA^1$ and $V \colon  Y \ra \DA^1$
are projective, $W*V$ is not projective in general.

\begin{remark}
  \label{rem:class-L-over-A1}
  We denote the class of the zero morphism
  $\DA^1 \xra{0} \DA^1$ by $\DL_{(\DA^1,0)}
  % \DL:=\DL_{\DA^1}
  :=[\DA^1,0].$

  Let us justify this. Similar as above one defines the
  Grothendieck ring $K_0(\Var_k)$ of varieties over $k,$ with
  multiplication given by $[X] \cdot [Y]:= [X \times Y].$ The map
  $K_0(\Var_k) \ra K_0(\Var_\DA^1)$ given by $[X] \mapsto [X, 0]$
  is then a morphism of unital rings. It maps the class $\DL_k$ of
  $\DA^1 \ra \Spec k$ to $\DL_{(\DA^1,0)}.$
\end{remark}

\begin{definition}
  A \define{Landau-Ginzburg (LG) motivic measure} is a morphism of unital rings from
  $K_0(\Var_{\DA^1})$ to another ring.
\end{definition}

\section{Thom-Sebastiani Theorem and smoothness}
\label{sec:thom-sebast-smoothness}

We now start to consider categories of matrix factorizations.
Our notation and many results are explained in
\cite{valery-olaf-matfak-semi-orth-decomp}. 
Our aim in this section is to prove the
Thom-Sebastiani Theorem~\ref{t:thom-sebastiani} and the
smoothness result of Theorem~\ref{t:MF-DSgW-Cechobj-smooth}.

We fix a field $k$ which can be arbitrary in
section~\ref{sec:object-orient-vcech-MF} and is assumed 
to be algebraically closed and of characteristic zero starting
from section~\ref{sec:categ-sing}. By a scheme we mean a scheme
over $k,$ 
and by a variety a reduced separated
scheme of finite type over $k$ (as in
section~\ref{sec:groth-ring-vari}).

In this and the following section dg means "differential
$\DZ_2$-graded". When we refer to results from
section~\ref{sec:groth-ring-saturated-dg-cats} we always mean the
differential $\DZ_2$-graded version
(see Remark~\ref{rem:more-general}) for $\groundring=k.$

\subsection{Object oriented \v{C}ech enhancements for matrix
  factorizations} 
\label{sec:object-orient-vcech-MF}

This section runs parallel to
\cite[section~\ref{enhance:sec:vcech-enhanc-loc-integral}]{valery-olaf-enhancements-in-prep};
we will therefore often refer to results there
and assume that the reader is familiar with the
notation and arguments there.

We say that a scheme $X$
satisfies condition~\ref{enum:srNfKd} 
if
\begin{enumerate}[label=(srNfKd)]
\item
  \label{enum:srNfKd}
  $X$ is a separated regular Noetherian scheme of finite
  Krull dimension. 
\end{enumerate}
This is the condition we have worked with in
\cite{valery-olaf-matfak-semi-orth-decomp}.  From the discussion
there it is clear that this condition implies
condition~\ref{enhance:enum:GSP+} in
\cite{valery-olaf-enhancements-in-prep}.

\begin{remark}
  \label{rem:alternative-for-srNfKd}
  If schemes $X$ and $Y$ satisfy condition~\ref{enum:srNfKd} it
  is in general not true that so does $X \times Y.$
  Hence as soon as we work on products we need to 
  require condition~\ref{enum:srNfKd} there.
  To avoid this annoyance one may work with 
  smooth varieties, i.\,e. separated smooth schemes of finite type
  (over the field $k$). 
  Every smooth variety 
  satisfies condition~\ref{enum:srNfKd}, and products of smooth
  varieties are again smooth varieties.
\end{remark}

Let $X$ be a scheme satisfying condition~\ref{enum:srNfKd} and 
let $\mathcal{U}=(U_s)_{s \in S}$ be a finite affine open covering
of $X.$
Given a vector bundle $P$ on $X$ we can consider its (finite)
ordered \v{C}ech resolution
\begin{equation*}
  % \label{eq:P-*-Cech-resolution}
  % P \ra 
  % \mathcal{C}_\ord(P):=
  \mathcal{C}^\bullet_\ord(P):=
  % \mathcal{C}_{\ord, \mathcal{U}}^\bullet(P):=
  \Big(\prod_{s_0 \in S} \leftidx{_{U_{\{s_0\}}}}{P}{} \ra
  \prod_{s_0, s_1 \in S,\; s_0 < s_1} \leftidx{_{U_{\{s_0,s_1\}}}}{P}{} \ra
  \dots
  \Big)
\end{equation*}
with the usual differentials where
we abbreviate $U_I:= \bigcap_{i \in I} U_i$ for a subset $I \subset
S$ and use the notation
$\leftidx{_V}{P}{}:=j_*j^*(P)$ if $j \colon V \hra X$ is the
inclusion of an open subscheme; 
note that $\mathcal{C}^\bullet_\ord(P)$ depends on $\mathcal{U}$
and also 
on the choice of a total 
order $<$ on $S.$ However, we can and will neglect the choice of
$<$ since different choices lead to isomorphic resolutions. 

Let $W \colon X \ra \DA^1$ be a morphism.
Consider the
functor that maps
a vector bundle $P$ on $X$ to 
%its (finite)
%ordered \v{C}ech resolution 
$\mathcal{C}^\bullet_\ord(P).$
% see
% \eqref{eq:P-*-Cech-resolution}. 
% In fact $P \mapsto \mathcal{C}^\bullet_\ord(P)$ is a functor.
If we apply it to an object $E \in
\MF(X,W)$ we obtain a (bounded) complex in $Z_0(\Qcoh(X,W))$
that we denote by $\mathcal{C}^\bullet_\ord(E).$ 
We denote its totalization by
$\mathcal{C}_\ord(E):=\Tot(\mathcal{C}^\bullet_\ord(E)) \in \Qcoh(X,W).$

Let $\MF_\Cechobj(X,W)$ (omitting $\mathcal{U}$ from the
notation) be the smallest 
full dg subcategory of $\Qcoh(X,W)$
that contains all objects $\mathcal{C}_\ord(E)$
for $E \in \MF(X, W),$ is closed under shifts,
under cones of
closed degree zero morphisms and unter taking homotopy equivalent
objects (i.\,e.\ objects that are isomorphic in $[\Qcoh(X,W)]$).
It is strongly pretriangulated.

\begin{proposition}
  \label{p:cech-object-enhancement-MF}
  The dg category $\MF_\Cechobj(X,W)$ 
  is naturally an enhancement of $\bfMF(X,W).$
  More precisely, the natural functor
  \begin{equation*}
    \epsilon \colon  [\MF_\Cechobj(X,W)] \ra \DQcoh(X,W)
  \end{equation*}
  is full and faithful and its essential image coincides with
  the closure under isomorphisms of $\bfMF(X,W) \subset
  \DQcoh(X,W)$
  (see \cite[Thm.~\ref{semi:t:equivalences-curved-categories}]{valery-olaf-matfak-semi-orth-decomp}).
  We call $\MF_\Cechobj(X,W)$ the \define{object
    oriented \v{C}ech enhancement} of $\bfMF(X,W).$ 
\end{proposition}

\begin{proof}
  It is clear that the essential image of $\epsilon$ is as
  claimed: given $E \in \MF(X,W),$ the obvious morphism $E \ra
  \mathcal{C}_\ord(E)$ 
  becomes an isomorphism in $\DQcoh(X,W).$

  Let $E, F \in \MF(X,W).$
  In order to prove that $\epsilon$ is full and faithful it is
  enough to show that 
  \begin{equation*}
    \Hom_{[\Qcoh(X,W)]}(\mathcal{C}_\ord(E),
    [m]\mathcal{C}_\ord(F))
    \ra
    \Hom_{\DQcoh(X,W)}(\mathcal{C}_\ord(E), [m]\mathcal{C}_\ord(F))
  \end{equation*}
  is an isomorphism, for any $m \in \DZ_2.$
  Note that $\mathcal{C}_\ord(F)$ is constructed as an iterated cone
  from shifts of objects $\leftidx{_{V}}{F}{}:=j_*j^*(F),$ where
  $I \subset S$ and $j \colon  V:=U_I := \bigcap_{i \in I} U_i \ra X$ is
  the corresponding open 
  embedding. Hence, as in the proof of
  \cite[Lemma~\ref{enhance:l:CD-isom-CP-VQ}]{valery-olaf-enhancements-in-prep},
  we need to show the following
  two claims.
  \begin{enumerate}
  \item
    \label{enum:E-VF-MF}
    $\Hom_{[\Qcoh(X,W)]}(E, [n]\leftidx{_V}{F}{}) \ra
    \Hom_{\DQcoh(X,W)}(E, [n]\leftidx{_V}{F}{})$
    is an isomorphism, for any $n \in \DZ_2.$
  \item
    \label{enum:CechE-to-E}
    $\Hom_{\Qcoh(X,W)}(\Tot(\mathcal{C}^\bullet_\ord(E)),
    \leftidx{_V}{F}{}) 
    \ra
    \Hom_{\Qcoh(X,W)}(E, 
    \leftidx{_V}{F}{})$
    is a quasi-isomorphism.
  \end{enumerate}
  
  Proof of \ref{enum:E-VF-MF}: Note that $\bR j_*=j_*$ and
  $\bL j^*=j^*,$ by \cite[Lemma~\ref{semi:l:componentwise-acyclics}]{valery-olaf-matfak-semi-orth-decomp},
  since $j$ is open and affine. Hence 
  by the adjunctions $(j^*, j_*)$ it is enough to show that
  \begin{equation*}
    \Hom_{[\MF(V,W)]}(j^*(E), [n]j^*(F)) \ra
    \Hom_{\bfMF(V,W)}(j^*(E), [n]j^*(F))
  \end{equation*}
  is an isomorphism (we use that $\bfMF(V,W) \ra \DQcoh(V,W)$
  is full and faithful, by \cite[Thm.~\ref{semi:t:equivalences-curved-categories}]{valery-olaf-matfak-semi-orth-decomp}).
  But $[\MF(V,W)] \sira \bfMF(V,W)$ since $V$ is affine, by
  \cite[Lemma~\ref{semi:l:affine-MF}]{valery-olaf-matfak-semi-orth-decomp}. 

  Proof of \ref{enum:CechE-to-E}: 
  The domain of the given morphism is the totalization of the
  bounded complex
  \begin{equation*}
    \dots
    \ra 
    \Hom_{\Qcoh(X,W)}(\mathcal{C}^1_\ord(E), \leftidx{_V}{F}{})
    \ra
    \Hom_{\Qcoh(X,W)}(\mathcal{C}^0_\ord(E), \leftidx{_V}{F}{})
    \ra 0
  \end{equation*}
  in $Z_0(\Sh(\Spec k, 0)).$ We can also view this complex as a
  $\DZ_2 \times \DZ$-graded double complex.
  Hence the given morphism is the totalization of a morphism of
  double complexes.
  Then \cite[Lemma~\ref{semi:l:double-complex-upper-halfplane}.\ref{semi:enum:double-complex-upper-halfplane-column-qisos}]{valery-olaf-matfak-semi-orth-decomp}
  % (our double complexes are not concentrated in non-negative
  % degrees in the $\DZ$-direction, but this does not matter
  % since they are bounded in this direction) 
  shows that it is enough to show that
  \begin{equation*}
    \Hom_{C(\Qcoh(X))}(\mathcal{C}^\bullet_\ord(E_s),
    \leftidx{_V}{F}{_t}) 
    \ra
    \Hom_{C(\Qcoh(X))}(E_s, \leftidx{_V}{F}{_t})
  \end{equation*}
  is a quasi-isomorphism for all $s, t \in \DZ_2.$ But this is
  true by the argument 
  %in the proof of
  %\cite[Lemma~2.3]{valery-olaf-enhancements-in-prep}
  % Lemma~\ref{l:CD-isom-CP-VQ} 
  that shows that 
  the morphism in
  \cite[Formula~\eqref{enhance:eq:CP-UIQ}]{valery-olaf-enhancements-in-prep}   
  is a quasi-isomorphism ($\mathcal{C}_*$ there is denoted 
  $\mathcal{C}^\bullet_\ord$ here; implicitly we replace $V$
  by one of its connected components).
\end{proof}

\begin{remark}
  [{cf.\ \cite[Rem.~\ref{enhance:rem:objects-of-cech-object-enhancement}]{valery-olaf-enhancements-in-prep}}]
  \label{rem:objects-of-cech-object-enhancement-MF}
  The objects of 
  $\MF_\Cechobj(X,W)$ are precisely the objects of $\Qcoh(X,W)$
  that are homotopy equivalent to an object of
  the form $\mathcal{C}_\ord(E)$, for $E \in \MF(X,W).$
\end{remark}

Let $Y$ be another scheme and assume that $Y$ and
$X \times Y$ satisfy condition~\ref{enum:srNfKd}
(cf.~Remark~\ref{rem:alternative-for-srNfKd}).
We fix a morphism $V\colon Y \ra \DA^1$ and a 
finite affine open
covering $\mathcal{V}$ of $Y.$ We consider the product covering
$\mathcal{U} \times \mathcal{V}$ on $X \times Y.$
In order to prove the analog of
\cite[Prop.~\ref{enhance:p:cech-*-object-enhancement-product}]{valery-olaf-enhancements-in-prep}
we let $\MF_\Cechobjbox(X \times Y, W*V)$ be the smallest full dg
subcategory 
of $\Qcoh(X \times Y, W*V)$ that contains all objects
$\mathcal{C}_\ord(E) \boxtimes \mathcal{C}_\ord(F)$ for $E \in
\MF(X,W)$ and $F \in \MF(Y,V),$
all objects $\mathcal{C}_\ord(G)$ for $G \in \MF(X \times Y, W*V),$
is closed under shifts, cones of 
closed degree zero morphisms and under taking homotopy equivalent
objects. It is strongly pretriangulated.

\begin{proposition}
  \label{p:cech-object-enhancement-product-MF}
  The dg category $\MF_\Cechobjbox(X \times Y, W*V)$
  is naturally an enhancement of $\bfMF(X \times Y, W*V).$ 
  In fact, it is equal to the enhancement 
  $\MF_\Cechobj(X \times Y, W*V).$
  % % i.\,e.\ 
  % % the canonical functor 
  % % \begin{equation*}
  % %   \epsilon \colon  [\mfPerf_\Cechobjbox(X \times Y)] \ra \mfPerf(X
  % %   \times Y)
  % % \end{equation*}
  % % is an equivalence of triangulated categories.
  % We call it the \define{generalized object oriented \v{C}ech
  %   enhancement}.
  % % of $\bfMF(X \times Y, W*V).$ 
\end{proposition}

\begin{proof}
  Use the techniques of proof from
  Proposition~\ref{p:cech-object-enhancement-MF}
  and
  \cite[Prop.~\ref{enhance:p:cech-*-object-enhancement-product},
  Cor.~\ref{enhance:c:cech-*-object-enhancement-product}]{valery-olaf-enhancements-in-prep}. 
\end{proof}

Consider now $X \times X$ with the morphism
$W * (-W) \colon  X \times X \ra \DA^1$ and with the
product covering $\mathcal{U} \times \mathcal{U},$ and assume
that
$X \times X$ satisfies
condition~\ref{enum:srNfKd}.
Let $\Delta \colon  X \ra X \times X$ be the diagonal inclusion.
Note that $\Delta^*(W*(-W))=0$ so that the dg functor 
$\Delta_* \colon  \Qcoh(X,0) \ra \Qcoh(X \times X, W*(-W))$ is
well-defined. 

\begin{lemma}
  \label{l:boxtimes-cech-object-to-diagonal-cech-K-iso-D-MF}
  Let $E \in \MF(X, W),$ $F \in \MF(X, -W),$ $G \in
  \MF(X,0),$ and let $m \in \DZ_2.$
  Then the canonical map
  \begin{multline*}
    \Hom_{[\Qcoh(X \times X, W*(-W))]}(\mathcal{C}_\ord(E) \boxtimes
      \mathcal{C}_\ord(F),
      [m]\Delta_*(\mathcal{C}_\ord(G)))\\
    \ra
    \Hom_{\DQcoh(X \times X, W*(-W))}(\mathcal{C}_\ord(E) \boxtimes
      \mathcal{C}_\ord(F),
      [m]\Delta_*(\mathcal{C}_\ord(G)))
  \end{multline*}
  % \begin{equation*}
  %   \leftidx{_{[\Qcoh(X \times X, W*(-W))]}^m}{(\mathcal{C}_\ord(E) \boxtimes
  %     \mathcal{C}_\ord(F),
  %     \Delta_*(\mathcal{C}_\ord(G)))}{}
  %   \ra
  %   \leftidx{_{\DQcoh(X \times X, W*(-W))}^m}{(\mathcal{C}_\ord(E) \boxtimes
  %     \mathcal{C}_\ord(F),
  %     \Delta_*(\mathcal{C}_\ord(G)))}{}
  % \end{equation*}
  is an isomorphism.
  % Here we write 
  % $\leftidx{_?^m}{(-,-)}{}$ instead of $\Hom_{?}(-,[m]-).$
\end{lemma}

\begin{proof}
  Again use the above techniques and the proof of
  \cite[Lemma~\ref{enhance:l:boxtimes-cech-object-to-diagonal-cech-K-iso-D}]{valery-olaf-enhancements-in-prep}
  (note that $\bR \Delta_*=\Delta_*$ by
  \cite[Remark~\ref{semi:rem:derived-direct-image-for-affine-morphism}]{valery-olaf-matfak-semi-orth-decomp}). 
\end{proof}

We come back to the product situation $X \times Y$ with morphism
$W*V$ and covering
$\mathcal{U} \times \mathcal{V}.$
% Recall the dg bifunctor
% $(- \boxtimes -) \colon  \Qcoh(X,W) \times
% \Qcoh(Y,V) \ra \Qcoh(X \times Y, W*V)$ from \cite[section~\ref{semi:sec:extern-tens-prod}]{valery-olaf-matfak-semi-orth-decomp}.

\begin{lemma}
  \label{l:dg-functor-boxtimes-full-and-faithful-MF}
  The dg functor
  \begin{equation}
    \label{eq:dg-functor-boxtimes-MF}
    \boxtimes \colon  \MF_\Cechobj(X,W) \otimes
    \MF_\Cechobj(Y,V) \ra 
    \MF_\Cechobj(X\times Y, W*V)
  \end{equation}
  induced from $(- \boxtimes -) \colon  \Qcoh(X,W) \times
  \Qcoh(Y,V) \ra \Qcoh(X \times Y, W*V)$ 
  is quasi-fully faithful, i.\,e.\ induces quasi-isomorphisms
  between morphisms spaces.
\end{lemma}

\begin{proof}
  This is an easy generalization of
  \cite[Lemma~\ref{enhance:l:dg-functor-boxtimes-full-and-faithful}]{valery-olaf-enhancements-in-prep} since we
  can consider the graded components separately.
\end{proof}

The dg bifunctor~\eqref{eq:dg-functor-boxtimes-MF}
lifts the dg bifunctor $\boxtimes \colon 
\bfMF(X,W) \otimes \bfMF(Y,V) \ra \bfMF(X\times Y, W*V)$ of
triangulated categories
(cf.~\cite[Rem.~\ref{enhance:rem:dg-lift-of-boxtimes}]{valery-olaf-enhancements-in-prep}).

\subsubsection{Equivalence of enhancements}
\label{sec:equiv-enhanc-MF}

% As in the previous section, $\mathcal{U}$ (resp.\ $\mathcal{V}$) is an finite
% affine open covering of $X$ (resp.\ $Y$), and we consider
% the covering $\mathcal{U} \times \mathcal{V}$ on $X \times Y.$

\begin{lemma}
  \label{l:enhancements-equivalent-MF}
  % \rule{1mm}{0mm}
  % \begin{enumerate}
  % \item 
  %   \label{enum:inj-Cechobj-equivalent-MF}
  The enhancements $\InjQcoh_{\bfMF}(X,W)$
  (defined in
  \cite[section~\ref{semi:sec:enhanc-inject}]{valery-olaf-matfak-semi-orth-decomp})
  and
  $\MF_\Cechobj(X,W)$
  of $\bfMF(X,W)$ 
  are equivalent.
  % \item 
  %   \label{enum:Cechobj-box-equivalent-MF}
  %   In the product situation, the enhancements
  %   $\MF_\Cechobj(X \times Y, W*V)$
  %   and $\MF_\Cechobjbox(X \times Y, W*V)$ are equivalent.
  % \end{enumerate}
\end{lemma}

\begin{proof}
  For the first statement use the method
  of proof of
  \cite[Lemma~6.2]{bondal-larsen-lunts-grothendieck-ring} or
  \cite[Prop.~\ref{semi:p:Cech-enhancement}]{valery-olaf-matfak-semi-orth-decomp}.
  % For the second statement observe that the inclusion
  % $\MF_\Cechobj(X \times Y, W*V) \subset 
  % \MF_\Cechobjbox(X \times Y, W*V)$ is obviously a
  % quasi-equivalence.
\end{proof}

\subsubsection{Version for arbitrary curved sheaves}
\label{sec:vers-arbitr-sheav-MF}

In the following section~\ref{sec:lifting-duality-MF} 
we need a small generalization of the previous constructions and
results. 

Recall from \cite[Thm.~\ref{semi:t:big-curved-categories}]{valery-olaf-matfak-semi-orth-decomp} that
the functor $\DQcoh(X,W) \ra \DSh^\co(X,W)$ is full and
faithful and that $\InjSh(X,W)$ is naturally an enhancement of
$\DSh^\co(X,W).$
Let $\bfMF'(X,W)$ be the essential image of $\bfMF(X,W)$ under
the full and faithful functor $\bfMF(X,W) \ra \DSh^\co(X,W)$
(see \cite[Thm.~\ref{semi:t:equivalences-curved-categories}]{valery-olaf-matfak-semi-orth-decomp}); so 
$\bfMF(X,W) \ra \bfMF'(X,W)$ is an equivalence.
 
Denote by $\MF'_\Cechobj(X,W)$ the smallest full dg subcategory of
$\Sh(X,W)$ that contains all objects of $\MF_\Cechobj(X,W)$ and is
closed under taking homotopy equivalent objects.  
Then the inclusion $\MF_\Cechobj(X,W) \ra \MF'_\Cechobj(X,W)$
is a quasi-equivalence.
If we define $\MF'_\Cechobjbox(X \times Y, W*V)$ similarly
it is clear that all propositions, lemmata and remarks of
section~\ref{sec:object-orient-vcech-MF} remain true if we replace
$\MF_\Cechobj$ by $\MF'_\Cechobj,$
$\MF_\Cechobjbox$ by $\MF'_\Cechobjbox,$ $\bfMF$ by $\bfMF',$ and $\DQcoh(-,?)$ by $\DSh^\co(-,?).$
The full dg
subcategory 
$\InjSh_{\bfMF'}(X,W)$ of $\InjSh(X,W)$ consisting of objects of
$\bfMF'(X,W)$ is naturally an enhancement of
$\bfMF'(X,W),$ and the obvious variation of
Lemma~\ref{l:enhancements-equivalent-MF} is true;
in fact, all the enhancements of
$\bfMF'(X,W)$
% and $\bfMF'(X \times Y, W*V)$ 
we have defined are equivalent.  

\subsubsection{Lifting the duality}
\label{sec:lifting-duality-MF}

Recall the duality 
\begin{equation*}
  % \label{eq:duality-in-section-lifting}
  D=D_X=(-)^\cek=\sheafHom(-,\mathcal{D}) \colon  \bfMF(X,W)^\opp \ra \bfMF(X,-W)
\end{equation*}
from \cite[section~\ref{semi:sec:duality}]{valery-olaf-matfak-semi-orth-decomp} where 
$\mathcal{D}=\mathcal{D}_X=(\matfak{0}{}{\mathcal{O}_X}{}) \in
\MF(X,0).$ 
Our aim is to lift 
its extension 
\begin{equation}
  \label{eq:duality-MF-extended}
  D \colon  \bfMF'(X,W)^\opp \ra \bfMF'(X,-W)
\end{equation}
to a dg functor
$\MF'_\Cechobj(X,W) \ra \MF'_\Cechobj(X,-W)$ between the respective
enhancements. Consider the dg functor
\begin{equation*}
  % \label{eq:duality-lift-MF}
  \tildew{D}:= \sheafHom(-, \mathcal{C}_\ord(\mathcal{D})) \colon 
  \Sh(X,W)^\opp \ra \Sh(X,-W).
\end{equation*}

\begin{lemma} 
  \label{l:sheafHomCC-vs-CsheafHom-MF} 
  Let $E \in \MF(X,W)$ and consider the canonical morphism
  $\alpha \colon E \ra \mathcal{C}_\ord(E)$ 
  in $Z_0(\Sh(X,W)).$ Then the induced morphism
  % $\tildew{D}(\alpha) \colon \tildew{D}(\mathcal{C}_\ord(E)) \ra
  % \tildew{D}(E),$ explicitly given by  
  % \begin{equation*}
  %   \alpha^* \colon 
  %   \sheafHom(\mathcal{C}_\ord(E),\mathcal{C}_\ord(\mathcal{D})) 
  %   \ra \sheafHom(E,\mathcal{C}_\ord(\mathcal{D}))=
  %   \mathcal{C}_\ord(\sheafHom(E,\mathcal{D}))  
  %   = \mathcal{C}_\ord(E^\cek),
  % \end{equation*}
  \begin{equation*}
    \tildew{D}(\alpha) \colon 
    \tildew{D}(\mathcal{C}_\ord(E)) =
    \sheafHom(\mathcal{C}_\ord(E),\mathcal{C}_\ord(\mathcal{D})) 
    \ra 
    \tildew{D}(E)
    =
    \sheafHom(E,\mathcal{C}_\ord(\mathcal{D}))
%    =\mathcal{C}_\ord(\sheafHom(E,\mathcal{D}))  
    = \mathcal{C}_\ord(E^\cek),
  \end{equation*}
  is a homotopy equivalence, i.\,e. an isomorphism in
  $[\Sh(X,-W)].$ See \cite[Rem.~\ref{enhance:rem:Cech-of-dual}]{valery-olaf-enhancements-in-prep} for the
  last identification.
\end{lemma}

\begin{proof}
  Write $\alpha^*:= [1]D(\alpha).$
  We have to show that
  $\Cone(\alpha^*)=\sheafHom(\Cone(\alpha),
  \mathcal{C}_\ord(\mathcal{D}))$  is
  contractible.
  Using the method of proof of
  \cite[Lemma~\ref{enhance:l:sheafHomCC-vs-CsheafHom}]{valery-olaf-enhancements-in-prep}
  (we can assume that $X$ is irreducible)
  we see that 
  $\Cone(\alpha^*)$ has a filtration with subquotients 
  $\Cone(\alpha^*)_I^K$
  labeled
  by pairs $(I,K)$ where   
  $I \subset S$ is a non-empty subset and $K \subset S \setminus I$
  a (possibly empty) subset, such that
  $\Cone(\alpha^*)_I^K$
  consists (if we forget some differentials) of all summands
  $\sheafHom(\leftidx{_{U_J}}{E}{},\leftidx{_{U_I}}{\mathcal{D}})$
  for $K \subset J \subset (I \cup K).$ Moreover, for fixed
  $(I,K),$ all these
  summands are isomorphic to  
  $\mathcal{H}_I^K:=\sheafHom(E_{U_{I \cup K}},\leftidx{_{U_I}}{\mathcal{D}}),$
  and  
  $\Cone(\alpha^*)_I^K$
  is isomorphic to the totalization of the augmented chain
  complex of a (non-empty) simplex $\Sigma$ with coefficients in
  $\mathcal{H}_I^K.$ By the latter we mean the complex in 
  $Z_0(\Sh(X,-W))$ that arises from tensoring the augmented
  chain complex of $\Sigma$ with the object
  $\mathcal{H}_I^K \in \Sh(X,-W).$ Since the augmented chain
  complex is homotopy equivalent to zero, the same is true for
  this complex, and then for its totalization.
\end{proof}

% \begin{corollary}
%   \label{c:double-dual-homotopy-equi}
%   The complex
%   $\sheafHom(\sheafHom(\mathcal{C}_\ord(E),
%   \mathcal{C}_\ord(\mathcal{D})),\mathcal{C}_\ord(\mathcal{D}))$ is
%   homotopy equivalent to $\mathcal{C}_\ord(E^{\cek
%     \cek})=\mathcal{C}_\ord(E).$ 
% \end{corollary}

% \begin{proof}
% %  By Lemma~\ref{l:sheafHomCC-vs-CsheafHom} 
%   The complex
%   $\sheafHom(\mathcal{C}_\ord(E),\mathcal{C}_\ord(\mathcal{D}))$
%   is homotopy equivalent to 
%   $\mathcal{C}_\ord(E^\cek),$ hence
%   $\sheafHom(\sheafHom(\mathcal{C}_\ord(E),
%   \mathcal{C}_\ord(\mathcal{D})),\mathcal{C}_\ord(\mathcal{D}))$ is
%   homotopy equivalent to
%   $\sheafHom(\mathcal{C}_\ord(E^\cek),\mathcal{C}_\ord(\mathcal{D}))$,
%   which 
%   by the same lemma is homotopy equivalent to
%   $\mathcal{C}_\ord(E^{\cek\cek})=\mathcal{C}_\ord(E).$
% \end{proof}

\begin{corollary}
  \label{c:duality-lift-well-defined-MF}
  The dg functor $\tildew{D}$ induces a dg functor 
  \begin{equation*}
    \tildew{D} = \sheafHom(-, \mathcal{C}_\ord(\mathcal{D})) \colon 
    \MF'_\Cechobj(X,W)^\opp \ra \MF'_\Cechobj(X,-W)
  \end{equation*}
  which lifts the duality $D$ in \eqref{eq:duality-MF-extended}.
  % in the sense that the diagram
  % \begin{equation*}
  %   \xymatrix{
  %     {[\MF'_\Cechobj(X,W)]^\opp} \ar[r]^-{[\tildew{D}]}
  %     \ar[d] &
  %     {[\MF'_\Cechobj(X,W)]}
  %     \ar[d]^-{\can} \\
  %     {\bfMF'(X,W)^\opp} \ar[r]^-{D} & {\bfMF'(X,W)}
  %   }
  % \end{equation*}
  % commutes up to an isomorphism of functors.
\end{corollary}

\begin{proof}
  Adapt the proof of
  \cite[Cor.~\ref{enhance:c:duality-lift-well-defined}]{valery-olaf-enhancements-in-prep}. 
  % Let $E$ be a bounded complex of vector bundles. 
  % Lemma~\ref{l:sheafHomCC-vs-CsheafHom} shows that   
  % $\tildew{D}(\mathcal{C}_\ord(E))$
  % % $\sheafHom(\mathcal{C}_\ord(E),
  % % \mathcal{C}_\ord(\mathcal{D}))$ 
  % is homotopy equivalent to $\mathcal{C}_\ord(E^\cek)$
  % and hence in $\MF'_\Cechobj(X,W).$
  % This implies the first claim.
  % % It follows that $\tildew{D}$ maps objects of  
  % % $\MF'_\Cechobj(X,W)^\opp$ to objects of $\MF'_\Cechobj(X,W).$
  % % Certainly  
  % % $\sheafHom(-, \mathcal{C}_\ord(\mathcal{D}))$ preserves
  % % homotopy equivalences. 
  % % This proves the first claim.
  %
  % For the second claim we need to define an isomorphism between
  % the two compositions in the above diagram. On
  % $\mathcal{C}_\ord(E)$ for $E$ as above we
  % define it to be the quasi-equivalence
  % $D(\mathcal{C}_\ord(E))=D(E)=E^\cek \ra
  % \mathcal{C}_\ord(E^\cek)$ followed by the inverse of the
  % homotopy 
  % equivalance
  % $\tildew{D}(\mathcal{C}_\ord(E)) \ra
  % \tildew{D}(E)=\mathcal{C}_\ord(E^\cek)$
  % from Lemma~\ref{l:sheafHomCC-vs-CsheafHom}.
  % This is compatible with morphisms and sufficient by
  % Remark~\ref{rem:objects-of-cech-object-enhancement}. 
\end{proof}

The canonical morphism
\begin{align}
  \label{eq:theta-F-MF}
  \theta_F  \colon  F &
  \ra
  \tildew{D}^2(F)=\sheafHom(\sheafHom(F,
  \mathcal{C}_\ord(\mathcal{D})),\mathcal{C}_\ord(\mathcal{D})),\\
  \notag
  f & \mapsto (\lambda \mapsto \lambda(f))
\end{align}
(for $F \in \Sh(X,W)$) defines a morphism
$\theta \colon  \id \ra \tildew{D}^2$ of dg functors 
$\Sh(X,W) \ra \Sh(X,W),$
and, by Corollary~\ref{c:duality-lift-well-defined-MF}, 
also of dg functors
$\MF'_\Cechobj(X,W) \ra \MF'_\Cechobj(X,W).$

\begin{lemma} 
  \label{l:map-to-double-dual-homotopy-equi-MF}
  For each $F \in \MF'_\Cechobj(X,W),$ 
  the morphism
  $\theta_F$ 
  in \eqref{eq:theta-F-MF}
  is a homotopy equivalence.
\end{lemma}

\begin{proof}
  Adapt the proof of
  \cite[Lemma~\ref{enhance:l:map-to-double-dual-homotopy-equi}]{valery-olaf-enhancements-in-prep}. Instead of
  quasi-isomorphisms we need to speak about morphisms in
  $Z_0(\Sh(X,\pm W))$ that become isomorphisms in $\DSh(X, \pm
  W).$
\end{proof}

\begin{corollary}
  \label{c:duality-lifted-to-enhancement-MF}
  The dg functor 
  % \begin{equation*}
  $\tildew{D}= \sheafHom(-, \mathcal{C}_\ord(\mathcal{D})) \colon 
  \MF'_\Cechobj(X,W)^\opp \ra \MF'_\Cechobj(X,-W)$
  % \end{equation*}
  is a quasi-equivalence. The induced functor $[\tildew{D}]$ on
  homotopy categories is an equivalence and a duality in the
  sense that the natural morphism $\theta \colon  \id \ra
  [\tildew{D}]^2$ is an isomorphism.
\end{corollary}

\begin{proof}
  Lemma~\ref{l:map-to-double-dual-homotopy-equi-MF}
  shows that $\theta \colon  \id \ra [\tildew{D}]^2$ is an isomorphism. 
  In particular, $[\tildew{D}]$ is an equivalence, and
  $\tildew{D}$ is a quasi-equivalence.
  % by Lemma~\ref{l:dg-functor-qequi-test}.
\end{proof}

\subsection{The singularity category of a function}
\label{sec:categ-sing}

We assume now and for the rest of
section~\ref{sec:thom-sebast-smoothness} 
that our field $k$ is algebraically
closed and of characteristic zero.
Let $X$ be a smooth variety, i.\,e.\ a separated smooth scheme of
finite type (over $k$), cf.\ Remark~\ref{rem:alternative-for-srNfKd}.
% (so in particular it satisfies condition~\ref{enum:srNfKd})
Let $W \colon  X\ra \DA^1$ be a morphism. 
We identify $k=\DA^1(k)$ with the set of closed points of
$\DA^1.$

\begin{definition}
  \label{d:category-of-sings}
  We define the \define{singularity category of $W$} as
  the product
  \begin{equation*}
    \bfMF(W):=\prod_{a \in k} \bfMF(X, W-a).
  \end{equation*}
\end{definition}

Note that only finitely many factors of this product are
non-zero. To show this we can assume that $X$ is connected
(see \cite[Rem.~\ref{semi:rem:X-disconnected}]{valery-olaf-matfak-semi-orth-decomp}). If
$W$ is constant, then $W=b$ for some $b \in k$ and
$\bfMF(W)=\bfMF(X, 0)$ by
\cite[Lemma~\ref{semi:l:case-W=constant-nonzero}]{valery-olaf-matfak-semi-orth-decomp}. 
Otherwise $W$ is flat and Orlov's theorem says that $\coker \colon \bfMF(X,W-a) \ra
D_\Sg(X_a)$ is an equivalence
(\cite[Thm.~\ref{semi:t:factorizations=singularity}]{valery-olaf-matfak-semi-orth-decomp})
where $X_a$ is 
the scheme theoretic fiber over $a \in k.$
% considered as a closed point in $\DA^1.$
By generic smoothness on the target 
(\cite[Cor.~III.10.7]{Hart})
% (\cite[Thm.~25.3.3]{vakil-foundations-algebraic-geometry-2013-june})
$X_a$ is smooth for all but
finitely many values $a \in k.$ If $X_a$ is smooth, then
$D_\Sg(X_a)=0.$

\begin{lemma}
  \label{l:D-Sg-W-vanishes}
  We have $\bfMF(W)=0$ if and only if $W$ is a smooth morphism.
\end{lemma}

\begin{proof}
  If $W \colon  X \ra \DA^1$ is smooth, then it is in particular flat, 
  so Orlov's equivalence 
  $\coker \colon \bfMF(X,W-a) \sira D_\Sg(X_a)$ and the fact that
  all $X_a$
  are regular show that $\bfMF(W)=0.$

  Conversely, assume that $\bfMF(W)=0.$ We can in addition assume
  that $X$ is connected and non-empty. Then $W$ is either
  constant or flat. If $W$ is constant, we obtain 
  $\bfMF(X,0)=\bfMF(W)=0.$ This is a contradiction since $\bfMF(X,0)$ obviously has
  non-zero objects
  (use \cite[Prop.~\ref{semi:p:case-W-equals-zero-acycl-equals-exact}]{valery-olaf-matfak-semi-orth-decomp}). 
  So assume that $W$ is flat. Then 
  $D_\Sg(X_a)=0$ for all $a \in k,$ so all fibers $X_a$ are (regular
  and) smooth. This together with flatness of $W$ already implies
  that $W$ is smooth (by \cite[Def.~4.3.35]{Liu}).  
\end{proof}

\begin{remark}
  \label{rem:intuition-D-Sg-W}
  As made precise by Lemma~\ref{l:D-Sg-W-vanishes},
  one may think of $\bfMF(W)$ as measuring the
  singularity of $W.$
  The above discussion implies
  that
  $\bfMF(W)$ is nonzero for a constant function $W$ (if $X
  \not= \emptyset$),
  hence a 
  constant function is considered to be singular. This would not
  be the case if we had defined $\bfMF(W)$ as the product of
  the categories $D_\Sg(X_a).$
\end{remark}

Let $\Sing(W) \subset X$ be the closed subscheme defined by the 
vanishing of the section $dW \in \Gamma(X, \Omega^1_{X/k})$
of the cotangent bundle. Its closed points are the critical
points of $W.$ Let $\Crit(W) =W(\Sing(W)(k)) \subset \DA^1(k)=k$
be the (finite) set of critical values of $W$.
The above discussion shows that
\begin{equation*}
  \bfMF(W)=\prod_{a \in \Crit(W)} \bfMF(X, W-a).
\end{equation*}
and we emphasize again that this product is finite.

Recall that we defined in \cite[section~\ref{semi:sec:enhancements}]{valery-olaf-matfak-semi-orth-decomp}
and in section~\ref{sec:object-orient-vcech-MF}
the enhancements $\InjQcoh_{\bfMF}(X, W-a),$ $\MF_{\Cechmor}(X,
W-a)),$ 
% (after fixing an affine open covering of $X$), 
$\MF(X, W-a)/\AcyclMF(X, W-a),$ $\MF_\Cechobj(X, W-a)$
% (after fixing a finite affine open covering of $X$) 
and $\MF'_\Cechobj(X,W-a)$
of $\bfMF(X, W-a)$ 
and 
showed that they are equivalent (three of these enhancements
depend on the choice of a (finite) affine open covering of $X$).  
Fix one of these enhancements and denote it by $\bfMF(X,
W-a)^\dg.$ Then
\begin{equation*}
  \bfMF(W)^\dg:=\prod_{a \in \Crit(W)} \bfMF(X, W-a)^\dg
\end{equation*}
is an enhancement of $\bfMF(W).$
Since the pretriangulated dg category $\bfMF(W)$ might not be
triangulated 
(cf.\ Lemma~\ref{l:triang-versus-pretriang-and-karoubi})
we will mainly work with its "triangulated dg envelope"
\footnote{
  If $\mathcal{A}$ and $\mathcal{B}$ are non-empty dg categories,
  scalar 
  extension along the two projections $\mathcal{A} \times
  \mathcal{B} \ra \mathcal{A}$ and $\mathcal{A} \times
  \mathcal{B} \ra \mathcal{B}$ defines an
  equivalence $\Perf(\mathcal{A} \times \mathcal{B}) \ra
  \Perf(\mathcal{A}) \times \Perf(\mathcal{B})$ of dg
  categories. This explains the second equality. 
}
\begin{equation}
  \label{eq:def-Perf-W-dg}
  \bfMF(W)^{\dg,\natural} := \Perf(\bfMF(W)^\dg)
  = \prod_{a \in \Crit(W)} \Perf(\bfMF(X, W-a)^\dg).
\end{equation}
Then $\bfMF(W)^{\dg,\natural}$ is an enhancement of the Karoubi
envelope of $\bfMF(W).$
% (the upper index $\natural$ indicates "taking direct summands"). 
Note that the quasi-equivalence class of $\bfMF(W)^{\dg,\natural}$ does not
depend on the above choices of enhancements, by
  Lemma~\ref{l:dg-functors-restricts-to-perfects-embedding-and-qequiv}.\ref{enum:qequi-induces-qequi-between-perfs}.

\begin{remark}
  \label{rem:Perf-W-Cech-concretely}
  Let us give a more concrete description of $\bfMF(W)^{\dg,\natural}$
  that we 
  will mainly use later on:
  For each $a \in \Crit(W)$
  choose an object $E(a) \in \bfMF(X, W-a)^\dg$ that becomes a
  classical generator in $[\bfMF(X,W-a)^\dg] \cong \bfMF(X, W-a)$
  (use
  \cite[Prop.~\ref{semi:p:existence-of-classical-generator-in-mf}]{valery-olaf-matfak-semi-orth-decomp}). 
  Let $A(a)$ be the endomorphism dg algebra of $E(a)$ in 
  $\bfMF(X,W-a)^\dg,$ i.\,e.\
  \begin{equation*}
    A(a)=\End_{\bfMF(X,W-a)^\dg}(E(a)).
  \end{equation*}
  Then
  $A=\prod_{a \in \Crit(W)} A(a)$ is the endomorphism dg algebra of
  $E=(E(a))$ in $\bfMF(W)^\dg.$
  Proposition~\ref{p:pretriang-classical-generator}
  yields a quasi-equivalence
  \begin{equation*}
    \Perf(A) =
    \prod_{a \in \Crit(W)} \Perf(A(a)) 
    \ra 
    \bfMF(W)^{\dg,\natural}
    % = \prod_{a \in \Crit(W)} \Perf(\MF(X, W-a)^\dg)
  \end{equation*}
  and also provides a triangulated equivalence
  \begin{equation}
    \label{eq:Karoubi-D-Sg-W-equiv-perA}
    \ol{\bfMF(W)}=
    \prod_{a \in \Crit(W)} \ol{\bfMF(X, W-a)} \sira
    \prod_{a \in \Crit(W)} \per(A(a)) = \per(A)
  \end{equation}
  where $\ol{\mathcal{T}}$ denotes the Karoubi envelope of a
  triangulated category $\mathcal{T}.$
  Moreover, it says that smoothness and properness of
  $\bfMF(W)^\dg$ (resp.\ $\bfMF(W)^{\dg,\natural}$) can be tested on $A,$
  and that $\bfMF(W)^{\dg,\natural}$ is saturated if and only if $A$ is
  smooth and proper.
\end{remark}

\subsection{Products and generators}
\label{sec:products-generators}

Let $X$ and $Y$ (and hence $X \times Y$) be smooth varieties
% (so in particular it satisfies condition~\ref{enum:srNfKd})
and let $W \colon  X\ra \DA^1$ and $V \colon  Y \ra \DA^1$ be morphisms. 
Our aim is to prove
Proposition~\ref{p:boxtimes-of-classical-generators} below. We
start with some preparations.

% \subsubsection{External tensor product and the cokernel functor}
% \label{sec:extern-tens-prod}

An object $E \in \Coh(X_0)$ can be considered as an object
$\mu(E):=\big(\matfak{0}{}{E}{}\big) \in \Coh(X,W).$ 
For flat $W$ recall the equivalence $\coker \colon  \bfMF(X,W) \ra
D_\Sg(X_0)$ from
\cite[Thm.~\ref{semi:t:factorizations=singularity}]{valery-olaf-matfak-semi-orth-decomp}.

\begin{lemma}
  [{\cite[Lemma~2.18]{lin-pomerleano}}]
  \label{l:cokern-of-MF-reso-of-mu-CohX0}
  Assume that $W \colon X \ra \DA^1$ is flat, and let $E \in \Coh(X_0).$ 
  Suppose that $P \ra \mu(E)$ is a morphism
  in $Z_0(\Coh(X,W))$ with $P \in \MF(X,W)$ and cone in
  $\Acycl[\Coh(X,W)]$ (such a morphism exists by
  \cite[Thm.~\ref{semi:l:resolutions}.\ref{semi:enum:MF-reso}]{valery-olaf-matfak-semi-orth-decomp}).
  Then there is an isomorphism $\coker(P):=\coker(p_1)
  \cong E$ in $D_\Sg(X_0).$
\end{lemma}

\begin{proof}
  We elaborate on the proof of \cite[Lemma~2.18]{lin-pomerleano}.
  From the proof of
  \cite[Prop.~1.23]{orlov-tri-cat-of-sings-and-d-branes}
  we see that there is an exact sequence (for any $l \gg 0$)
  \begin{equation}
    \label{eq:E-loc-free-transfer-to-CM}
    0 \ra E' \ra
    L^{-l+1} \ra \dots \ra L^0
    \ra E \ra 0
  \end{equation}
  in $\Coh(X_0)$ where all $L^r$ are locally free
  coherent sheaves and $E'$ is a Cohen-Macaulay sheaf (as defined
  in
  % \fotnote{ 
  %   Note that $X_0$ is Gorenstein as a (locally) complete
  %   intersection in a regular scheme. 
  %   (
  %   Habe ich von Orlov
  %   \cite[proof of Prop.~2.8]{orlov-equivalences-LG-models}
  %   abgeschrieben... Referenz?
  %   )
  %   Hence it makes sense to
  %   speak about CM-sheaves.
  % }
  \cite[Lemma-Def.~1]{orlov-mf-nonaffine-lg-models}).
  
  The proof of \cite[Thm.~3.5]{orlov-mf-nonaffine-lg-models}
  shows that there is an object $Q :=
  \big(\matfak{Q_1}{q_1}{Q_0}{q_0}\big) \in \MF(X, W)$ such that
  $\coker(q_1) = E'$ as coherent sheaves.
  % Let $R:=\big(\matfak{Q_0}{W}{Q_0}{1}\big)$
  % and note that 
  % $0 \ra [1]Q \xra{(1,-q_1)}
  % R \ra \mu(E') \ra 0$
  % is a short exact sequence in $Z_0(\Coh(X,W)).$
  % It gives rise to a triangle in $\DCoh(X,W).$
  % % Let $P':=\Tot([1]Q \xra{(1,-q_1)} R) \in \MF(X,W).$
  % % Then the obvious morphism $P' \ra \mu(E')$ in $Z_0(\Coh(X,W))$
  % % has cone in $\Acycl[\Coh(X,W)].$ 
  % % Since $R = 0$ in $[\MF(X,W)]$
  % % the obvious morphism $Q \ra P'$ is an isomorphism in
  % % $[\MF(X,W)].$ 
  % Since $R = 0$ in $[\MF(X,W)]$
  % we obtain an isomorphism $\mu(E') \cong Q$
  % % Hence we have isomorphisms $Q \sira P'
  % % \sira \mu(E')$ 
  % in $\DCoh(X,W)].$
  % % and
  % $E' = \coker(q_1) \cong \coker(p'_1)$ in $D_\Sg(X_0).$
  Let $K:=\big(\matfak{Q_1}{1}{Q_1}{W}\big)$
  and note that 
  $0 \ra K \xra{(1,q_1)}
  Q \ra \mu(E') \ra 0$
  is a short exact sequence in $Z_0(\Coh(X,W)).$
  It gives rise to a triangle in $\DCoh(X,W).$
  Since $K = 0$ in $[\MF(X,W)]$
  we obtain an isomorphism $Q  \sira \mu(E')$
  % Hence we have isomorphisms $Q \sira P'
  % \sira \mu(E')$ 
  in $\DCoh(X,W)].$
  % and
  % $E' = \coker(q_1) \cong \coker(p'_1)$ in $D_\Sg(X_0).$

  If $L \in \Coh(X_0)$ is locally free
  we claim that $\mu(L)$ vanishes in
  $\DCoh(X,W).$
  Indeed, $L$ is Cohen-Macaulay, so the above argument shows
  that there is
  an object $M \in \MF(X, W)$ such that
  $\coker(m_1) = L$ in $\Coh(X_0)$ and 
  % $\mu(L) \cong M$ in $\DCoh(X,W).$
  $M \sira \mu(L)$ in $\DCoh(X,W).$
  Since $L$ vanishes in $D_\Sg(X_0)$ we see that $M$ vanishes in
  $\bfMF(X,W)$ and a fortiori in $\DCoh(X,W).$ 

  If we apply $\mu$ to
  \eqref{eq:E-loc-free-transfer-to-CM} and use this claim for the
  $\mu(L^r)$ we see
  that $[l]\mu(E')\cong \mu(E)$ in $\DCoh(X,W).$

  By assumption we have $P \cong \mu(E)$ in $\DCoh(X,W).$
  Combined with the above isomorphisms this shows that $P \cong
  [l]Q$ in $\DCoh(X,W).$ Since both $P$ and $Q$ are in $\MF(X,W)$
  and $\bfMF(X,W) \ra \DCoh(X,W)$ is an equivalence, we have $P
  \cong [l]Q$ in $\bfMF(X,W).$ This shows that
  $\coker(p_1) \cong
  [l]\coker(q_1)
  = [l]E' \cong E$
  in $D_\Sg(X_0)$
  where we use \eqref{eq:E-loc-free-transfer-to-CM} for the last
  isomorphism.
\end{proof}

\begin{corollary}
  \label{c:cokern-and-boxtimes}
  Assume that $W \colon X \ra \DA^1$ and $V \colon  Y \ra \DA^1$ are flat.
  Let $E \in \Coh(X_0)$ and $F \in \Coh(Y_0).$
  Let $P \ra \mu(E)$ be a morphism
  in $Z_0(\Coh(X,W))$ with $P \in \MF(X,W)$ and cone in
  $\Acycl[\Coh(X,W)],$
  and let $Q \ra \mu(F)$ be a morphism
  in $Z_0(\Coh(Y,V))$ with $Q \in \MF(Y,V)$ and cone in
  $\Acycl[\Coh(Y,V)].$
  % (such morphisms exist by
  % Theorem~\semirefwithoutcite{l:resolutions}.\ref{semi:enum:MF-reso}).
  % Let
  % $T :=P \boxtimes Q \in
  % \MF(X \times Y, W*V)$ be the external tensor product of $P$ and
  % $Q.$
  Then
  \begin{equation*}
    % \coker(t_1)
    \coker(P \boxtimes Q)
    \cong 
    E \boxtimes F
  \end{equation*}
  in $D_\Sg((X \times Y)_0).$
  Here we 
  use the closed embedding $X_0 \times Y_0 \subset (X \times
  Y)_0$ 
  in order to consider $E \boxtimes F \in
  \Coh(X_0 \times Y_0)$ as an object of $\Coh((X \times Y)_0).$
\end{corollary}

\begin{proof}
  The morphism $P \boxtimes Q \ra \mu(E)
  \boxtimes \mu(F)= \mu(E \boxtimes F)$ has cone in
  $\Acycl[\Coh(X \times Y, W*V)$: it factors as $P
  \boxtimes Q \ra \mu(E) \boxtimes Q \ra \mu(E) \boxtimes
  \mu(F),$ both morphisms have cone in $\Acycl[\Coh(X \times Y,
  W*V),$ and we can use the octahedral axiom.
  Since $W*V$ is flat we can apply 
  Lemma~\ref{l:cokern-of-MF-reso-of-mu-CohX0}.
\end{proof}

% \subsubsection{Generators and products}
% \label{sec:generators-products}

% Our aim is to prove
% Proposition~\ref{p:boxtimes-of-classical-generators} below.
% We start
% with some general facts.
% and then prepare for the proof.

We need that certain categories have classical generators.

\begin{theorem}
  \label{t:generators-of-DbCoh}
  For $Z,$ $Z_1,$ $Z_2$ separated schemes of finite type, we
  have:
  \begin{enumerate}
  \item 
    \label{enum:generator-existence}
    % If $Z$ is a quasi-projective scheme, 
    The category
    $D^b(\Coh(Z))$ has a classical generator.
  \item 
    \label{enum:generators-product}
    % Let $Z_1$ and $Z_2$ be quasi-projective schemes, and let
    If
    $T_1$ and $T_2$ are classical generators of $D^b(\Coh(Z_1))$
    and $D^b(\Coh(Z_2)),$ respectively,
    then $T_1\boxtimes T_2$ is a
    classical generator of $D^b(\Coh(Z_1 \times Z_2)).$
  \end{enumerate}
\end{theorem}

\begin{proof}
  See \cite[Thm.~7.38]{rouquier-dimensions} or
  \cite[Thm.~6.3]{lunts-categorical-resolution} for the first
  statement.  The proof of
  \cite[Thm.~6.3]{lunts-categorical-resolution} shows that there
  are classical generators $S_1$ and $S_2$ of $D^b(\Coh(Z_1))$
  and $D^b(\Coh(Z_2)),$ respectively, such that $S_1 \boxtimes
  S_2$ is a classical generator of $D^b(\Coh(Z_1 \times Z_2)).$
  From $S_1 \in \thick(T_1)$ we obtain $S_1 \boxtimes S_2 \in
  \thick(T_1 \boxtimes S_2),$ so $T_1 \boxtimes S_2$ is a
  classical generator of $D^b(\Coh(Z_1 \times Z_2)).$
  Similarly, we see that $T_1\boxtimes T_2$ is a
  classical generator of $D^b(\Coh(Z_1 \times Z_2)).$
  % This easily implies the second statement (cf.\ the beginning of
  % the proof of Proposition~\ref{p:boxtimes-of-classical-generators}.
\end{proof}

If $Z$ is a locally Noetherian scheme (over our $k$) its regular
locus is open (\cite[Rem.~6.25(4)]{goertz-wedhorn-AGI}).
We equip its closed complement $Z^\sing \subset Z$ of
singular points with the unique structure of a reduced closed
subscheme of $Z.$ 

\begin{proposition}
  \label{p:classical-gen-from-class-gen-of-cohZsing}
  Let $Z$ be a scheme 
  satisfying condition
  (ELF) in 
  \cite{orlov-formal-completions-and-idempotent-completions},
  % ~\ref{enhance:enum:ELF} in
  % \cite{valery-olaf-enhancements-in-prep},
  and let $i  \colon Z^\sing
  \hra Z$ be the inclusion of the singular locus.
  Let $T \in D^b(\Coh(Z^\sing))$ be a classical generator. 
  Then
  the image of $i_*(T)$ 
  in $D_\Sg(Z)$ is a classical generator of $D_\Sg(Z).$
  % the stabilization 
  % $\stab(i_*(T))$ 
  % % of $i_*(T)$ 
  % is a classical generator of $D_\Sg(Z).$
  % Here the stabilization functor $\stab$ is the obvious functor
  % $D^b(\Coh(Z)) \ra D_\Sg(Z).$ 
\end{proposition}

\begin{proof}
  We use the notation of
  \cite{orlov-formal-completions-and-idempotent-completions}.
  The object $i_*(T)$ is a classical generator of
  $D_{Z^\sing}^b(\Coh(Z))$ (by.\
  \cite[Lemma~6.9]{lunts-categorical-resolution})
  and
  the obvious functor
  \begin{equation*}
    D^b_{Z^\sing}(\Coh(Z))/\mfPerf_{Z^\sing}(Z) \ra D_\Sg(Z)
  \end{equation*}
  is full and faithful, and dense in the sense that any object of
  $D_\Sg(Z)$ is a direct summand of an object of
  $D^b_{Z^\sing}(\Coh(Z))/\mfPerf_{Z^\sing}(Z)$ (\cite[Lemma~2.6
  and
  Prop.~2.7]{orlov-formal-completions-and-idempotent-completions}).
  These statements obviously imply that 
  % the stabilization of
  $i_*(T)$ becomes a classical generator of $D_\Sg(Z).$
\end{proof}

We come back to our setting with $W \colon  X \ra \DA^1$ and $V \colon  Y \ra
\DA^1.$ Recall that $\Sing(W) \subset X$ is the closed subscheme
defined by the vanishing of $dW.$ If $Z$ is a scheme we denote by
$|Z|$ the corresponding reduced closed subscheme.

\begin{remark}
  \label{rem:SingW-cap-Xa-versus-Xa-sing}
  Assume that $X$ is connected, and  
  let $a \in k.$
  If $W=a$ then $\Sing(W) \cap X_a=X$ and
  $(X_a)^\sing=\emptyset.$ 
  Otherwise
  % Assume that $W$ is not constant and $=a.$ Then 
  the singular
  points of $X_a$ are precisely the elements of the
  scheme-theoretic intersection $\Sing(W) \cap X_a,$ i.\,e.\ we
  have the equality
  \begin{equation}
    \label{eq:reduced-SingW-cap-Xa-equals-reduced-Xa-sing}
    |\Sing(W) \cap X_a| = (X_a)^\sing
  \end{equation}
  of varieties. This is
  trivial if $W$ is constant $\not=a,$ and otherwise it follows
  from the Jacobian criterion applied to $W-a$ (see e.\,g.\ the
  proof of \cite[Thm.~III.\S4.4]{redbook}).
\end{remark}

We obviously have
\begin{equation}
  \label{eq:Sing-W*V}
  \Sing(W*V) =   \Sing(W)\times \Sing(V).
\end{equation}
This implies that
$\Crit(W*V)=\Crit(W)+\Crit(V):=\{a+b\mid a \in
\Crit(W), b \in \Crit(V)\}.$ 

\begin{lemma}
  \label{l:Sing-convolution-is-product-of-Sings}
  % Assume that $X$ and $Y$ are connected.
  Let $c \in k.$ 
  Then
  \begin{equation*}
    |\Sing(W*V) \cap (X \times Y)_c| = \coprod_{a \in \Crit(W),\; b
      \in \Crit(V),\; a+b=c} 
    |\Sing(W) \cap X_a| \times |\Sing(V) \cap Y_b|
  \end{equation*}
  as subvarieties of $X \times Y.$ If $c \not\in \Crit(W*V)$ then
  $|\Sing(W*V) \cap (X \times Y)_c| = \emptyset.$
  % where
  % $|\Sing(W) \cap X_a| \times |\Sing(V) \cap Y_b|$ is non-empty 
  % for only finitely many pairs $(a,b) \in k^2.$
\end{lemma}

\begin{proof}
  % We use Remark \ref{rem:SingV-cap-Xa-versus-Xa-sing}.
  % Assume that $V$ is constant, say $V=a_0.$
  % If $W$ is constant, say $W=b_0,$ then
  % $V*W=a_0+b_0$ and the claim is trival.
  % If $W$ is non-constant, then $V*W$ is non-constant and 
  % both sides are equal to $X \times |(Y_{c - a_0})^\sing|.$
  % 
  % Hence we can assume that $V$ and $W$ are both non-constant.
  % Then the same is true for $V*W.$
  % 
  The set $\Crit(W) \subset k$ of critical values of $W$ is
  finite, by generic smoothness on the target. Hence $|\Sing(W)|
  = \coprod_{a \in \Crit(W)} |\Sing(W) \cap X_a|,$ and similarly
  for $V$ and $W*V.$ Hence we can rewrite both sides of \eqref{eq:Sing-W*V} and
  obtain
  \begin{equation*}
    \coprod_{c \in \Crit(W*V)}|\Sing(W*V) \cap (X \times Y)_c| 
%    & 
    =
    \coprod_{a \in \Crit(W),\; b  \in \Crit(V)} 
    |\Sing(W) \cap X_a| \times |\Sing(V) \cap Y_b|\\
%    & =
%    \coprod_{a \in \Crit(W),\; b  \in \Crit(V)} 
%    |\Sing(W*V) \cap (X_a \times Y_b)|.
  \end{equation*}
  These statements imply the lemma.
% yields
  % \begin{equation*}
  %   \coprod_{c \in \Crit(V*W)}|\Sing(V*W) \cap (X \times Y)_c| 
  %   = \coprod_{a \in \Crit(V),\; b \in \Crit(W)}
  %   |\Sing(V) \cap X_a| \times |\Sing(W) \cap Y_b|.
  % \end{equation*}
  % This certainly implies the above equality.
\end{proof}

% Recall that the dg bifunctor
% $\boxtimes \colon  \MF(X,W) \times \MF(Y,V) \ra \MF(X \times Y, W*V)$
% induces the bifunctor
% \begin{equation}
%   \label{eq:boxtimes-bfMF}
%   \boxtimes \colon  \bfMF(X,W) \times \bfMF(Y,V) \ra \bfMF(X \times Y,
%   W*V). 
% \end{equation}
% % between the Verdier quotients of the corresponding homotopy
% % categories.

The functor
$\boxtimes \colon  \bfMF(X,W-a) \times \bfMF(Y,V-b) \ra \bfMF(X \times Y, W*V-a-b)$
gives rise to the functor
$\boxtimes \colon  \bfMF(W) \times \bfMF(V) \ra \bfMF(W*V)$ defined by
\begin{equation}
  \label{eq:boxtimes-on-DSgW}
  E \boxtimes F :=
  \Big(\bigoplus_{a+b=c} E(a) \boxtimes F(b)\Big)_{c \in k}
\end{equation}
for $E=(E(a))_{a \in k}$ and
$F=(F(b))_{b \in k}.$ Note that only finitely many of the objects
$E(a)\boxtimes F(b)$ are non-zero.

Obviously, an object $E=(E(a))_{a \in k} \in \bfMF(W)$ is a classical
generator if and only if $E(a) \in \bfMF(X, W-a)$ is
a classical
generator for the finitely many critical values $a$ of $W.$

\begin{proposition}
  \label{p:boxtimes-of-classical-generators}
  Let $E \in \bfMF(W)$ and
  $F \in \bfMF(V)$ be classical generators. Then $E \boxtimes F$
  is a classical generator of $\bfMF(W*V).$
\end{proposition}

\begin{proof}
  Observe first that it is enough to prove the result for
  suitably chosen 
  classical generators $E$ and $F,$ see the end of the proof of
  Theorem~\ref{t:generators-of-DbCoh}.
  % Indeed, assume $E' \in D_\Sg(X)$ is a classical generator.
  % Then $E(a) \in \thick(E'(a)),$ hence $E(a) \boxtimes F(b) \in
  % \thick(E'(a) \boxtimes F(b)).$
  % This implies that (for fixed $c \in k$) the thick envelope of
  % $\bigoplus_{a+b=c} E'(a) \boxtimes F(b)$ in $\bfMF(X \times Y,
  % W*V-c)$ contains the thick envelope
  % of $\bigoplus_{a+b=c} E(a) \boxtimes F(b).$
  % Hence if
  % $\bigoplus_{a+b=c} E(a) \boxtimes F(b)$ is a classical generator,
  % so is
  % $\bigoplus_{a+b=c} E'(a) \boxtimes F(b).$ Similarly we can replace
  % $F$ by another classical generator $F'.$

  It is certainly enough to prove the proposition under the
  additional assumption that both $X$ and $Y$ are connected
  (cf.\ \cite[Rem.~\ref{semi:rem:X-disconnected}]{valery-olaf-matfak-semi-orth-decomp}). 
  Then $W$ is either constant or flat, and similarly for $V.$  

  \textbf{Case 1:} Both $W$ and $V$ are flat.

  \textbf{Step~1:}
  Fix a critical value $a \in \Crit(W)$ of $W.$
  Let $S_a \in D^b(\Coh((X_{a})^\sing))$ be
  a classical
  generator
  (which exists by Theorem~\ref{t:generators-of-DbCoh}.\ref{enum:generator-existence}).
  By replacing $S_a$ by the direct sum of its cohomologies we can and
  will assume that $S_a \in \Coh((X_{a})^\sing).$
  Let $s_a \colon  (X_{a})^\sing\hra X_{a}$ be the closed embedding.
  By 
  \cite[Thm.~\ref{semi:l:resolutions}.\ref{semi:enum:MF-reso}]{valery-olaf-matfak-semi-orth-decomp}
  there is an object $E(a) \in \MF(X,W-a)$ together with a
  morphism 
  $E(a) \ra \mu(s_{a*}(S_a))$ 
  in $Z_0(\Coh(X,W-a))$ whose cone is in
  $\Acycl[\Coh(X,W-a)].$ 
  By Lemma~\ref{l:cokern-of-MF-reso-of-mu-CohX0}
  we have $\coker(E(a))\cong s_{a*}(S_a)$ in $D_\Sg(X_a).$
  % Since the image of $s_{a*}(S_a)$ in $D_\Sg(X_a)$ is a classical
  % generator of $D_\Sg(X_a)$
  Proposition~\ref{p:classical-gen-from-class-gen-of-cohZsing}
  then shows that
  $E(a)$ is a classical generator of $\bfMF(X, W-a).$
  Letting $a$ vary we see that $E:=(E(a))_{a \in \Crit(W)}$ is a
  classical generator of $\bfMF(W).$

  \textbf{Step~2:}
  Similarly we find for each $b \in \Crit(V)$ an object $T_b \in
  \Coh((Y_{b})^\sing)$ 
  that is a classical generator of $D^b(\Coh((Y_{b})^\sing))$ and
  then $F(b) \in \MF(Y,V-b)$ 
  together with a
  morphism 
  $F(b) \ra \mu(t_{b*}(T_b))$ 
  in $Z_0(\Coh(Y,V-b))$ whose cone is in
  $\Acycl[\Coh(Y,V-b)]$ 
  such that
  $\coker(F(b)) \cong t_{b*}(T_b)$ in $D_\Sg(Y_b)$ where 
  $t_b \colon (Y_{b})^\sing \hra Y_{b}.$ 
  Then $F:=(F(b))_{b \in \Crit(V)}$ is the
  classical generator of $\bfMF(V)$ we will consider.

  \textbf{Step~3:}
  Fix $c \in \Crit(W*V)$ a critical value of $W*V.$
  Theorem~\ref{t:generators-of-DbCoh}.\ref{enum:generators-product},
  Lemma~\ref{l:Sing-convolution-is-product-of-Sings}
  and
  equation~\eqref{eq:reduced-SingW-cap-Xa-equals-reduced-Xa-sing}
  in Remark~\ref{rem:SingW-cap-Xa-versus-Xa-sing}, and
  Proposition~\ref{p:classical-gen-from-class-gen-of-cohZsing}
  imply that the image of 
  \begin{equation*}
    \bigoplus_{a \in \Crit(W),\; b \in \Crit(V),\; a+b=c} 
    s_{a*}(S_a) \boxtimes t_{b*}(T_b)  
  \end{equation*}
  in $D_\Sg((X \times Y)_c)$ is a classical generator
  of $D_\Sg((X \times Y)_c).$ 
  But Corollary~\ref{c:cokern-and-boxtimes}
  shows that this object is isomorphic to $\coker((E \boxtimes
  F)(c))$ in $D_\Sg((X \times Y)_c).$ 

  Hence $E \boxtimes F$ is a classical generator of $\bfMF(W*V).$
  This proves the proposition if both
  $W$ and $V$ are flat.

  \textbf{Case 2:} Precisely one of $W,$ $V$ is flat.

  Without loss of generality assume that $W$ is
  flat and that $b_0 :=V \in k.$
  Then $\bfMF(V)=\bfMF(Y,0)$ by
  \cite[Lemma~\ref{semi:l:case-W=constant-nonzero}]{valery-olaf-matfak-semi-orth-decomp},
  and
  \cite[Rem.~\ref{semi:rem:classical-generator-in-mf-zero}]{valery-olaf-matfak-semi-orth-decomp}
  shows that there is a vector bundle $Q$ on $Y$ that is a
  classical generator of $D^b(\Coh(Y))$ such that
  $\mu(Q)=(\matfak{0}{}{Q}{})$ is a classical generator of
  $\bfMF(Y,0).$ 
  Now define $F \in \bfMF(V)$ by $F(b_0):=\mu(Q)$
  and $F(b)=0$ for all $b\not=b_0.$
  Define the classical generator $E=(E(a))$ of $\bfMF(W)$ as in
  Step 1. We have $\Sing(W*V)= \Sing(W) \times Y=\coprod_{a \in
    \Crit(W)} (X_a)^\sing \times Y$
  % $\Crit(W*V)= \Crit(W) + b_0,$
  % $(X \times Y)_c= X_{c-b_0} \times Y,$
  and $((X \times Y)_c)^\sing = (X_{c-b_0})^\sing \times Y.$
  Adjusting the above method it is easy to see that $E \boxtimes
  F$ is a classical generator of $\bfMF(W*V).$

  \textbf{Case 3:} Both $W$ and $V$ are constant.

  Then $\bfMF(W)=\bfMF(X,0),$ $\bfMF(V)=\bfMF(Y,0)$ and
  $\bfMF(W*V)=\bfMF(X \times Y, 0).$ Let $Q$ be a
  vector bundle on $Y$ as in the previous case, and let $P$ be a
  vector bundle on $X$ that generates $D^b(\Coh(X))$ classically
  and such that $\mu(P)$ is a
  classical generator of $\bfMF(X,0).$
  Theorem~\ref{t:generators-of-DbCoh}.\ref{enum:generators-product}
  (or \cite[Lemma~3.4.1, 3.1, 2.1]{bondal-vdbergh-generators}
  since $X$ and $Y$ are smooth) shows that $P \boxtimes Q$ is a
  classical generator of $D^b(\Coh(X \times Y)).$ Then
  $\mu(P \boxtimes Q)=\mu(P) \boxtimes \mu(Q)$ is 
  a classical generator of
  $\bfMF(X \times Y, 0),$ by
  \cite[Rem.~\ref{semi:rem:classical-generator-in-mf-zero}]{valery-olaf-matfak-semi-orth-decomp}. This shows
  what we need.
\end{proof}

\subsection{Thom-Sebastiani Theorem}
\label{sec:thom-sebastiani-theorem}

Note that the definition of
$\Perfotimes$ in
\eqref{eq:def-triang-tensor-product-of-dg-cats}
simplifies 
since we work over the field $k.$
We can and will assume that 
$\mathcal{A} \otimes^\bL \mathcal{B}= \mathcal{A} \otimes
\mathcal{B}$ and 
$\mathcal{A} \Perfotimes \mathcal{B} = \Perf(\mathcal{A} \otimes
\mathcal{B}).$

\begin{theorem}
  [Thom-Sebastiani Theorem]
  \label{t:thom-sebastiani}
  Let $X$ and $Y$ be smooth varieties
  % separated smooth schemes of finite type  
  with morphisms $W \colon  X\ra \DA^1$ and $V \colon  Y \ra \DA^1.$ 
  Then the two dg categories
  \begin{equation*}
    \bfMF(W)^{\dg,\natural} \Perfotimes \bfMF(V)^{\dg,\natural}
    \quad
    \text{and}
    \quad
    \bfMF(W*V)^{\dg,\natural}  
  \end{equation*}
  are quasi-equivalent.
  % \begin{remark}
  %   \label{rem:thomseb-Hmo}
  An equivalent statement is that the two dg categories
  \begin{equation*}
    \bfMF(W)^{\dg} \otimes \bfMF(V)^{\dg}
    \quad
    \text{and}
    \quad
    \bfMF(W*V)^{\dg}  
  \end{equation*}
  are Morita equivalent, i.\,e.\ isomorphic in 
  $\Hmo_\groundring.$ 
%\end{equation*}
\end{theorem}

The assertion of this theorem is not new. A proof is contained
in the preprint
\cite{preygel-thom-sebastiani-and-duality-MF-arxiv}
using higher techniques of derived algebraic geometry.
A different proof is claimed in \cite{lin-pomerleano}. 

\begin{proof}
  The equivalence of the two statements follows from
  Proposition~\ref{p:perf-to-perf-perf}
  and Lemma~\ref{l:L-otimes-respects-Morita-equivalences}.

  Fix finite affine open coverings 
  % $\mathcal{U}$ 
  of $X$ and 
  % $\mathcal{V}$ of 
  $Y$ and 
  consider
  the product covering 
  % $\mathcal{U} \times \mathcal{V}$ 
  of $X \times Y.$
  We can assume that we have used the 
  object oriented \v{C}ech enhancements
  (see Proposition~\ref{p:cech-object-enhancement-MF})
  when defining $\bfMF(W)^\dg,$ i.\,e.
  \begin{align*}
    \bfMF(W)^\dg & = \MF(W)_\Cechobj :=
    \prod_{a \in \Crit(W)}\MF_\Cechobj(X, W-a),\\
    \bfMF(W)^{\dg,\natural} & = \MF(W)_\Cechobj^\natural :=
    \prod_{a \in \Crit(W)} \Perf(\MF_\Cechobj(X, W-a)).
  \end{align*}
  Similarly we consider and define the dg categories
  $\MF_\Cechobj(Y, V-b),$ 
  $\bfMF(V)^\dg  = \MF(V)_\Cechobj$ and
  $\bfMF(V)^{\dg,\natural}  = \MF(V)_\Cechobj^\natural.$
  On $X \times Y$ we consider the 
  % generalized 
  object oriented
  \v{C}ech enhancement
  $\MF_\Cechobj(X \times Y, W*V)$ of $\bfMF(X \times Y,
  W*V)$ 
  (see
  Proposition~\ref{p:cech-object-enhancement-product-MF}) 
  and then define
  $\bfMF(W*V)^\dg  = \MF(W*V)_\Cechobj$ and
  $\bfMF(W*V)^{\dg,\natural}  = \MF(W*V)_\Cechobj^\natural$ accordingly.
  To ease the notation we abbreviate
  $\Hom_\Cechobj=\Hom_{\MF_\Cechobj(X, W-a)},$ 
  and similarly 
  for the other dg categories just mentioned.

  Let $E = (E(a))_{a \in k} \in \bfMF(W)$ be a classical
  generator.  Its canonical lift to the enhancement
  $\MF(W)_\Cechobj$ is the object $\mathcal{C}_\ord(E):=
  (\mathcal{C}_\ord(E(a)))_{a \in k}.$ As explained in
  Remark~\ref{rem:Perf-W-Cech-concretely} we obtain the dg algebra
  $A=\prod A(a)=\End_\Cechobj(\mathcal{C}_\ord(E))$ and a
  quasi-equivalence $\Perf(A) \ra \MF(W)_\Cechobj^\natural.$
  Similarly, starting from a classical generator $F \in
  \bfMF(V),$ we obtain a dg algbra $B=\prod
  B(b)=\End_\Cechobj(\mathcal{C}_\ord(F))$ and a
  quasi-equivalence $\Perf(B) \ra \MF(V)_\Cechobj^\natural.$
  Proposition~\ref{p:triangulated-tensor-product-of-cats-and-their-Perfs}
  and Lemma~\ref{l:Perfotimes-respects-qequivalences} 
  then
  provide quasi-equivalences
  \begin{equation*}
    A \Perfotimes B = \Perf(A \otimes B) \ra \Perf(A) \Perfotimes
    \Perf(B) 
    \ra 
    \MF(W)_\Cechobj^\natural \Perfotimes \MF(V)_\Cechobj^\natural.
  \end{equation*}
  
  On the other hand $E \boxtimes F$ is a classical
  generator of $\bfMF(W * V)$ by
  Proposition~\ref{p:boxtimes-of-classical-generators}.
  As its lift to the enhancement $\MF(W*V)_\Cechobj$ we use
  the object $\mathcal{C}_\ord(E) \boxtimes \mathcal{C}_\ord(F)$
  defined in the obvious manner, cf.\ \eqref{eq:boxtimes-on-DSgW}.
  % \begin{equation*}
  %   \mathcal{C}_\boxtimes(E \boxtimes F)
  %   := \mathcal{C}_\ord(E) \boxtimes \mathcal{C}_\ord(F).
  % \end{equation*}
  Let 
  $M=\prod M(c)=\End_\Cechobj(\mathcal{C}_\ord(E) \boxtimes
  \mathcal{C}_\ord(F)).$    
  Remark~\ref{rem:Perf-W-Cech-concretely} again provides a
  quasi-equivalence
  \begin{equation*}
    \Perf(M) 
    \ra \MF(W*V)_\Cechobj^\natural.
  \end{equation*}

  By Lemma~\ref{l:dg-functors-restricts-to-perfects-embedding-and-qequiv}.\ref{enum:qequi-induces-qequi-between-perfs}
  it is hence sufficient to show that there is a quasi-isomorphism
  \begin{equation}
    \label{eq:otimes-Cech-to-Cech-boxtimes}
    A \otimes B \ra M
  \end{equation}
  of dg algebras.
  For $c \in \Crit(W*V)$ the dg algebra
  $M(c)$ is a matrix algebra in the sense that
  \begin{multline*}
    M(c)=\End_\Cechobj((\mathcal{C}_\ord(E) \boxtimes
    \mathcal{C}_\ord(F))(c))\\
    = 
    \bigoplus_{a+b=c,\; a'+b'=c}
    \Hom_\Cechobj(\mathcal{C}_\ord(E(a)) \boxtimes
    \mathcal{C}_\ord(F(b)), \mathcal{C}_\ord(E(a')) \boxtimes
    \mathcal{C}_\ord(F(b'))),  
  \end{multline*}
  where the (finite) direct sum is taken over all $a, a' \in
  \Crit(W)$ 
  and $b, b' \in 
  \Crit(V)$ satisfying the given condition.
  Note that $A \otimes B = \prod_{c \in \Crit(W*V)} (A \otimes
  B)(c)$
  where $(A \otimes B)(c)$ is defined by
  \begin{equation*}
    (A \otimes B)(c) := \prod_{a + b=c} A(a) \otimes B(b)
    =  \prod_{a + b=c} \End_\Cechobj(\mathcal{C}_\ord(E(a))) 
    \otimes \End_\Cechobj(\mathcal{C}_\ord(F(b))).
  \end{equation*}
  We define 
  the morphism \eqref{eq:otimes-Cech-to-Cech-boxtimes}
  of dg algebras using 
  Lemma~\ref{l:dg-functor-boxtimes-full-and-faithful-MF}.
  This lemma then says that $(A \otimes B)(c)$ goes
  quasi-isomorphically (even isomorphically, by inspection of the
  proof) onto the diagonal subalgebra of $M(c).$  
  Hence we need to show that the off diagonal part of each $M(c)$
  is acyclic.

  Let $a,a' \in \Crit(W)$ and $b,b'
  \in \Crit(V)$ and assume that $a+b=c=a'+b'$ but $a\not=a'$ (and
  hence $b
  \not= b'$).  
  We need to prove that both cohomologies of
  $\Hom_\Cechobj(\mathcal{C}_\ord(E(a)) \boxtimes
  \mathcal{C}_\ord(F(b)), \mathcal{C}_\ord(E(a')) \boxtimes
  \mathcal{C}_\ord(F(b')))$ are zero. Equivalently we need to
  show that 
  \begin{equation*}
    \Hom_{\bfMF(X \times Y, W*V-c)}(E(a) \boxtimes F(b), [p]E(a') \boxtimes
    F(b'))
  \end{equation*}
  is zero for $p \in \DZ_2.$
  We can use the morphism oriented \v{C}ech enhancement
  $\MF_\Cechmor(X \times Y, W*V-c)$ 
  (see
  \cite[Prop.~\ref{semi:p:Cech-enhancement}]{valery-olaf-matfak-semi-orth-decomp})
  for this and need to show
  that both 
  cohomologies of
  \begin{equation*}
    \Hom_{\MF_\Cechmor(X \times Y, W*V-c)}(E(a) \boxtimes F(b),
    E(a') \boxtimes F(b'))
    = 
    C(\mathcal{U} \times \mathcal{V}, \sheafHom (E(a) \boxtimes
    F(b), E(a') \boxtimes F(b'))
  \end{equation*}
  vanish.
  It is certainly sufficient to show that the object
  $\sheafHom(E(a) \boxtimes F(b), E(a') \boxtimes F(b'))$ of
  $\MF(X \times Y, 0)$ is zero in
  $[\MF(X \times Y, 0)]$
  (use for example
  \cite[Lemma~\ref{semi:l:qiso-induces-Cech-qiso}]{valery-olaf-matfak-semi-orth-decomp}). 

  We have $\sheafHom(E(a), E(a')) \in \MF(X,
  a-a')$ and
  $\sheafHom(F(b), F(b')) \in \MF(Y, b-b'),$
  cf.\ 
  \cite[section~\ref{semi:sec:sheaf-hom-tensor-product}]{valery-olaf-matfak-semi-orth-decomp}.
  The $\boxtimes$-product of these two objects is then in 
  $\MF(X \times Y, 0)$ and the obvious closed degree zero
  morphism
  \begin{equation}
    \label{eq:boxtimes-sheafhom-thomseb}
    \boxtimes \colon  
    \sheafHom(E(a), E(a')) \boxtimes
    \sheafHom(F(b), F(b')) \ra
    \sheafHom(E(a)\boxtimes F(b), E(a')\boxtimes F(b'))
  \end{equation}
  is an isomorphism:
  this can be checked componentwise and locally 
  on $\Spec R \subset X$ and $\Spec S \subset Y$
  and boils down to
  the fact that the obvious map 
  \begin{equation*}
    \Hom_R(M,M') \otimes \Hom_S(N,N') \ra \Hom_{R \otimes S}(M \otimes
    N, M' \otimes N')
  \end{equation*}
  is an isomorpism for
  $M, M' \in \Mod(R)$ and $N, N' \in \Mod(S)$ 
  with $M$ and $N$ finitely generated projective.

  Since we assume that $a \not= a',$
  \cite[Lemma~\ref{semi:l:case-W=constant-nonzero}]{valery-olaf-matfak-semi-orth-decomp}
  shows that 
  $\sheafHom(E(a), E(a'))=0$ in $[\MF(X, a-a'].$
  We then see from \eqref{eq:boxtimes-sheafhom-thomseb}
  that $\sheafHom(E(a)\boxtimes F(b), E(a')\boxtimes F(b'))=0$
  in $[\MF(X \times Y, 0)].$
\end{proof}

\subsection{Smoothness}
\label{sec:smoothness}

% \subsection{Smoothness of matrix factorizations}
% % for a smooth scheme}
% \label{sec:smoothn-matr-fact}

Theorem~\ref{t:MF-DSgW-Cechobj-smooth} below is the analog of
\cite[Cor.~\ref{enhance:c:smooth-quasi-projective}]{valery-olaf-enhancements-in-prep} 
% Theorem~\ref{t:mfPerf-Cechobj-smooth} 
for matrix factorizations,
and we use the same strategy of proof.
%Our situation here slightly easier since we assume from the beginnin%g that
%$X$ regular....}
%\fotnote{
% Eigentlich beweis von
% Thm.~\ref{enhance:t:mfPerf-Cechobj-smooth-vs-diagonal-sheaf-perfect}
% in \cite{valery-olaf-enhancements-in-prep}
%}
% Therefore\fotnote{
%   maybe write this in the next section
% } 
% we assume that the reader is familiar with the
% notation and arguments explained in  
% \cite{valery-olaf-enhancements-in-prep}.
% % Appendix~\ref{sec:smoothness-dg-enhancement-smooth-scheme}.

% The following theorem is the analog of
% \cite[Cor.~\ref{enhance:c:mfPerf-Cechobj-smooth-vs-diagonal-sheaf-perfect}]{valery-olaf-enhancements-in-prep};
% we use the same ideas to prove it.\fotnote{
% situation slightly easier since we assume from the beginning that
% $X$ regular....}
% % is section is parallel to section~\ref{sec:smoothn-enhanc}.

\begin{theorem} 
  \label{t:MF-DSgW-Cechobj-smooth}
  Let $X$ be a smooth variety 
  % separated smooth scheme of finite type
  with a morphism $W \colon  X \ra \DA^1.$
  Then the dg categories $\bfMF(W)^\dg$ and $\bfMF(W)^{\dg,\natural}$ are
  % (homologically) 
  smooth over $k.$
\end{theorem}

\begin{proof}
  Recall that smoothness is invariant under quasi-equivalence.
  We proceed as in the beginning of the proof of
  the Thom-Sebastiani Theorem~\ref{t:thom-sebastiani},
  but we use the enhancements $\MF'_\Cechobj(X, W-a)$
  (see section~\ref{sec:vers-arbitr-sheav-MF}).
  The reason is that the
  duality $D=(-)^\cek \colon  \MF_\Cechobj(X, W-a)^\opp \ra \MF_\Cechobj(X,
  -W+a)$ can then be lifted to the dg functor 
  $\tildew{D} \colon  \MF'_\Cechobj(X, W-a)^\opp \ra \MF'_\Cechobj(X,
  -W+a),$ see Corollary~\ref{c:duality-lift-well-defined-MF}.
  So we assume that 
  $\bfMF(W)^\dg=\MF(W)'_\Cechobj:=\prod_{a \in k}
  \MF'_\Cechobj(X, W-a)$ and 
  $\bfMF(W)^{\dg,\natural} = {\MF(W)'}_\Cechobj^\natural :=
  \prod
  \Perf(\MF'_\Cechobj(X, W-a)).$ 
  It is clear how to extend the duality $D$ and its lift
  $\tildew{D}$ to $D \colon  \bfMF(W)^\opp \ra \bfMF(-W)$ 
  and $\tildew{D} \colon  (\MF(W)'_\Cechobj)^\opp \ra \MF(-W)'_\Cechobj,$
  respectively. 

  Let $E=(E(a))_{a \in k}$ be a classical generator of
  $\bfMF(W)$ and consider the dg algebra $A=\prod
  A(a)=\End_\Cechobj(\mathcal{C}_\ord(E)).$
  Here we abbreviate $\End_\Cechobj=\End_{\MF(W)'_\Cechobj}$
  and use similar notation in the following.
  By Remark~\ref{rem:Perf-W-Cech-concretely}
  it is enough to prove that $A$ is smooth,
  i.\,e.\ that $A \in \per(A\otimes A^{\opp}).$
  
  Since $E^\cek$ is a classical
  generator of $\bfMF(-W),$
  Proposition~\ref{p:boxtimes-of-classical-generators}
  says that $E \boxtimes E^\cek$ is a classical generator of
  $\bfMF(W*(-W)).$ 
  We will use the lift
  \begin{equation*}
    P:=\mathcal{C}_\ord(E) \boxtimes
    \tildew{D}(\mathcal{C}_\ord(E))
  \end{equation*}
  of this generator to the
  enhancement $\MF(W*(-W))'_\Cechobj,$ cf.\
  Lemma~\ref{l:sheafHomCC-vs-CsheafHom-MF}. Here
  $\MF(W*(-W))'_\Cechobj$ 
  is defined in the obvious way using the product covering of $X
  \times X.$  
  Note that
  \begin{equation*}
    \tildew{D} \colon  A^\opp =
    \End_\Cechobj(\mathcal{C}_\ord(E))^\opp
    \ra
    \End_\Cechobj(\tildew{D}(\mathcal{C}_\ord(E)))
  \end{equation*}
  is a quasi-isomorphism of dg algebras
  % since $\tildew{D}$ is a quasi-equivalence (see
  by Corollary~\ref{c:duality-lifted-to-enhancement-MF}. 
  Recall that we showed in the proof of the Thom-Sebastiani
  Theorem~\ref{t:thom-sebastiani}
  that the natural morphism \eqref{eq:otimes-Cech-to-Cech-boxtimes}
  is a quasi-isomorphism. Transferred to our setting this means
  that the morphism
  \begin{equation*}
    \boxtimes \colon 
    \End_\Cechobj(\mathcal{C}_\ord(E))
    \otimes
    \End_\Cechobj(\mathcal{C}_\ord(E^\cek))
    \ra
    \End_\Cechobj(\mathcal{C}_\ord(E) \boxtimes
    \mathcal{C}_\ord(E^\cek))
    % \xra{\boxtimes}
    % \End_\Cechobjbox(\mathcal{C}_\ord(E) \boxtimes
    % \tildew{D}(\mathcal{C}_\ord(E)))
  \end{equation*}
  is a quasi-isomorphism of dg algebras.  We also have the
  isomorphism $\tildew{D}(\mathcal{C}_\ord(E)) \sira
  \mathcal{C}_\ord(E^\cek)$ in $\prod_{a \in k}[\Sh(X, -W+a)]$
  from Lemma~\ref{l:sheafHomCC-vs-CsheafHom-MF}. These
  three facts (and the fact that $\otimes$ preserves
  quasi-isomorphisms) show that both arrows in
  \begin{equation*}
    A \otimes A^\opp
    \xra{\id \otimes \tildew{D}}
    \End_\Cechobj(\mathcal{C}_\ord(E))
    \otimes
    \End_\Cechobj(\tildew{D}(\mathcal{C}_\ord(E)))
    \xra{\boxtimes}
    \End_\Cechobj(P)
    % \xra{\boxtimes}
    % \End_\Cechobjbox(\mathcal{C}_\ord(E) \boxtimes
    % \tildew{D}(\mathcal{C}_\ord(E)))
  \end{equation*}
  are quasi-isomorphism
  of dg algebras.
  Restriction of dg modules along their composition defines an
  equivalence of the corresponding perfect derived categories; combined with
  Proposition~\ref{p:pretriang-classical-generator} 
  we obtain a full and faithful functor 
  \begin{equation*}
    % \label{eq:equivalence-F-MF}
    F:=\Hom_\Cechobj(P,-) \colon 
    [\MF(W *(-W))'_\Cechobj]
    \ra \per(A \otimes A^\opp)
  \end{equation*}
  of triangulated categories.
  Note that $A \otimes A^\opp= \prod_{a, a' \in k} A(a)
  \otimes A(a')^\opp$ and $\per(A \otimes A^\opp)=
  \prod_{a, a' \in k} \per (A(a) \otimes A(a')^\opp).$ 
  Under this identification, the $(a, a')$-component of $F$ is
  given by 
  \begin{equation*}
    F_{a,a'}=\Hom_\Cechobj(P_{a,a'}, -) \colon 
    [\MF'_\Cechobj(X \times X, W *(-W)-a+a')]
    \ra \per(A(a) \otimes A(a')^\opp)
  \end{equation*}
  where $P_{a,a'}:=\mathcal{C}_\ord(E(a)) \boxtimes \tildew{D}(\mathcal{C}_\ord(E(a'))).$
  We also see that smoothness of $A$ is equivalent to 
  smoothness of all $A(a),$ for $a \in k.$
  % the condition 
  % % $A \in \per(A \otimes A^\opp)$ is equivalent to 
  % $A(a) \in \per(A(a) \otimes A(a)^\opp)$ for all $a \in k.$

  Let $\Delta \colon  X \ra X \times X$ be the diagonal embedding.  Note
  that $\Delta^*(W * (-W))=0.$ Consider the object
  $\mathcal{D}=\mathcal{D}_X=(\matfak{0}{}{\mathcal{O}_X}{}) \in
  \bfMF(X,0)$ 
  and the canonical morphism $\mathcal{D} \ra
  \mathcal{C}_\ord(\mathcal{D})$ in $\Qcoh(X,0)$ that becomes an
  isomorphism in $\DQcoh(X,0).$ Since $\Delta$ is affine and
  proper, $\Delta_*(\mathcal{D}) \ra
  \Delta_*(\mathcal{C}_\ord(\mathcal{D}))$ in $\Qcoh(X \times X,
  W*(-W))$ becomes an isomorphism in $\DQcoh(X \times X,
  W*(-W)),$ and $\Delta_*(\mathcal{D})$ is in $\bfMF'(X \times X,
  W*(-W))$ (cf.\
  \cite[Rem.~\ref{semi:rem:derived-direct-image-for-affine-morphism}
  and
  Lemma~\ref{semi:l:pi-proper-Rpi-lower-star-and-MF-and-adjunction}]{valery-olaf-matfak-semi-orth-decomp}).
  
  Find 
  $I \in \InjQcoh(X \times X, W*(-W))$ 
  and 
  $T \in \MF(X \times X, W*(-W))$  
  together with morphisms 
  $\Delta_*(\mathcal{C}_\ord(\mathcal{D})) \ra I$ 
  and
  $T \ra I$ 
  in $Z_0(\Qcoh(X \times X, W*(-W)))$ that become
  isomorphisms in $\DQcoh(X \times X, W*(-W)).$
  These morphisms induce quasi-morphisms 
  % \begin{multline}
  %   \label{eq:identify-diagonal}
  %   F_{a,a}(\mathcal{C}_\ord(T))=
  %   \Hom_{\Sh(X \times X, W*(-W))}(P_{a,a},
  %   \mathcal{C}_\ord(T))
  %   \ra 
  %   \Hom_{\Sh(X \times X, W*(-W))}(P_{a,a},I)\\
  %   \notag
  %   \la
  %   \Hom_{\Sh(X \times X, W*(-W))}(P_{a,a},
  %   \Delta_*(\mathcal{C}_\ord(\mathcal{D})))
  % \end{multline}
  \begin{equation*}
    \label{eq:identify-diagonal}
    F_{a,a}(\mathcal{C}_\ord(T))
    % =
    % \Hom_{\Sh(X \times X, W*(-W))}(P_{a,a},
    % \mathcal{C}_\ord(T))
    \ra 
    \Hom_{\Sh(X \times X, W*(-W))}(P_{a,a},I)
    \la
    \Hom_{\Sh(X \times X, W*(-W))}(P_{a,a},
    \Delta_*(\mathcal{C}_\ord(\mathcal{D})))
  \end{equation*}
  of dg $A(a) \otimes A(a)^\opp$-modules (for $a \in k$):
  this is proved using
  the version of
  Proposition~\ref{p:cech-object-enhancement-product-MF}
  explained in section~\ref{sec:vers-arbitr-sheav-MF},
  Lemma~\ref{l:sheafHomCC-vs-CsheafHom-MF} and
  Lemma~\ref{l:boxtimes-cech-object-to-diagonal-cech-K-iso-D-MF};
  and 
  \cite[Thm.~\ref{semi:t:big-curved-categories} and
  Remark~\ref{semi:rem:morphisms-to-injectives}]{valery-olaf-matfak-semi-orth-decomp}).
  These three dg modules are
  in $\per(A(a) \otimes A(a)^\opp)$ 
  since
  $F(\mathcal{C}_\ord(T)) \in \per(A \otimes A^\opp)$
  and hence $F_{a,a}(\mathcal{C}_\ord(T)) \in \per(A(a) \otimes
  A(a)^\opp).$ 

  Observe that the obvious adjunctions provide isomorphisms
  of dg $A(a) \otimes A(a)^\opp$-modules
  \begin{align*}
    \Hom_{\Sh(X \times X, W*(-W))}(P_{a,a},&
    \Delta_*(\mathcal{C}_\ord(\mathcal{D})))\\
    & \sira
    \Hom_{{\Sh(X,0)}}(\Delta^*(\mathcal{C}_\ord(E(a)) \boxtimes
    \tildew{D}(\mathcal{C}_\ord(E(a)))),
    \mathcal{C}_\ord(\mathcal{D}))\\ 
    & =
    \Hom_{{\Sh(X,0)}}(\mathcal{C}_\ord(E(a)) \otimes
    \tildew{D}(\mathcal{C}_\ord(E(a))),
    \mathcal{C}_\ord(\mathcal{D}))\\
    & \sira
    \Hom_{{\Sh(X,0)}}(\mathcal{C}_\ord(E(a)), \sheafHom(
    \tildew{D}(\mathcal{C}_\ord(E(a))),
    \mathcal{C}_\ord(\mathcal{D})))\\
    & =
    \Hom_{{\Sh(X,0)}}(\mathcal{C}_\ord(E(a)),
    \tildew{D}^2(\mathcal{C}_\ord(E(a)))).
  \end{align*}
  Now use
  Lemma~\ref{l:map-to-double-dual-homotopy-equi-MF}.
  The canonical morphism
  $\theta=\theta_{\mathcal{C}_\ord(E(a))} \colon  \mathcal{C}_\ord(E(a)) 
  \ra
  \tildew{D}^2(\mathcal{C}_\ord(E(a)))$ is a homotopy equivalence,
  so 
  \begin{equation*}
    % (\theta_{\mathcal{C}_\ord(E(a))})_* \colon 
    \theta_* \colon 
    \Hom_{{\Sh(X,0)}}(\mathcal{C}_\ord(E(a)),
    \mathcal{C}_\ord(E(a)))
    \ra
    \Hom_{{\Sh(X,0)}}(\mathcal{C}_\ord(E(a)),
    \tildew{D}^2(\mathcal{C}_\ord(E(a))))
  \end{equation*}
  is a homotopy equivalence; moreover, it is a morphism of 
  dg $A(a) \otimes A(a)^\opp$-modules. The object on the left
  is the 
  diagonal dg $A(a) \otimes A(a)^\opp$-module $A(a)$ which is
  hence in 
  $\per(A(a) \otimes A(a)^\opp).$ This proves smoothness of
  $A(a),$ for any $a \in k.$ As observed above this just means
  that $A$ is smooth.
\end{proof}

\begin{corollary}
  \label{c:MF-DSgW-Cechobj-smooth}
  Let $X$ be a 
  smooth variety
  with a morphism $W \colon  X
  \ra \DA^1.$ Then the dg category $\bfMF(X,W)^\dg$ is smooth
  over $k.$ 
\end{corollary}

\begin{proof}
  We can assume that $\bfMF(X,W)^\dg=\MF'_\Cechobj(X,W).$
  In the proof of Theorem~\ref{t:MF-DSgW-Cechobj-smooth}
  we have seen that
  $A(0)=\End_\Cechobj(\mathcal{C}_\ord(E(0)))$ is smooth.
  This implies the claim.
\end{proof}

\subsection{Properness}
\label{sec:properness}

\begin{proposition}
  \label{p:PerfWdg-proper}
  Let $X$ be a smooth variety 
  % separated smooth scheme of finite type 
  with a
  morphism $W \colon  X \ra \DA^1,$ and assume that $W|_{\Sing(W)} \colon 
  \Sing(W) \ra \DA^1$
  is proper
  (for example $W$ could be proper),
  or equivalently, that $\Sing(W)$ is complete. 
  Then the dg
  categories $\bfMF(X,W)^\dg,$ 
  $\Perf(\bfMF(X,W)^\dg),$
  $\bfMF(W)^\dg,$ and $\bfMF(W)^{\dg,\natural}$
  are proper over $k.$
\end{proposition}

\begin{proof}
  We know that 
  $|\Sing(W)| = \coprod_{a \in \Crit(W)} |\Sing(W) \cap X_a|.$
  Hence 
  $|\Sing(W)| \ra \DA^1$ factors as
  $|\Sing(W)| \ra \Crit(W) \subset \DA^1.$
  This implies that $\Sing(W) \ra \DA^1$ is proper if and only if
  $\Sing(W) \ra \Spec k$ is proper.

  Let $E$ be a classical generator of $\bfMF(X,W)$ and
  $A$ the dg algebra of its endomorphisms in $\bfMF(X,W)^\dg.$
  It is certainly enough to show $A$ is proper
  (cf.\ Remark~\ref{rem:Perf-W-Cech-concretely}), 
  i.\,e.\ that $A \in \per(k).$
  Since $k$ is a field this just means 
  $H_l(A)=\Hom_{\bfMF(X,W)}(E, [l]E)$ is finite dimensional for $l \in
  \DZ_2.$
  We can assume that $X$ is connected, so that
  $W$ is either flat or constant.

  Assume that $W$ is flat. 
  Then we have the equivalence
  $\coker \colon  \bfMF(X,W) \sira D_\Sg(X_0)$ 
  % from 
  % \cite[Thm.~\ref{semi:t:factorizations=singularity}]{valery-olaf-matfak-semi-orth-decomp} 
  and 
  $\dim_k \Hom_{D_\Sg(X_0)}(M,N)< \infty$
  for all $M,N \in D_\Sg(X_0)$ by
  \cite[Cor.~1.24]{orlov-tri-cat-of-sings-and-d-branes}:
  note that $X_0$ is Gorenstein and that $(X_0)^\sing=|\Sing(W)
  \cap X_0|$
  (see
  equation~\eqref{eq:reduced-SingW-cap-Xa-equals-reduced-Xa-sing} 
  in Remark~\ref{rem:SingW-cap-Xa-versus-Xa-sing})
  is
  complete. This implies that $A$ is proper over $k.$

  Now assume that $W$ is constant. In case $W\not=0$ we have
  $\bfMF(X,W)=0$ by
  \cite[Lemma~\ref{semi:l:case-W=constant-nonzero}]{valery-olaf-matfak-semi-orth-decomp}
  and
  the claim is trivial. So assume $W=0.$ We can assume that
  $E=(\matfak{0}{}{P}{})$ with $P$ a vector bundle on $X$ (see
  \cite[Rem.~\ref{semi:rem:classical-generator-in-mf-zero}]{valery-olaf-matfak-semi-orth-decomp}) and that
  $\bfMF(X,0)^\dg=\MF_\Cechmor(X,0).$ 
  % is the morphism oriented
  % \v{C}ech enhancement with respect to some open affine open
  % covering $\mathcal{U}$ of $X$ (see
  % section~\ref{semi:sec:morphism-cech-enhancement}).  
  Then
  % \begin{equation*}
  $A= C(\mathcal{U}, \sheafHom(E,E))=
  C(\mathcal{U}, (\matfak{0}{}{\sheafHom(P,P)}{})),$
  % \end{equation*}
  % The $l$-th cohomology of $A$ 
  % (for $l \in \DZ_2$)
  % is 
  and hence $H_l(A)=\bigoplus H^n(X,
  \sheafHom(P,P))$ where the direct sum is over
  all $n \in \DZ$ with $n=l$ in $\DZ_2.$ We have $\dim_k H_l(A) <
  \infty$ 
  by \cite[Thm.~III.2.7]{Hart} and \cite[Thm.~3.2.1]{EGAIII-i}
  since $X$ is Noetherian of finite dimension
  and $X=\Sing(W)$ is complete.
  % proper over $k.$
  % 
  % Assume that $W$ is a projective morphism. To see that
  % $\Sing(W)$ is projective over $k$ we can assume that $X$ is
  % connected. Then either $W$ is constant and then $\Sing(W)=X,$
  % or $W$ is flat and then $\Sing(W)$ is a closed
  % subscheme\fotnote{ 
  %   correct scheme structure? wohl only closed subset...
  % } 
  % of the disjoint union of the projective varieties $X_c$ with $c
  % \in \Crit(W).$
\end{proof}

\subsection{Conclusion}
\label{sec:conclusion}

Recall the Grothendieck ring of saturated dg categories from
Proposition~\ref{p:multiplication-of-Zsat-descends}
and Definition~\ref{def:Grothendieck-group-of-saturated-dg-cats}.
Since we work here in the differential $\DZ_2$-graded setting and
over the field $k$ 
(cf.\ Remark~\ref{rem:more-general})
we denote it by $K_0(\sat_k^{\DZ_2}).$
Similarly we denote the 
monoid from Proposition~\ref{p:ulsat-as-a-commutative-monoid}
by $\ul{\sat_k^{\DZ_2}}.$

\begin{theorem}
  \label{t:combined-thomseb-smooth-proper}
  Let $X$ be a smooth variety 
  % separated smooth scheme of finite type 
  with a morphism $W \colon  X \ra \DA^1$ such that $\Sing(W)$ is
  complete (for example $W$ could be proper).
  Then 
  $\Perf(\bfMF(X,W)^\dg)$
  and
  $\bfMF(W)^{\dg,\natural}$ are saturated dg categories
  and hence
  define elements
  $\ul{\Perf(\bfMF(X,W)^\dg)}$ and
  $\ul{\bfMF(W)^{\dg,\natural}}$ of $K_0(\sat_k^{\DZ_2}).$
  If $Y$ is another smooth variety
  % separated smooth scheme of finite type 
  with a morphism
  $V \colon  Y \ra \DA^1$ such that $\Sing(V)$ is complete,
  then
  \begin{equation}
    \label{eq:PerfWPerfVPerfWV}
    \ul{\bfMF(W)^{\dg,\natural}} \bullet \ul{\bfMF(V)^{\dg,\natural}} = \ul{\bfMF(W*V)^{\dg,\natural}}
  \end{equation}
  in the 
  monoid $\ul{\sat_k^{\DZ_2}}$ and hence in the ring
  $K_0(\sat_k^{\DZ_2}).$
\end{theorem}

\begin{proof}
  The dg categories
  $\Perf(\bfMF(X,W)^\dg)$
  and
  $\bfMF(W)^{\dg,\natural}$ 
  are
  smooth, proper and triangulated, i.\,e.\ saturated,
  by Theorem~\ref{t:MF-DSgW-Cechobj-smooth},
  Corollary~\ref{c:MF-DSgW-Cechobj-smooth},
  Proposition~\ref{p:PerfWdg-proper},
  Corollary~\ref{c:triang-versus-pretriang-and-karoubi},
  and
  Lemma~\ref{l:A-and-PerfA}.
  Equality~\eqref{eq:PerfWPerfVPerfWV} is then a direct
  consequence of the Thom-Sebastiani
  Theorem~\ref{t:thom-sebastiani}.
\end{proof}

\begin{remark}
  \label{rem:groth-monoid}
  Consider the set $M$ of
  isomorphism classes
  $[X]_{\DA^1}$ of $\DA^1$-varieties $W \colon  X \ra \DA^1$ with
  $X$ smooth over $k$ and $\Sing(W)$ complete.
  If $W \colon  X \ra \DA^1$ and $V \colon X \ra
  \DA^1$ are $\DA^1$-varieties satisfying these conditions, so
  does $W * V \colon  X \times Y \ra \DA^1$
  (by equation~\eqref{eq:Sing-W*V}).
  Hence 
  $[X]_{\DA^1} \cdot [Y]_{\DA^1} := [X \times Y]_\DA^1$ turns
  $M$ into a commutative monoid with unit the class of $\Spec k
  \xra{0} \DA^1.$ One may view $M$ as a "Grothendieck monoid" 
  of certain varieties over $\DA^1.$
  Then 
  % Theorems~\ref{t:thom-sebastiani} and
  Theorem~\ref{t:combined-thomseb-smooth-proper} 
  says that mapping the class of $W \colon  X \ra \DA^1$ as above to
  $\ul{\bfMF(W)^{\dg,\natural}}$ 
  defines a (unital) morphism
  $M \ra \ul{\sat_k^{\DZ_2}}$ of monoids.
\end{remark}

\section{Landau-Ginzburg motivic measure}
\label{sec:land-ginzb-motiv}

% \section{Main results}
% \label{sec:main-results}

Let $k$ be an algebraically closed field of characteristic zero.
We continue to work in the differential $\DZ_2$-graded setting.
Our aim in this section is to prove
Theorem~\ref{t:category-of-singularities-induces-ring-morphism}.
% \subsection{A morphism of additive groups}
% \label{sec:morph-addit-groups}

Recall the Grothendieck ring $K_0(\Var_{\DA^1})$ of varieties
over $\DA^1$ from section~\ref{sec:groth-ring-vari}
and the Grothendieck ring $K_0(\sat_k^{\DZ_2})$ of saturated dg
categories from 
Proposition~\ref{p:multiplication-of-Zsat-descends}.
We first state an additive precursor
of
Theorem~\ref{t:category-of-singularities-induces-ring-morphism}
which only uses the additive structures on 
$K_0(\Var_{\DA^1})$ and $K_0(\sat_k^{\DZ_2}).$

\begin{proposition}
  \label{p:morphism-of-additive-groups}
  There is a unique morphisms
  \begin{equation*}
    % D_\Sg \colon  
    K_0(\Var_{\DA^1}) \ra K_0(\sat_k^{\DZ_2})
  \end{equation*}
  of abelian groups that maps $[X]_{\DA^1}=[X, W]$
  to $\ul{\Perf(\bfMF(X,W)^\dg)}$ whenever 
  $X$ is a smooth variety and 
  $W \colon  X \ra \DA^1$ is a
  proper morphism. 
  This morphism of abelian groups is uniquely determined by its
  values on $[X,W]$ for smooth 
  (connected) $X$ and projective $W.$
\end{proposition}

\begin{proof}
  Recall the isomorphism $K_0^\bl(\Var_{\DA^1}) \sira
  K_0(\Var_{\DA^1})$ of abelian groups from Theorem
  \ref{t:groth-group-var-over-A1-smooth-and-blowup}
  % .\ref{enum:blowup-presentation} 
  (and that one may restrict
  to connected varieties or projective morphisms
  in \ref{enum:blowup-presentation}).
  This shows uniqueness.

  If $X$ and $W$ are as above, then
  $\Perf(\bfMF(X,W)^\dg)$ is saturated by
  Theorem~\ref{t:combined-thomseb-smooth-proper}.
  Hence to show existence we only need to see that
  the relation $[\emptyset]_{\DA^1}=0$ and 
  the blowing-up relations 
  go to zero under  
  $[X,W] \mapsto \ul{\Perf(\bfMF(X,W)^\dg)}.$
  % for $X$ a
  % smooth variety and $W \colon  X \ra \DA^1$ proper.
  It is trivial that $[\emptyset]_{\DA^1}$ goes to   
  $\ul{\Perf(\emptyset)}=0.$
  It is enough to consider the blowing-up relations when 
  blowing-up a connected smooth subvariety, and in this case we
  can use
  \cite[Cor.~\ref{semi:c:semi-orthog-1-mf-dg-level} and
  \ref{semi:c:semi-orthog-2-mf-dg-level}]{valery-olaf-matfak-semi-orth-decomp}
  and Proposition~\ref{p:semi-od-goes-to-Perf}.
\end{proof}

Let us formulate the main result of this article.

\begin{theorem}
  \label{t:category-of-singularities-induces-ring-morphism}
  Let $k$ be an algebraically closed field of 
  characteristic zero.
  There is a unique morphism
  \begin{equation*}
    % \label{eq:MF-mot-meas}
    \mu \colon  K_0(\Var_{\DA^1})\ra K_0(\sat_k^{\DZ_2})
  \end{equation*}
  of rings (= a Landau-Ginzburg motivic measure) that maps
  $[X,W]$ to 
  $\ul{\bfMF(W)^{\dg,\natural}}$ whenever $X$ is a smooth
  variety and
  $W \colon  X \ra \DA^1$ is a proper morphism. 

  In particular, $\mu$ is a morphism of abelian groups and maps 
  $[X,W]$ to 
  $\ul{\bfMF(W)^{\dg,\natural}}$ whenever $X$ is a smooth (connected)
  variety and
  $W \colon  X \ra \DA^1$ is a projective morphism. These two properties
  determine $\mu$ uniquely.
\end{theorem}

\begin{proof}
  If $\mathcal{A}$ and $\mathcal{B}$ are saturated dg categories,
  then $\mathcal{A} \times \mathcal{B}$ is saturated and
  $\ul{\mathcal{A} \times \mathcal{B}} =\ul{\mathcal{A}} +
  \ul{\mathcal{B}}$ in $K_0(\sat_k^{\DZ_2})$ since there are
  semi-orthogonal decompositions $[\mathcal{A} \times
  \mathcal{B}]=\langle [\mathcal{A}], [\mathcal{B}]\rangle =
  \langle [\mathcal{B}], [\mathcal{A}]\rangle.$
  If we use the isomorphism
  $K_0^\bl(\Var_{\DA^1}) \sira
  K_0(\Var_{\DA^1})$ of abelian groups 
  and proceed as in the proof of
  Proposition~\ref{p:morphism-of-additive-groups}
  (using the defining equation \eqref{eq:def-Perf-W-dg}) we see
  that there 
  is a unique morphism 
  $\mu \colon  K_0(\Var_{\DA^1})\ra K_0(\sat_k^{\DZ_2})$
  of abelian groups 
  mapping $[X, W]$ to
  $\ul{\bfMF(W)^{\dg,\natural}}$ whenever $X$ is smooth and $W$ is
  proper, and that it is uniquely determined by its values on
  $[X,W]$ for $X$ smooth (connected) and $W$ projective.
  It is clear that $\mu$ sends the unit $[\Spec k, 0]$ of 
  $K_0(\Var_{\DA^1})$ to the unit $\ul{\Perf(k)}$ of
  $K_0(\sat_k^{\DZ_2}).$ 

  We need to prove that $\mu$ is compatible with
  multiplication.  
  Recall that the multiplication is easy to define on 
  $K^0(\Var_{\DA^1})$ but not on
  $K_0^\bl(\Var_{\DA^1}).$
  % Hence we now use\fotnote{
  %   maybe brauche das gar nicht!
  % } 
  % the
  % isomorphism 
  % $K_0^\sm(\Var_{\DA^1}) \sira
  % K_0(\Var_{\DA^1})$ of rings from
  % Theorem
  % \ref{t:groth-group-var-over-A1-smooth-and-blowup}.\ref{enum:smooth-presentation}. 
  % We
  % can and will assume that $K_0^\sm(\Var_{\DA^1})$ is
  % defined using smooth varieties which are in addition
  % connected and quasi-projective over $\DA^1.$ Note that such a
  % variety is automatically quasi-projective.
  Let $X$ and $Y$ be smooth connected varieties
  with projective morphisms
  $W \colon X \ra \DA^1$ and $V \colon Y\ra \DA^1.$
  By definition of $\mu$ and
  Theorem~\ref{t:combined-thomseb-smooth-proper} 
  we have
  \begin{equation*}
    \mu([X,W]) \bullet \mu([Y,V]
    =
    \ul{\bfMF(W)^{\dg,\natural}} \bullet \ul{\bfMF(V)^{\dg,\natural}} 
    =
    \ul{\bfMF(W*V)^{\dg,\natural}}
  \end{equation*}
  in $K_0(\sat_k^{\DZ_2}).$
  If $W$ or $V$ is constant, then $W*V$ is projective and hence
  $\ul{\bfMF(W*V)^{\dg,\natural}}= \mu([X \times Y, W*V),$ so $\mu$
  is multiplicative.
 
  Hence we can assume that both $W$ and $V$ are flat.
  Since $W*V$ might not be projective, 
  it is not clear 
  that $\mu$ maps  
  $[X \times Y, W*V]$ to $\ul{\bfMF(W*V)^{\dg,\natural}}.$
  In order to prove this it is enough to find
  smooth varieties $Z_i$ together with
  projective 
  morphisms $W_i \colon  Z_i \ra \DA^1$ and integers $n_i$ such that
  \begin{align*}
    [X \times Y, W*V] & = \sum_i \;n_i\;[Z_i,W_i]
    && \text{in $K_0(\Var_{\DA^1}),$ and}\\
    \ul{\bfMF(W*V)^{\dg,\natural}} & = \sum
    \;n_i\;\ul{\bfMF(W_i)^{\dg,\natural}} 
    && \text{in $K_0(\sat_k^{\DZ_2}).$}
  \end{align*}
  This can be done using Proposition~\ref{p:compactification}
  below
  which shows that the morphism $W*V$ can be "compactified"
  in a nice way.
  Using the 
  % We use the
  notation introduced there, 
  % Furthermore we abbreviate
  % $D_{i_1 \dots i_p}:=D_{i_1} \cap \dots \cap D_{i_p}$ and
  % denote by 
  % $h_{i_1 \dots i_p}$ the morphism $D_{i_1 \dots i_p} \ra \DA^1$
  % induced by $h.$
  it is easy
  to see that
  \begin{equation*}
    [X\times Y,W*V]=[Z,h]-\sum_{i}[D_i,h_i]+\sum_{i<j}[D_{ij},h_{ij}]
    - \dots + (-1)^{s-1}[D_{12 \dots s}, h_{12 \dots s}]
  \end{equation*}
  in $K_0(\Var_{\DA^1}).$
  On the right-hand side, $Z$ and all $D_{i_1 \dots i_p}$ 
  are smooth quasi-projective varieties, and $h$ and all 
  $h_{i_1 \dots i_p}$ are projective morphisms, 
  by 
  % part
  % \ref{enum:compactification-intersections-proper-and-smooth}
  % of 
  Proposition~\ref{p:compactification}.\ref{enum:compactification-intersections-proper-and-smooth}.
  Hence we obtain
  \begin{equation*}
    \mu([X\times Y,W*V])=\ul{\bfMF(h)^{\dg,\natural}}-
    \sum_{i} \ul{\bfMF(h_i)^{\dg,\natural}}+ \dots
    + (-1)^{s-1} \ul{\bfMF(h_{12 \dots s})^{\dg,\natural}}.
  \end{equation*}
  Lemma~\ref{l:D-Sg-W-vanishes}
  and
  Proposition~\ref{p:compactification}.\ref{enum:compactification-intersections-proper-and-smooth}
  again show that 
  $\ul{\bfMF(h_{i_1 \dots i_p})^{\dg,\natural}}=0$
  for all tuples $(i_1,
  \dots, i_p)$ with $p \geq 1.$
  Hence it is enough to show that
  $\ul{\bfMF(h)^{\dg,\natural}}=\ul{\bfMF(W*V)^{\dg,\natural}}.$ 
  
  Let $j \colon  X \times Y \ra Z$ be the open inclusion, and let $a \in
  k.$ The functor $j^* \colon \bfMF(Z, h-a) \ra \bfMF(X \times Y,
  W*V-a)$ lifts to a dg functor $j^* \colon  \bfMF(Z, h-a)^\dg \ra
  \bfMF(X \times Y, W*V-a)^\dg$ if we work for example with the
  enhancements using injective quasi-coherent sheaves.  From the
  defining equation~\eqref{eq:def-Perf-W-dg} it is clearly enough
  to show that this functor is a quasi-equivalence, or
  equivalently, that $j^* \colon \bfMF(Z, h-a) \ra \bfMF(X \times Y,
  W*V-a)$ is an equivalence.  Note that $W*V$ and hence $h$ are
  flat, so we can use Orlov's equivalence
  \cite[Thm.~\ref{semi:t:factorizations=singularity}]{valery-olaf-matfak-semi-orth-decomp}
  and have to prove that $j^* \colon D_\Sg(Z_a) \ra D_\Sg((X \times
  Y)_a)$ is an equivalence.  But
  equation~\eqref{eq:reduced-SingW-cap-Xa-equals-reduced-Xa-sing}
  in Remark~\ref{rem:SingW-cap-Xa-versus-Xa-sing} and
  Proposition~\ref{p:compactification}.\ref{enum:compactification-critical-points}
  imply that
  \begin{equation*}
    (Z_a)^\sing=|\Sing(h) \cap Z_a|
    =
    |\Sing(W*V) \cap Z_a| 
    =((X \times Y)_a)^\sing
    \subset (X \times Y)_a,
  \end{equation*}
  so we can use
  \cite[Prop.~1.3]{orlov-equivalences-LG-models}.
\end{proof}

\begin{remark}
  \label{rem:lg-measure-onW*V}
  Let $X$ and $Y$ be smooth varieties with proper morphisms
  $W \colon  X \ra \DA^1$ and $V \colon  Y \ra \DA^1.$
  Then we see from
  Theorems~\ref{t:category-of-singularities-induces-ring-morphism}
  and \ref{t:combined-thomseb-smooth-proper}
  that
  \begin{equation*}
    \mu([X \times Y, W*V])= \mu([X,W]) \bullet \mu([Y,V]) =
    \ul{\bfMF(W)^{\dg,\natural})} \bullet \ul{\bfMF(V)^{\dg,\natural})} =
    \ul{\bfMF(W*V)^{\dg,\natural}}. 
  \end{equation*}
  This shows that the Landau-Ginzburg motivic measure $\mu$ 
  sends $W*V \colon  X \times Y \ra \DA^1$ to $\ul{\bfMF(W*V)^{\dg,\natural}}$ 
  even though $W*V$ might not be proper. This statement is
  slightly more general than what we showed in the proof of
  Theorem~\ref{t:category-of-singularities-induces-ring-morphism}.\end{remark}

\begin{remark}
  \label{rem:class-L-over-A1-minus-one-goes-to-zero}
  From \cite[Corollary~\ref{semi:c:semi-orthog-1-mf-dg-level}]{valery-olaf-matfak-semi-orth-decomp}
  we see that $\mu([\DP^n_k,0])= (n+1) \ul{\Perf(k)}=n+1.$
  Recall the element $\DL_{(\DA^1,0)}
  % \DL_{\DA^1}
  :=[\DA^1,0] \in K_0(\Var_\DA^1)$
  from Remark~\ref{rem:class-L-over-A1}.
  Then we obtain $\mu(\DL_{(\DA^1,0)})=
  \mu([\DP^1_k,0])-\mu([\Spec k,0])=2-1=1.$
  This implies that $\mu$ factors to a morphism
  \begin{equation*}
    \mu \colon  K_0(\Var_{\DA^1})/(\DL_{(\DA^1,0)}-1) \ra K_0(\sat_k^{\DZ_2})
  \end{equation*}
  of rings, cf.\
  \cite[sect.~8.2]{bondal-larsen-lunts-grothendieck-ring}. 
  
  If $W \colon  X \ra \DA^1$ is a proper and smooth morphism, then
  certainly $\bfMF(W)=0$ by Lemma~\ref{l:D-Sg-W-vanishes} and
  hence $\mu([X,W])=0.$ This yields many other elements of the
  kernel of $\mu.$ For example $[\DA^1, \id_{\DA^1}]$ lies in the
  kernel of $\mu.$
\end{remark}

\section{Compactification}
\label{sec:compactification}

Let $k$ be an algebraically closed field of characteristic zero.

\begin{proposition}
  \label{p:compactification}
  Let $X$ and $Y$ be smooth 
  varieties and let
  $W \colon  X \ra \DA^1$ and $V \colon Y\ra \DA^1$ be 
  projective
  morphisms (hence $X$ and $Y$ are
  quasi-projective varieties).
  Consider the convolution
  \begin{equation*}
    W*V \colon X\times Y\xra{W \times V} \DA^1\times \DA^1\xra{+} \DA^1.
  \end{equation*}
  Then there exists a smooth quasi-projective variety $Z$ with
  an open embedding $X\times Y \hra Z$ and a
  projective morphism $h \colon Z\ra \DA^1$ such that the
  following conditions are satisfied.
  \begin{enumerate}[label=(\roman*)]
  \item
    \label{enum:compactification-diagram}
    The diagram
    \begin{equation*}
      \xymatrix{
        {X\times Y} \iar[r] \ar[d]^-{W*V} & {Z} \ar[d]^-h\\
        {\DA^1} \ar@{}[r]|-{=} & {\DA^1}
      }
    \end{equation*}
    commutes.
  \item
    \label{enum:compactification-critical-points}
    All critical points of $h$ are contained
    in $X\times Y,$ i.\,e.\ $\Sing(W*V)=\Sing(h).$
  \item
    \label{enum:compactification-boundary-snc}
    We have $Z \setminus X \times Y =  \bigcup_{i=1}^s D_i$ where
    the
    $D_i$ are pairwise distinct smooth
    prime divisors.
    More precisely, $Z \setminus X \times Y$ is the support
    of a snc (= simple normal crossing) divisor.
  \item
    \label{enum:compactification-intersections-proper-and-smooth}
    For every $p$-tuple $(i_1, \dots, i_p)$ of indices
    (with $p \geq 1$) the morphism
    \begin{equation*}
      h_{i_1 \dots i_p} \colon  D_{i_1 \dots i_p}:=D_{i_1} \cap \dots \cap
      D_{i_p} \ra \DA^1   
    \end{equation*}
    induced by $h$ is projective and smooth. In
    particular, all
    $D_{i_1 \dots i_p}$ are smooth quasi-projective varieties.
  \end{enumerate}
\end{proposition}

\begin{remark}
  \label{rem:compactification-W-or-V-constant}
  To prove Proposition~\ref{p:compactification}
  one may assume that both $X$ and $Y$ are connected. 
  If the map $W$ is not flat then its image is one point
  $W(X)=a\in \DA^1,$ $X$ is projective,  
  % $X$ is projective, 
  and the map 
  $W*V \colon  X\times Y\ra \DA^1$ is already projective.
  % with fibers $(X\times Y)_b=X\times Y_{b-a}.$ 
  So we can take $Z=X\times Y$ and
  $h=W*V.$ 
  This shows that 
  Proposition~\ref{p:compactification}
  is interesting only in case both $W$ and $V$ are flat.
  The proof given below works in general.
\end{remark}

We need some preparations for the proof of this proposition. 
Let $U$ be a scheme 
% of finite type over $k$ 
and $I \subset \mathcal{O}_U$ an ideal
sheaf. We say that the pair
$(U,I)$ satisfies condition~\ref{enum:non-zero} if
\begin{enumerate}[label=(K)]
\item
  \label{enum:non-zero}
  $U$ is a reduced
  scheme of finite type (over $k$), $I$ is not
  zero on any irreducible component of $U,$ and the closed
  subscheme $\mathcal{V}(I)$ defined by $I$ contains the singular locus
  $U^\sing$ of $U.$
\end{enumerate}

\begin{remark}
  \label{rem:resolution-and-principalization}
  We recall some results on resolution of
  singularities and monomialization (principalization) from
  \cite{kollar-singularities}. 
  Assume that 
  $(U,I)$ as above satisfies condition \ref{enum:non-zero}.
  Let $\tildew{U} \ra U$ 
  be the
  resolution of singularities from
  \cite[Thm.~3.36]{kollar-singularities} (it seems preferable to
  start with a reduced scheme there). 
  Then $\tildew{U}$ together with the inverse image ideal sheaf
  $\tildew{I}$ of $I$ under $\tildew{U} \ra U$ also satisfies
  condition~\ref{enum:non-zero}: $\tildew{U}$ is again reduced
  (\cite[Lemma~8.1.2]{Liu}) of
  finite type, $\tildew{I}$ is not zero
  on any irreducible component of $\tildew{U}$ since $\tildew{U}
  \ra U$ is birational (as confirmed to us by J{\'a}nos
  Koll{\'a}r), 
  and $\tildew{U}^\sing = \emptyset.$
  So we can apply
  monomialization (principalization) 
  \cite[Thm.~3.35]{kollar-singularities} 
  to this inverse image ideal sheaf 
  % of $I$ under $\tildew{U} \ra U$ 
  (and the empty snc
  divisor)
  and obtain a morphism $c(U)=c_I(U) \ra \tildew{U}.$
  Let $\gamma=\gamma_U=\gamma_{U,I}$
  % $ \colon  c(U) \ra U$
  be the composition $c(U) \ra \tildew{U} \ra U.$ Then $c(U)$ is
  a smooth 
  scheme of 
  finite type over $k,$
  the inverse image ideal sheaf
  $\gamma\inv(I) \cdot \mathcal{O}_{c(U)} \subset
  \mathcal{O}_{c(U)}$ is the ideal sheaf of a
  snc divisor, and $\gamma$ is an isomorphism over $U \setminus
  \mathcal{V}(I).$ Moreover, $\gamma$ is a composition of blowing-up
  morphisms and in particular a proper morphism. If $U$ is
  quasi-projective (resp.\ projective), so is $c(U),$ and
  $\gamma$ is a projective morphism.
  As described in \cite[3.34.1]{kollar-singularities},
  the association
  \begin{equation}
    \label{eq:principalize-resolve}
    (U, I) \mapsto \big(c_I(U) \xra{\gamma} U\big)
  \end{equation}
  commutes with smooth (and in particular \'etale)
  morphisms.
  This means that any smooth or \'etale morphism
  $f \colon U' \ra U,$ gives rise to a pullback diagram
  \begin{equation}
    \label{eq:c-smooth-bc}
    \xymatrix{
      {c_{f\inv(I) \cdot \mathcal{O}_{U'}}(U')} \ar[r] \ar[d]^-{\gamma_{U'}} &
      {c_I(U)} \ar[d]^-{\gamma_{U}} \\
      {U'} \ar[r]^-{f} &
      {U.}
    }
  \end{equation}
\end{remark}

The following proposition provides useful
compactifications
% of $W \colon X \ra \DA^1$ and $V \colon  Y \ra \DA^1$ 
and describes them "\'etale locally".
We view $\DA^1 \subset \DP^1,$ $z \mapsto [1,z],$ 
as an open subvariety, and let $\infty=[0,1] \in \DP^1.$
We write $\DA^1_\infty$ for $\DA^1$ viewed as an open
neighborhood of $\infty$ via $z \mapsto [z,1].$ 

\begin{proposition}
  \label{p:snc-compactification-at-infinity}
  Let $X$ be a smooth (quasi-projective)
  variety and let $W \colon X \ra \DA^1$
  be a
  projective morphism.
  Let $I_{\infty}$ be the ideal
  sheaf of the closed subvariety
  $\{\infty\} \subset \DP^1.$
  % , where $\infty=[0,1],$
  % and view $\DA^1 \subset \DP^1,$ $z \mapsto [1,z],$ as an open
  % subvariety.
  Then there is a smooth projective variety $\ol{X}$ with an open
  embedding $X \hra \ol{X}$
  and a projective morphism
  $\ol{W} \colon \ol{X} \ra \DP^1$ such
  that the diagram
  \begin{equation}
    \label{eq:pullback-snc-compacti-at-infinity}
    \xymatrix{
      {X} \iar[r] \ar[d]^-{W} & {\ol{X}} \ar[d]^-{\ol{W}}\\
      {\DA^1} \iar[r] & {\DP^1}
    }
  \end{equation}
  is a pullback diagram
  and such that the inverse image ideal sheaf
  $\ol{W}\inv(I_\infty) \cdot \mathcal{O}_{\ol{X}}
  \subset \mathcal{O}_{\ol{X}}$
  is a locally monomial ideal, i.\,e.\ the ideal sheaf of a
  snc divisor.

  % Let $\ol{X}_\infty:=\ol{W}\inv(\infty)$ be the fiber at infinity
  % (with the reduced structure).

  In particular, for any (closed)
  point $p$ in the fiber
  $\ol{X}_\infty:=\ol{W}\inv(\infty)$ at infinity, there are
  an \'etale morphism
  $u \colon  U \ra \ol{X}$ with $p$ in its image,
  % , a point $q \in U$ with $u(q)=p,$
  uniformizing parameters $\ul{x}=(\ul{x}_1, \dots, \ul{x}_m)$ on $U$
  % (centered at $q$)
  and a tuple $\mu=(\mu_1, \dots, \mu_s)$ of positive integers,
  for some $1 \leq s \leq m,$ such that the diagram
  \begin{equation*}
    % \xymatrix{
    %   {U} \ar[d]^-{u} \ar[r]^-{\ul{x}} &
    %   {\DA^m} \ar[dd]^-{\mu}\\
    %   {\ol{X}} \ar[d]^-{\ol{W}} \\
    %   {\DP^1} \ar@{}[r]|-{\supset} &
    %   {\DA^1_\infty} \\
    % }
    \xymatrix{
      {\ol{X}} \ar[d]^-{\ol{W}} &
      {U} \ar[l]_-{u} \ar[r]^-{\ul{x}} &       
      {\DA^m} \ar[d]^-{\mu}\\
      {\DP^1} \ar@{}[rr]|-{\supset} &&
      {\DA^1_\infty} \\
    }
  \end{equation*}
  commutes, where $\ul{x}$ is the morphism given by the
  uniformizing parameters and $\mu$ is the morphism mapping
  $(t_1,\dots, t_m)$ to $t^\mu:=t_1^{\mu_1} \dots t_s^{\mu_s}.$
  % and $\DA^1_\infty$ is $\DA^1$ viewed as an open neighborhood of
  % $\infty$ in $\DP^1$ via $z \mapsto [z,1].$
\end{proposition}

\begin{proof}
  By assumption on $W$ we have a commutative diagram, for some $N
  \in \DN,$
  \begin{equation*}
    \xymatrix{
      {X} \iar[r] \ar[d]^-{W} & {\DP^N_{\DA^1}} \iar[r] \ar[d]
      & {\DP^N_{\DP^1}} \ar[d] 
      % \ar@{}[r]|-{=} & {\DP^N \times \DP^1}
      \\
      {\DA^1} \ar@{}[r]|-{=} & {\DA^1} \iar[r] & {\DP^1}
    }
  \end{equation*}
  where the first arrow in the first row is a closed
  embedding. Let $K$ be the closure of $X$ in
  $\DP^N_{\DP^1}=\DP^N \times \DP^1.$ 
  Then $K$ is a projective variety with an open embedding $X
  \subset K$ and 
  a projective morphism $\kappa \colon  K \ra \DP^1$
  such that the diagram
  \begin{equation*}
    \xymatrix{
      {X} \iar[r] \ar[d]^-{W} & {K} \ar[d]^-{\kappa}\\
      {\DA^1} \iar[r] & {\DP^1}
    }
  \end{equation*}
  is a pullback diagram.
  This compactifies $W \colon  X \ra \DA^1$ at infinity.
  The singular points of $K$ are all contained in the fiber
  of $\kappa$ over $\infty.$ So clearly 
  $(K, \kappa\inv(I_\infty)\cdot \mathcal{O}_K)$ satisfies
  condition~\ref{enum:non-zero},
  and we obtain a morphism
  $\gamma \colon  c_{\kappa\inv(I_\infty) \cdot \mathcal{O}_K}(K) \ra K$
  as explained in 
  Remark~\ref{rem:resolution-and-principalization}.
  Define $\ol{X}:= c_{\kappa\inv(I_\infty) \cdot _K}(K)$
  and $\ol{W}:= \kappa \comp \gamma.$ Then $\ol{X}$ is smooth
  projective and $\ol{W}$ is projective, and from the
  construction we obtain the
  pullback diagram~\eqref{eq:pullback-snc-compacti-at-infinity}.

  It remains to provide the local description of $\ol{W}$ around
  $p \in \ol{X}_\infty.$ We can assume that $p$ is a closed
  point.
  Let $\ol{X}':= \ol{W}\inv(\DA^1_\infty)$ and view the
  restriction $\ol{W} \colon  \ol{X}' \ra \DA^1_\infty$
  as a regular function on $\ol{X}'.$ 
  It generates the inverse image ideal sheaf
  $\ol{W}\inv(I_\infty) \cdot \mathcal{O}_{\ol{X}},$
  so its divisor is the snc divisor 
  of this ideal sheaf. 
  Hence there 
  is an open neighborhood $U'$ of $p$ in $\ol{X}$ with
  uniformizing parameters
  $(\ul{y}_1, \dots, \ul{y}_m)$ centered at $p$ and
  a tuple $\mu=(\mu_1, \dots, \mu_s)$ of positive integers,
  for some $1 \leq s \leq m,$ such that
  $\ol{W}= v \ul{y}_1^{\mu_1} \dots \ul{y}_s^{\mu_s}$
  for some unit $v$ in $\mathcal{O}_{\ol{X}}(U').$
  Let $u \colon  U \ra U'$ be the \'etale morphism extracting a $\mu_1$-th
  root of $v.$ Then
  $\ul{x}_1:= \ul{y}_1 v^{1/\mu_1}, \ul{x}_2:=\ul{y}_2,
  \dots, \ul{x}_m:=\ul{y}_m$
  defines uniformizing parameters
  on $U$ which satisfy $\ul{x}_1^{\mu_1} \dots \ul{x}_s^{\mu_s}=
  \ol{W} \comp u.$ 
\end{proof}

We introduce another condition needed in the proof of
Proposition~\ref{p:compactification}. 
Let $(U, I)$ satisfy condition~\ref{enum:non-zero}
and let
$\gamma \colon c(U) \ra U$ be as in
Remark~\ref{rem:resolution-and-principalization}.
Write the snc divisor corresponding to
$\gamma\inv(I) \cdot \mathcal{O}_{c(U)}$ as 
$\sum_{i=1}^s n_i E_i$ with pairwise distinct prime divisors
$E_1, \dots, E_s$ and all $n_i>0.$ Let $f \colon U \ra \DA^1$ be a
regular function.  We say that the triple $(U,I,f)$ satisfies
condition~\ref{enum:NoCrit-Sm} if
\begin{enumerate}[label=(NoCrit-Sm)]
\item
  \label{enum:NoCrit-Sm}
  No critical point of the morphism $f \comp \gamma \colon  c_I(U) \ra
  \DA^1$ is contained in $E_1 \cup \dots \cup E_s,$ and for every
  tuple $(i_1, \dots, i_p)$ of indices (with $p \geq 1$) the
  morphism $E_{i_1} \cap \dots \cap E_{i_p} \ra \DA^1$ induced by
  $f \comp \gamma$ is smooth.  
\end{enumerate}

\begin{proof}[Proof of Proposition~\ref{p:compactification}]
  % We start with the description of a morphism that will be essential
  % for our proof.
  Consider the morphism
  \begin{align*}
%    \label{eq:def-sigma}
    \sigma \colon  \DP^1 \times \DA^1 & \ra \DP^1 \times \DP^1,\\
%    \notag
    ([z'_0, z'_1], v') & \mapsto ([z'_0, z'_1], [z'_0, z'_0v'-z'_1]).
  \end{align*}
  % We embed $\DA^1 \subset \DP^1$ by $z \mapsto [1,z]$ and write
  % $\infty:=[0,1].$
  % Then 
  The image of $\sigma$ is
  $\DA^1 \times \DA^1 \cup \{(\infty, \infty)\}.$
  The fiber of $\sigma$ over $(\infty, \infty)$
  is $E:= \{\infty\} \times \DA^1,$ and $\sigma$ induces an isomorphism
  $\DA^1 \times \DA^1 \sira \DA^1 \times \DA^1.$ Note moreover that
  the diagram
  % \begin{equation}
  %   \label{eq:addition-becomes-projection}
  %   \xymatrix{
  %     {\DP^1 \times \DA^1} \ar@{}[r]|-{\supset} \ar[d]^-{\pr_2} &
  %     {\DA^1 \times \DA^1} \ar[r]^-{\sigma}_-{\sim} \ar[d]^-{\pr_2} &
  %     {\DA^1 \times \DA^1} \ar[d]^-{+} \ar@{}[r]|-{\subset} &
  %     {\DP^1 \times \DP^1} \\
  %     {\DA^1} \ar@{}[r]|-{=} &
  %     {\DA^1} \ar@{}[r]|-{=} &
  %     {\DA^1}
  %   }
  % \end{equation}
  \begin{equation}
    \label{eq:addition-becomes-projection}
    \xymatrix{
      {\DP^1 \times \DP^1} 
      &
      {\DA^1 \times \DA^1} \ar[d]^-{+} \ar@{}[l]|-{\supset} 
      &
      {\DA^1 \times \DA^1} \ar[l]_-{\sigma}^-{\sim}
      \ar[d]^-{\pr_2} 
      &
      {\DP^1 \times \DA^1} \ar@{}[l]|-{\subset} \ar[d]^-{\pr_2} 
      \\
      &
      {\DA^1} \ar@{}[r]|-{=} &
      {\DA^1} \ar@{}[r]|-{=} &
      {\DA^1} 
   }
  \end{equation}
  commutes. It says that addition $\DA^1 \times \DA^1 \xra{+} \DA^1$
  corresponds under $\sigma$ to the second projection; this
  projection  can easily be extended to the projective morphism 
  on the right. This little construction already does the job in
  case $X=Y=\DA^1$ and $W=V=\id_{\DA^1}.$

  Now let $X,$ $Y$ be smooth (quasi-projective)
  varieties and let
  $W \colon X \ra \DA^1$ and $V \colon  Y \ra \DA^1$ be projective morphisms.
  We can and will assume that $X$ and $Y$ are irreducible.

  We choose $X \hra \ol{X} \xra{\ol{W}} \DP^1$
  and $Y \hra \ol{Y} \xra{\ol{V}} \DP^1$
  having the properties described in
  Proposition~\ref{p:snc-compactification-at-infinity}.
  Consider the pullback diagram
  \begin{equation}
    \label{eq:pullback-named-T}
    \xymatrix{
      {T} \ar[r]^-{\hatw{\sigma}}
      \ar[d]^-{\theta} & {\ol{X} \times \ol{Y}}
      \ar[d]^-{\ol{W} \times \ol{V}}\\
      {\DP^1 \times \DA^1} \ar[r]^-{\sigma} & {\DP^1 \times \DP^1.}
    }
  \end{equation}
  Note that the upper horizontal arrow
  $\hatw{\sigma}$
  in this diagram
  induces an isomorphism
  $\hatw{\sigma} \colon  T':=\theta\inv(\DA^1 \times \DA^1) \sira X \times Y.$
  From \eqref{eq:addition-becomes-projection} it is obvious that
  under this isomorphism the morphism
  $\pr_2 \comp \theta|_{T'} \colon T' \ra \DA^1$
  corresponds to $W * V \colon X \times Y \ra \DA^1.$

  We need to analyze $T$ "\'etale locally" around an arbitrary
  point of 
  $\theta\inv(E).$
  Let $I_E \subset  \mathcal{O}_{\DP^1 \times \DA^1}$ be the ideal
  sheaf of $E.$
  Our analysis will in particular show that
  the pair
  $(T, \theta\inv(I_E) \cdot \mathcal{O}_T)$
  satisfies condition~\ref{enum:non-zero}, so that 
  the morphism $\gamma \colon  Z:=c_{\theta\inv(I_E) \cdot
    \mathcal{O}_T}(T) \ra T,$ is available.
  We will then see that $Z$ together with the composition
  $h \colon  Z \xra{\gamma} T \xra{\theta}
  \DP^1 \times \DA^1 \xra{\pr_2} \DA^1$
  does the job.

  Define
  \begin{equation*}
    B:=\Spec k[z_0,v,(z_0 v -1)\inv]
  \end{equation*}
  and embed this as an open subvariety of $\DP^1 \times \DA^1$
  via $(z'_0, v') \mapsto ([z'_0,1],v').$ 
  So $B$ is contained
  in $\DA^1_\infty \times \DA^1.$ 
  We have
  $B= \sigma\inv(\DA^1_\infty \times
  \DA^1_\infty).$
  Note that $B$ contains
  $E=\{\infty\} \times \DA^1=\{z'_0=0\}$ 
  and that
  $\sigma$ induces the morphism
  % (cf.\ \eqref{eq:def-sigma})
  \begin{align}
    \label{eq:sigma-restricted}
    \sigma \colon  B & \ra \DA^1_\infty \times \DA^1_\infty,\\
    \notag
    (z'_0, v') & \mapsto ([z'_0, 1], [z'_0\;(z'_0v'-1)\inv, 1]).
  \end{align}

  Let $t \in \theta\inv(E)$
  be the closed
  point around which we will analyze $T$ "\'etale locally".
  Define $(x_\infty, y_\infty) := \hat{\sigma}(t) \in
  \ol{X}_\infty \times \ol{Y}_{\infty}.$
  We use the local description of $\ol{W}$
  % (in the \'etale topology)
  around $x_\infty \in \ol{X}_\infty$
  given by 
  Proposition~\ref{p:snc-compactification-at-infinity}.
  There is an \'etale morphism $u \colon U \ra \ol{X}$ whose image
  contains $x_\infty$
  such that $\ol{W} \comp u$ can be factorized as
  % \begin{equation*}
  $U \xra{\ul{x}} \DA^m \xra{\mu} \DA^1_\infty \subset \DP^1$
  % \end{equation*}
  for uniformizing parameters $\ul{x}$ and a suitable $\mu.$
  Similarly we describe $\ol{V}$ locally around
  $y_\infty \in \ol{Y}_\infty$ by an \'etale morphism
  $u' \colon  U' \ra \ol{Y}$ such that $\ol{V} \comp u'$ is given by
  % \begin{equation*}
  $U' \xra{\ul{y}} \DA^n \xra{\nu} \DA^1_\infty \subset \DP^1$
  % \end{equation*}
  for suitable $\ul{y}$ and $\nu.$
  Then $(\ol{W} \times \ol{V}) \comp (u \times u')$
  is equal to the composition
  \begin{equation*}
    U \times U' \xra{\ul{x} \times \ul{y}} \DA^m \times \DA^n
    \xra{\mu \times \nu} \DA^1_\infty \times \DA^1_\infty
    \subset \DP^1 \times \DP^1.
  \end{equation*}
  Consider the pullback diagram
  % Let $S$ be defined such that 
  % We define $S$ such that the the square in the following diagram
  \begin{equation*}
    \xymatrix{
      % {c(S)} \ar[d]^-{\gamma}\\
      {S} \ar[d]^-{\theta'} \ar[r] &
      {\DA^m \times \DA^n} \ar[d]^-{\mu \times \nu}\\
      {B} \ar[r]^-{\sigma} &
      {\DA^1_\infty \times \DA^1_\infty.}
    }
  \end{equation*}
  % is a pullback diagram.
  Note that the pullback of $T$ along the \'etale morphism
  $u \times u'$ coincides with the pullback of $S$ along the
  \'etale morphism $\ul{x} \times \ul{y}.$ Let us denote this
  pullback by $\hatw{S}.$ The \'etale morphism $\hatw{S} \ra T$
  contains $t$ in its image.

  From \eqref{eq:sigma-restricted}
  we see that
  $S$ can be described explicitly as
  \begin{equation*}
    S= \Spec k[v, x, y,
    (x^\mu v-1)\inv]/(x^\mu - (x^\mu v -1)y^\nu)
  \end{equation*}
  where
  $x=x_1, \dots, x_m$ and $y=y_1, \dots, y_n$ and $x^\mu=x_1^{\mu_1}
  \dots x_s^{\mu_s}$ and similarly for $y^\nu.$
  Note that $\theta'^*(z_0)=x^\mu$ and $\pr_2 \comp \theta' = v.$ 
  Let $S' \ra S$ be the (surjective) \'etale morphism
  extracting the
  $\nu_1$-th 
  root of the invertible element $(x^\mu v -1).$ Define new
  coordinates 
  $y_1':=y_1(x^\mu v -1)^{1/\nu_1}, y_2':=y_2, \dots, y_n':= y_n.$
  Then $S'$ is given by
  \begin{equation*}
    S'= \Spec k[v, x, y', (x^\mu v-1)^{-1/\nu_1}]/(x^\mu - y'^\nu).
  \end{equation*}
  There is an obvious \'etale morphism from
  $S'$ to the open subscheme
  $S'':=\Spec k[v, x, y', (x^\mu v-1)^{-1}]/(x^\mu - y'^\nu)$
  of
  \begin{equation*}
    S'''= \Spec k[v] \times L
  \end{equation*}
  where $L :=\Spec k[x, y']/(x^\mu - y'^\nu).$

  Up to now we have constructed a zig-zag of \'etale morphisms
  $T \la \hatw{S} \ra S \la S' \ra S'' \ra S'''.$
  Let $\hatw{T}$ be the pullback of $\hatw{S} \ra
  S$ and $S' \ra S.$ Hence we have \'etale morphisms
  \begin{equation}
    \label{eq:etale-zig-zag-from-T-to-S'''}
    T \xla{\alpha} \hatw{T} \xra{\beta} S'''
  \end{equation}
  and $t$ is in the image of $\alpha$ (since $S' \ra S$ is
  surjective).
  The ideal sheaves $\theta\inv(I_E) \cdot \mathcal{O}_T$ on $T$
  and $(x^\mu)$ on $S'''$ have the same inverse
  image ideal sheaf on $\hatw{T}$ (which also comes from the
  ideal sheaf $(x^\mu)$ on $S$). 
  % Let us call this ideal sheaf $J.$ 
  Correspondingly, we have $\alpha\inv(T')=
  \beta\inv(S''' \setminus \mathcal{V}(x^\mu))$ (recall that
  $\hat{\sigma} \colon  T' \sira X \times Y$).
  Note also that $\alpha^*(\pr_2 \comp \theta)= \beta^*(v)$ as
  functions $\hatw{T} \ra \DA^1.$
  
  Lemma~\ref{l:monom-x-minus-monom-y-reduced} tells us that
  $L$ is a reduced scheme of finite type, that the ideal
  $(x^\mu)$ does not vanish on any irreducible component of $L$
  and that $\mathcal{V}(x^m) \supset L^\sing$ (the singular locus
  of each component is contained in $\mathcal{V}(x^\mu),$ and
  different components do not intersect outside
  $\mathcal{V}(x^\mu)$).
  This just means that $(L, (x^\mu))$ satisfies
  condition~\ref{enum:non-zero}. Hence the same is true for
  $(S''', (x^\mu)).$

  From 
  \cite[Prop.~I.4.9]{knutson-algebraic-spaces}
  we see that $\hatw{T}$ and $\alpha(\hatw{T})$ are reduced
  schemes. This implies that $\hatw{T}$ and $T$ (let $t$ vary) are
  quasi-projective varieties. 

  We claim that $T$ is irreducible. Obviously, $L
  \setminus \mathcal{V}(x^\mu)$ is open and 
  dense in $L.$ 
  By Lemma~\ref{l:open-morphism-inverse-image-dense} below,
  applied to $\hatw{T} \xra{\beta} S''' \ra L,$ we see that  
  $\beta\inv(S''' \setminus \mathcal{V}(x^\mu))$ is open and dense in
  $\hatw{T}.$ Recall that
  $\alpha\inv(T')=
  \beta\inv(S''' \setminus \mathcal{V}(x^\mu)).$ We obtain that
  $\alpha(\alpha\inv(T'))$ is open and dense in
  $\alpha(\hatw{T}).$ In particular, $t$ is in the closure of
  $T'$ in $T.$ Since $t \in \theta\inv(E)=T \setminus T'$ was
  arbitrary and $T' \sira X \times Y$ is irreducible this proves
  that $T$ is irreducible.

  Now it is clear that
  $\theta\inv(I_E) \cdot \mathcal{O}_T$ does not vanish
  $T,$ and certainly we have $T^\sing \subset \theta\inv(E)=
  \mathcal{V}(\theta\inv(I_E)\cdot \mathcal{O}_T).$
  This proves that 
  $(T, \theta\inv(I_E) \cdot \mathcal{O}_T)$
  satisfies condition~\ref{enum:non-zero}.

  Hence we can apply
  Remark~\ref{rem:resolution-and-principalization} 
  to $(T, \theta\inv(I_E) \cdot \mathcal{O}_T)$ and obtain a
  morphism $\gamma \colon  c_{\theta\inv(I_E) \cdot \mathcal{O}_T}(T)
  \ra T.$ Define
  $Z:=c_{\theta\inv(I_E) \cdot \mathcal{O}_T}(T)$
  and 
  \begin{equation*}
    h \colon  Z \xra{\gamma} T \xra{\theta}
    \DP^1 \times \DA^1 \xra{\pr_2} \DA^1.
  \end{equation*}
  Then $h \colon Z \ra \DA^1$ is a projective morphism and $Z$ is a
  smooth quasi-projective variety.
  By construction
  $\gamma$ induces an isomorphism $\gamma\inv(T') \sira T'.$
  Using the isomorphism
  $\hatw{\sigma} \colon  T' \sira X \times Y$ we hence find an
  open embedding $X \times Y \hra Z.$ We claim that this datum
  satisfies conditions
  \ref{enum:compactification-diagram}
  -
  \ref{enum:compactification-intersections-proper-and-smooth}.
  % of Lemma~\ref{l:compactification}
  Indeed,
  \ref{enum:compactification-diagram}
  and \ref{enum:compactification-boundary-snc}
  hold by construction.
  All 
  $D_{i_1 \dots i_p}:=D_{i_1} \cap \dots \cap D_{i_p}$ are smooth
  quasi-projective varieties
  since the $D_i$ are the irreducible components of the support
  of a snc divisor, and all
  morphisms $h_{i_1 \dots i_p} \colon  D_{i_1 \dots i_p}:=D_{i_1} \cap \dots \cap
  D_{i_p} \ra \DA^1$ induced by $h$    
  are
  projective since
  $h \colon  Z \ra \DA^1$ is projective.
  % and the
  % inclusions
  % $D_{i_1} \cap \dots D_{i_p} \subset Z$ are projective.

  Hence we need to show 
  condition~\ref{enum:compactification-critical-points}
  and the smoothness part of condition
  \ref{enum:compactification-intersections-proper-and-smooth},
  or, equivalently, that
  condition~\ref{enum:NoCrit-Sm}
  holds for the triple
  $(T, \theta\inv(I_E) \cdot \mathcal{O}_T, \pr_2 \comp \theta).$
  
  From Remark~\ref{rem:resolution-and-principalization} we see
  that this can be checked locally on $T,$ and even on an
  \'etale covering of $T$ (we use that the
  map~\eqref{eq:principalize-resolve} commutes with \'etale
  morphisms). 
  The zig-zag \eqref{eq:etale-zig-zag-from-T-to-S'''}
  of \'etale maps and the fact that we already know that $(S''',
  (x^\mu))$ satisfies condition~\ref{enum:non-zero} shows that it
  is enough to show 
  condition~\ref{enum:NoCrit-Sm}
  for
  $(S''', (x^\mu), v).$

  Consider $L$ with the ideal sheaf $(x^\mu)$ and the structure
  morphism $\can \colon  L \ra \Spec k.$
  Let $\gamma_L \colon  c(L) = c_{(x^\mu)}(L) \ra L$ be the morphism from
  Remark~\ref{rem:resolution-and-principalization}.
  Let $\sum_{j=1}^r m_j F_j$ (with pairwise distinct $F_j$'s and
  all $m_j>0$) be the snc divisor corresponding to
  $(\gamma_L^*(x^\mu)).$ Then 
  for every
  tuple $(j_1, \dots, j_q)$ of indices (with $q \geq 1$) the
  intersection $F_{j_1} \cap \dots \cap F_{j_q}$ is regular,
  i.\,e.\ $\can \comp \gamma_L \colon  F_{j_1} \cap \dots \cap F_{j_q} \ra
  \Spec k$ is smooth.

  Note that $(S''', (x^\mu), v)$ is obtained from $(L, (x^\mu),
  \can)$  via base change along the smooth
  morphism $\Spec k[v] \ra \Spec k.$ 
  So the two squares in the diagram 
  \begin{equation*}
    \xymatrix{
      {c_{(x^m)}(S''')} \ar[d]^-{\gamma_{S'''}} \ar[r] &
      {c_{(x^\mu)}(L)} \ar[d]^-{\gamma_L}\\
      {S'''} \ar[r] \ar[d]^-{v} & {L} \ar[d]^-{\can}\\
      {\DA^1=\Spec k[v]} \ar[r] & {\Spec k.}
    }
  \end{equation*}
  are pullback diagrams (we use that the
  map~\eqref{eq:principalize-resolve} commutes with smooth
  morphisms, cf.\ diagram \eqref{eq:c-smooth-bc}).
  Let $F'_j = \Spec k[v] \times F_j.$
  Then $\sum_{j=1}^r m_j F'_j$ 
  is the snc divisor corresponding to
  $(\gamma_{S'''}^*(x^\mu)),$ and it is 
  obvious that condition~\ref{enum:NoCrit-Sm}
  holds for
  $(S''', (x^\mu), v)$: the morphism $v \comp \gamma_{S'''}$ is
  smooth (since $\can \comp \gamma_L$ is smooth) and hence has
  no critical point at all, and all morphisms 
  $v \comp \gamma_L \colon  F'_{j_1} \cap \dots \cap F'_{j_q} \ra
  \Spec k[v]$ are smooth 
  since they are obtained from $\can \comp \gamma_L \colon  F_{j_1} \cap
  \dots \cap F_{j_q} \ra \Spec k$ by base change.
\end{proof}

\begin{lemma}
  \label{l:monom-x-minus-monom-y-reduced}
  Let 
  \begin{equation*}
    p:= x^\mu-y^\nu := x_1^{\mu_1} x_2^{\mu_2} \cdots x_s^{\mu_s} -
    y_1^{\nu_1} \cdots y_t^{\nu_t} 
  \end{equation*}
  be a polynomial in $k[x,y]:= k[x_1,\dots, x_s, y_1, \dots,
  y_t]$ with 
  $s,t>0$ and all $\mu_i >0$ and all $\nu_j>0.$
  Let $d=\gcd(\mu_1, \mu_2,
  \dots, \mu_s, \nu_1, \dots, \nu_t).$ Then
  \begin{equation}
    \label{eq:factorization-p}
    p=\prod_{\zeta \in \sqrt[d]{1}} (x^{\mu/d} - \zeta y^{\nu/d})
  \end{equation}
  is the factorization of $p$ into irreducibles in $k[x,y]$
  (where $\sqrt[d]{1}$ denotes the set of all $d$-th roots of
  unity in $k$);
  obviously, all
  factors are 
  distinct and appear with multiplicity one.
  Here we use the shorthand notation $x^{\mu/d}=x_1^{\mu_1/d} \cdots
  x_s^{\mu_s/d},$ and similarly for $y^{\nu/d}.$
  In particular,
  $p$ is irreducible in $k[x,y]$ if $d=1.$
  (If $s, t \leq n$ then the above factorization into
  irreducibles obviously is also a factorization into
  irreducibles in $k[x_1,\dots, x_n, y_1, \dots, y_n].$)
\end{lemma}

The proof of this lemma was motivated by a proof of its special
case $s=t=1$ on Stackexchange by Qiaochu Yuan.
We thank Jan B\"uthe for a discussion of the general case.

\begin{proof}
  From $T^d-1=\prod_{\zeta \in \sqrt[d]{1}} (T-\zeta)$
  we obtain by substituting $T=\frac UV$ that 
  $U^d-V^d=\prod_{\zeta \in \sqrt[d]{1}} (U-\zeta V).$
  From $p=(x^{\mu/d})^d-(y^{\nu/d})^d$ we hence obtain
  formula~\eqref{eq:factorization-p},
  and it is enough to show that each factor
  $(x^{\mu/d} - \zeta y^{\nu/d})$ is irreducible in $k[x,y].$

  For this it is enough to show that $p$ is irreducible if $d=1$
  (since then also any polynomial 
  $x^\mu-\lambda y^\nu$ will be irreducible for $\lambda \in
  k^\times$: put $y'_1:=\sqrt[\nu_1]{\lambda} \;y_1;$
  alternatively, adapt the
  following proof so that it works directly for
  $x^\mu-\lambda y^\nu$).
 
  % Assume $d=1.$
  Let $f$ be an irreducible factor of $p$ in $k[x,y].$ 
  The group $\DZ_{\mu_1}$ acts on $k[x,y]$ by algebra
  automorphisms such that the generator $1$ of $\DZ_{\mu_1}$ maps
  $x_1$ to $\zeta_{\mu_1}x_1$ where $\zeta_{\mu_1}$ is a
  fixed primitive $\mu_1$-th root of unity. 
  % (More canonically 
  % the group of $\mu_1$-th roots of unity acts.)
  By combining the analog commuting actions on the other
  variables we obtain an action of
  $Z:= \DZ_{\mu} \times \DZ_\nu:= \DZ_{\mu_1} \times
  \dots \times 
  \DZ_{\mu_s} \times \DZ_{\nu_1} \times \dots \times \DZ_{\nu_t}$ 
  on $k[x,y].$
  
  Note that $p \in k[x,y]^Z=k[x_1^{\mu_1}, \dots, y_t^{\nu_t}].$
  Any element of the $Z$-orbit $Z.f$ of $f$ also is an
  irreducible factor of $p.$ 
  Some of these irreducible factors might be associated.
  % (= they differ by a factor in $k^\times$).
  Let $F$ be the
  product of all these irreducible factors up to
  $k^\times$-multiples. 
  (More precisely we mean the following: the group $Z$ acts on
  $\DP(k[x,y]),$ and the multiplication of $k[x,y]$ induces a
  multiplication on $\DP(k[x,y])$ which is compatible with the
  $Z$-action. Let $F \in k[x,y]$ be an element such that $[F] =
  \prod_{g \in Z.[f]} g$ in $\DP(k[x,y]).$)
  Then $F | p.$

  It is clear that $z.F \in
  k^\times F$ for all $z \in Z.$ We claim that in fact $F \in
  k[x,y]^Z.$ 

  Let $\rho \colon  Z \ra k^\times$ be the morphism of groups such that
  $z.F= \rho(z)F$ for all $z \in Z.$ 
  If we apply the element $z_1:=(1,0, 0, \dots, 0) \in Z$
  to the monomial $x^\alpha y^\beta$ we obtain 
  $\zeta_{\mu_1}^{\alpha_1} x^\alpha y^\beta.$
  If this monomial $x^\alpha y^\beta$ appears with non-zero
  coefficient in $F$ we must have
  $\rho(z_1)=\zeta_{\mu_1}^{\alpha_1}.$
  Hence if another monomial 
  $x^{\alpha'} y^{\beta'}$ also appears with non-zero
  coefficient in $F,$ then
  $\zeta_{\mu_1}^{\alpha_1}=\zeta_{\mu_1}^{\alpha'_1},$ or
  equivalently, $\alpha_1-\alpha'_1 \in \DZ \mu_1.$
  This implies that we can write $F=x_1^{\gamma} G$ with
  $G \in k[x_1^{\mu_1}, x_2, \dots, x_n, y_1, \dots, y_n],$ for
  some $\gamma \in \DN$ (for example the smallest exponent of
  $x_1$ that appears in a monomial that appears in $F$ with
  non-zero coefficient). Since $F|p$ this implies that
  $x_1^{\gamma} | p$ which is obviously only possible if
  $\gamma=0.$ 
  We can iterate this argument and eventually see that $F \in
  k[x_1^{\mu_1}, x_2^{\mu_2}, \dots, x_s^{\mu_s}, y_1^{\nu_1}, \dots,
  y_t^{\nu_t}]=k[x,y]^Z,$ proving our claim.

  Hence we have $F|p$ in $k[x,y]^Z.$
  Write $a_1:= x_1^{\mu_1}, \dots, a_s:=x_s^{\mu_s},$ and
  $b_1:= y_1^{\nu_1}, \dots, b_t:= y_t^{\nu_t}.$
  Then $p=a-b:=a_1 \dots a_s - b_1 \dots b_t$ and this element
  is irreducible in $k[x,y]^Z=k[a,b]$ (it is linear in $a_1$ and
  the coefficient $a_2 \dots a_s$ of $a_1$ and the constant
  coefficient $b_1 \dots b_t$ have greatest common divisor $1$).
  Since $F$ is not a unit this implies that $F=p$ up to a
  multiple in $k^\times.$

  Denote by $\deg_{x_i}(g)$ the degree
  of an element $g \in k[x,y]$ in $x_i.$
  Let $l$ be the cardinality of the orbit of $[f]$ in
  $\DP(k[x,y]),$ i.\,e.\ $F$ is the product of $l$ irreducible
  elements obtained from $f.$ Then 
  \begin{equation*}
    \mu_i=\deg_{x_i}(p)=\deg_{x_i}(F)=l \deg_{x_i}(f).
  \end{equation*}
  This and the same argument for the degrees in the $y_j$'s show
  that 
  $l$ is a common divisor of all the $\mu_i$ and $\nu_j.$

  If $d=1$ 
  we obtain $l=1,$ i.\,e.\ $F=f$ up
  to a multiple in $k^\times.$ Hence $F$ and $p$ are irreducible in
  $k[x,y].$ 
\end{proof}

\begin{lemma}
  \label{l:open-morphism-inverse-image-dense}
  Let $f \colon  X \ra Y$ be an open morphism of Noetherian schemes.
  If $V \subset Y$ is open and dense, then $f\inv(V)$ is open and
  dense in $X.$
\end{lemma}

\begin{proof}
  Let $C$ be an irreducible component of $X.$ Let $C^\circ$ be
  obtained from $C$ by removing all points that lie in an
  irreducible component distinct from $C.$ Then $C^\circ$ is open
  in $X$ and non-empty, so $f(C^\circ)$ is open and non-empty
  and hence contains a point of $V.$
  Then $f\inv(V) \cap C^\circ$ is open in $C$ and non-empty,
  and hence dense in $C.$ This implies that $C \subset
  \ol{f\inv(V)}.$  
\end{proof}

\def\cprime{$'$} \def\cprime{$'$} \def\cprime{$'$} \def\cprime{$'$}
  \def\Dbar{\leavevmode\lower.6ex\hbox to 0pt{\hskip-.23ex \accent"16\hss}D}
  \def\cprime{$'$} \def\cprime{$'$}
\providecommand{\bysame}{\leavevmode\hbox to3em{\hrulefill}\thinspace}
\providecommand{\MR}{\relax\ifhmode\unskip\space\fi MR }
\providecommand{\MRhref}[2]{%
  \href{http://www.ams.org/mathscinet-getitem?mr=#1}{#2}
}
\providecommand{\href}[2]{#2}

% \bibliographystyle{amsalpha}
% \bibliography{literatur}

\end{document}